\documentclass{amsart}
\pdfoutput=1 
\usepackage{amssymb, latexsym, amsmath, verbatim, amsthm, amscd}
\usepackage{dsfont}
\usepackage{color}
\usepackage{mathtools}
\usepackage{xspace}
\usepackage{pinlabel}
\usepackage{graphicx}
\usepackage{hyperref}

\title[A Metrizable Contracting Boundary]{A Metrizable Topology on the
  Contracting Boundary of a Group}
\author{Christopher H. Cashen}
\address{Faculty of Mathematics\\University of
  Vienna\\Oskar-Morgenstern-Platz 1\\1090 Vienna, Austria}
\author{John M. Mackay}
\address{School of Mathematics\\University of Bristol\\Bristol, UK}

\hypersetup{
    pdftitle={A Metrizable Topology on the
  Contracting Boundary of a Group},    
    pdfauthor={Christopher H. Cashen and John M. Mackay},     
    pdfkeywords={contracting boundary, Morse boundary, boundary at infinity,
  contracting geodesic, divagation}, 
    colorlinks=true,       
    linkcolor=black,          
    citecolor=black,        
    filecolor=black,      
    urlcolor=black           
}

\thanks{The first author thanks the Isaac Newton Institute for Mathematical
  Sciences for support and hospitality during the program \textit{Non-positive curvature: group actions and cohomology} where work
  on this paper was undertaken. This work is partially supported by
  EPSRC Grant Number EP/K032208/1 and by the Austrian Science Fund (FWF): P30487-N35.\\  
  The second author was supported in part by the National Science Foundation under Grant DMS-1440140 while visiting the Mathematical Sciences Research Institute in Berkeley, California, during the Fall 2016 semester, and in part by EPSRC grant EP/P010245/1.}
\subjclass[2010]{20F65, 20F67}
\keywords{contracting boundary, Morse boundary, boundary at infinity,
  contracting geodesic, divagation}

\theoremstyle{plain}
\newtheorem{theorem}{Theorem}[section]
\newtheorem{lemma}{Lemma}[section]
\newtheorem{proposition}{Proposition}[section]
\newtheorem{corollary}{Corollary}[section]

\theoremstyle{remark}

\newtheorem*{remark}{Remark}

\newtheorem{observation}{Observation}[section]
\theoremstyle{definition}
\newtheorem{definition}{Definition}[section]
\newtheorem{example}{Example}[section]

\def\makeautorefname#1#2{\expandafter\def\csname#1autorefname\endcsname{#2}}
\let\fullref\autoref

\makeautorefname{theorem}{Theorem} 
\makeautorefname{lemma}{Lemma} 
\makeautorefname{proposition}{Proposition} 
\makeautorefname{corollary}{Corollary} 
\makeautorefname{definition}{Definition}
\makeautorefname{example}{Example}
\makeautorefname{section}{Section}
\makeautorefname{subsection}{Section}
\makeautorefname{subsubsection}{Section}
\makeautorefname{claim}{Claim}
\makeautorefname{observation}{Observation}

\makeatletter 
\let\c@lemma=\c@theorem 
\makeatother
\makeatletter 
\let\c@proposition=\c@theorem 
\makeatother
\makeatletter 
\let\c@corollary=\c@theorem 
\makeatother
\makeatletter 
\let\c@definition=\c@theorem 
\makeatother
\makeatletter 
\let\c@example=\c@theorem 
\makeatother
\makeatletter 
\let\c@question=\c@theorem 
\makeatother
\makeatletter 
\let\c@observation=\c@theorem 
\makeatother
\makeatletter
\@addtoreset{claim}{proposition}
\makeatother
\makeatletter
\@addtoreset{claim}{lemma}
\makeatother

\mathtoolsset{centercolon} 

\DeclareMathOperator{\diam}{diam}
\DeclareMathOperator{\invl}{invl}

\DeclareMathOperator{\R}{\mathbb{R}} 

\newcommand{\bdry}{\partial} 
\newcommand{\cbdry}{\bdry_c}
\newcommand{\closure}[1]{\overline{#1}} 
\newcommand{\act}{\curvearrowright} 
\newcommand{\from}{\colon\thinspace} 
\newcommand{\into}{\hookrightarrow} 
\newcommand{\nbhd}{N}
\newcommand{\bp}{o}
\newcommand{\thistopology}{topology of fellow-travelling quasi-geodesics\xspace}

\newcommand{\fq}{\mathcal{FQ}}
\newcommand{\fg}{\mathcal{FG}}
\newcommand{\dl}{\mathcal{DL}}

\newcommand{\casen}[2]{\vspace{1pt} \noindent \emph{Case #1:}{ #2.} }

\newcommand{\one}{\mathds{1}}

\newcommand{\cL}{L}
\newcommand{\cA}{A}

\newcommand{\theconstantformerlyknownasLambda}{M}
\newcommand{\limit}{\Lambda}

\begin{document}
\begin{abstract}
The `contracting boundary' of a proper geodesic metric space consists
of equivalence classes of geodesic rays that behave like rays in a
hyperbolic space.
We introduce a geometrically relevant, quasi-isometry invariant topology on the contracting
boundary.
When the space is the Cayley graph of a finitely generated group we
show that our new topology is metrizable.

\end{abstract}
\maketitle

\section{Introduction}\label{sec:intro}
There is a long history in geometry of attaching a `boundary
at infinity' or `ideal boundary' to a space. 
When a group acts geometrically on a space we might wonder to what
extent the group and the boundary of the space are related. 
In the setting of (Gromov) hyperbolic groups this relationship is very
strong: the boundary is determined by the group, up to homeomorphism.
In particular, the boundaries of all Cayley graphs of a hyperbolic
group are homeomorphic, so it makes sense to call any one of these
boundaries the boundary \emph{of the group}. 
This is not true, for example, in the case of a group acting
geometrically on a non-positively curved space: Croke and Kleiner
\cite{CroKle00} gave an example of a group acting geometrically on two
different CAT(0) spaces with non-homeomorphic visual boundaries, so
there is not a well-defined visual boundary associated to the group.

Charney and Sultan \cite{ChaSul15} sought to rectify this problem by
defining a `contracting boundary' for CAT(0) spaces. 
Hyperbolic boundaries and visual boundaries of CAT(0) spaces can be
constructed as equivalence classes of geodesic rays emanating from a
fixed basepoint.
These represent the metrically distinct ways of `going to infinity'.
Charney and Sultan's idea was to restrict attention to ways of going
to infinity in hyperbolic directions:
They consider equivalence classes of geodesic rays
that are `contracting', which is a way of quantifying how hyperbolic
such rays are.
They topologize the resulting set using a direct limit construction,
and show that this topology is preserved by quasi-isometries.
However, their construction has drawbacks:
basically, it has too many open sets.
In general it is not first countable.

In this paper we define a bordification of a proper geodesic metric
space by adding a contracting boundary with a quasi-isometry invariant topology.
When the space is a Cayley graph of a finitely generated group, we
prove that the topology on the boundary is
metrizable, which is a significant improvement over the direct limit
topology.
(See \fullref{motivation} for a motivating example.)
Furthermore, our topology  more closely resembles the topology of the boundary of a hyperbolic space,
which we hope will make it easier to work with.

Our contracting boundary consists of equivalence classes of
`contracting quasi-geodesics'.
The definition of contraction we use follows that of Arzhantseva,
Cashen, Gruber, and Hume~\cite{ArzCasGrub}; this is weaker than that of Charney
and Sultan, so our construction applies to more general spaces.
For example, we get contracting quasi-geodesics from cyclic subgroups
generated by non-peripheral elements of relatively
hyperbolic groups~\cite{DruSap05},
pseudo-Anosov elements of mapping class groups~\cite{Min96,DucRaf09},
fully irreducible free group automorphisms~\cite{Alg11}, and
generalized loxodromic
elements of acylindrically hyperbolic groups~\cite{Osi16,DahGuiOsi17,BesBroFuj15,Sis16}.
On CAT(0) spaces the two definitions agree, so our boundary is the
same as theirs \emph{as a set}, but our topology is coarser.

Cordes~\cite{Cor15} has defined a `Morse boundary' for
proper geodesic metric spaces by applying Charney and Sultan's direct
limit construction to the set of equivalence classes of Morse geodesic rays. 
This boundary has been further studied by Cordes and
Hume~\cite{CorHum16-stable-morse}, who relate it to the notion of
`stable subgroups' introduced by Durham and Taylor~\cite{DurTay15};
for a recent survey of these developments\footnote{In even more recent
  developments,
  Behrstock \cite{Beh17} produces interesting examples of right-angled
  Coxeter
  groups whose Morse boundaries contain a circle, and Charney and
  Murray \cite{ChaMur17} give conditions that guarantee that a
  homeomorphism between Morse boundaries of CAT(0) spaces is induced
  by a quasi-isometry.}, see Cordes~\cite{Cordes17-survey}.
It turns out that our notion of contracting geodesic is equivalent to
the Morse condition, and our contracting boundary agrees with the
Morse boundary as a set, but, again, our topology is coarser.

If the underlying space is hyperbolic then all of these boundaries
are homeomorphic to the Gromov boundary.
At the other extreme, all of these boundaries are empty in spaces with
no hyperbolic directions.
In particular, it follows from work of Drutu and  Sapir \cite{DruSap05}
that groups that are \emph{wide}, that is, no asymptotic cone
contains a cut point, will have empty contracting boundary.
This includes groups satisfying a law: for instance, solvable groups or
bounded torsion groups.

\medskip

The boundary of a proper hyperbolic space can be topologized as
follows. 
If $\zeta$ is a point in the boundary, an equivalence class of
geodesic rays issuing from the chosen basepoint, we declare a small
neighborhood of $\zeta$ to consist of boundary points $\eta$ such that
if $\alpha\in\zeta$ and $\beta\in\eta$ are representative geodesic
rays then $\beta$ closely fellow-travels $\alpha$ for a long time.
In proving that this topology is invariant under quasi-isometries,
hyperbolicity is used at two key points. 
The first is that quasi-isometries take geodesic rays uniformly close
to geodesic rays.
In general a quasi-isometry only takes a geodesic ray to a quasi-geodesic
ray, but hyperbolicity implies that this is within bounded distance of
a geodesic ray, with bound depending only on the quasi-isometry and
hyperbolicity constants. 
The second use of hyperbolicity is to draw a clear distinction between
fellow-travelling and not, which is used to show that the time for
which two geodesics fellow-travel is roughly preserved by quasi-isometries.
If $\alpha$ and $\beta$ are non-asymptotic geodesic rays issuing from a common
basepoint in a hyperbolic space, then closest point projection sends
$\beta$ to a bounded subset $\alpha([0,T_0])$ of $\alpha$, and there
is a transition in the behavior of $\beta$ at time $T_0$.
For $t<T_0$ the distance from $\beta(t)$ to $\alpha$ is bounded and
the diameter of the projection of $\beta([0,t])$ to $\alpha$ grows
like $t$.
After this time
$\beta$ escapes \emph{quickly} from $\alpha$, that is, $d(\beta(t),\alpha)$
grows like $t-T_0$, and the diameter of the projection of
$\beta([T_0,t])$ is bounded. 

We recover the second point for non-hyperbolic spaces using the contraction property.
Our definition of a \emph{contracting} set $Z$, see \fullref{def:contracting},
is that the diameter of the projection of a ball tangent to $Z$ is bounded
by a function of the radius of the ball whose growth rate is less than
linear. 
Essentially this means that sets far from $Z$ have large diameter
compared to the diameter of their projection. 
In contrast to the hyperbolic case, it is not true, in general, that if
$\alpha$ is a contracting geodesic ray and $\beta$ is a geodesic ray
not asymptotic to $\alpha$ then $\beta$ has  bounded projection to
$\alpha$. 
However, we \emph{can} still  characterize the escape of $\beta$
from $\alpha$ by the relation between the growth of the projection of $\beta([0,t])$ to
$\alpha$ and the distance from $\beta(t)$ to $\alpha$.
The main technical tool we introduce is a divagation estimate that
says if $\alpha$ is contracting and $\beta$ is a quasi-geodesic then
$\beta$ cannot wander slowly away from $\alpha$; if it is to escape,
it must do so quickly. 
More precisely, once $\beta$ exceeds a
threshold distance from $\alpha$, depending on the quasi-geodesic constants of $\beta$ and the contraction
function for $\alpha$, then the distance from $\beta(t)$ to $\alpha$
grows superlinearly compared to the growth of the projection of
$\beta([0,t])$ to $\alpha$.
In fact, for the purpose of proving that fellow-travelling time is
roughly preserved by quasi-isometries it will be enough to know that the this
relationship is at least a fixed linear function.

The first point cannot be recovered, and, in fact,  the topology as described
above, using only geodesic rays, is not quasi-isometry invariant for non-hyperbolic spaces \cite{Cas16gromovboundary}.
Instead, we introduce a finer topology that we call the \emph{topology of fellow-travelling quasi-geodesics}.
The idea is that $\eta$ is close to $\zeta$ if all \emph{quasi}-geodesics tending to $\eta$ 
closely fellow-travel quasi-geodesics tending to $\zeta$ for a long time.
See \fullref{def:nbhds} for a precise definition. 
Using our divagation estimates we show that this topology is
quasi-isometry invariant. 

The use of quasi-geodesic rays in our definition is quite natural in
the setting of coarse geometry, since then the rays under
consideration do not depend on the choice of a particular metric
within a fixed quasi-isometry class. 
Geodesics, on the other hand, are highly sensitive to the choice of
metric, and it is only the presence of a very strong hypothesis like
global hyperbolicity that allows us to define a quasi-isometry
invariant boundary topology using geodesics alone.

\medskip
\begin{example}\label{motivation}
Consider $H:=\langle a,b\mid
[a,b]=1\rangle*\langle c\rangle$, which can be thought of as the fundamental group of a
flat, square torus wedged with a circle. 
Let $X$ be the universal cover, with basepoint $\bp$ above the wedge
point. 

Connected components of the preimage of the torus are Euclidean planes
isometrically embedded in $X$.
Geodesic segments contained in such a plane behave more like Euclidean
geodesics than hyperbolic geodesics.
In fact, a geodesic ray $\alpha$ based at $\bp$ is contracting if and only if there exists a
bound $B_\alpha$ such that $\alpha$ spends time at most time $B_\alpha$ in any one
of the planes. 
Let $\alpha(\infty)$ denote the equivalence class of this ray as a
point in the contracting boundary.

In Charney and Sultan's topology, if $(\alpha^n)_{n\in\mathbb{N}}$ is a
sequence of contracting geodesic rays with the $B_{\alpha^n}$
unbounded, then $(\alpha^n(\infty))$ is not a convergent sequence in
the contracting boundary.
Murray \cite{Mur15}  uses this fact to show that the 
contracting boundary is not first countable. 

In the \thistopology it will turn out that $(\alpha^n(\infty))$
converges if and only if  there exists a
contracting geodesic $\alpha$ in $X$ such that the projections of the
$\alpha^n$ to geodesics in the Bass-Serre tree of $H$ (with respect to
the given free product splitting of $H$) converge to the projection of
$\alpha$. 

From another point of view, $H$ is hyperbolic relative to the Abelian
subgroup $A:=\langle a,b\rangle$. 
We show in \fullref{thm:relhyp} that this implies that there is a
natural map from the contracting boundary of $H$ to the Bowditch
boundary of the pair $(H,A)$, and, with the \thistopology, that this
map is a topological embedding. 
The embedding statement cannot be true for Charney and Sultan's
topology, since it is not first countable.
\end{example}

\medskip

After some preliminaries in \fullref{sec:prelim}, we define the
contraction property and recall/prove some basic technical results in
\fullref{sec:contraction} concerning the behavior of geodesics relative to contracting sets.
In \fullref{sec:contraction2} we extend these results to
quasi-geodesics, and derive the key divagation estimates, see
\fullref{cor:asymp} and \fullref{keylemma}.

In \fullref{sec:topology} introduce the topology of fellow-travelling
quasi-geodesics and show that it is first countable, Hausdorff, and regular.
In \fullref{sec:invariance} we prove that it is also quasi-isometry invariant.

We compare other possible topologies in \fullref{sec:comparison}.

In \fullref{sec:metrizability} we consider the case of a finitely
generated group. 
In this case we prove that the contracting boundary is second countable, hence metrizable.

We also prove a weak version of North-South dynamics for the action of a
group on its contracting boundary in \fullref{sec:dynamics}, in the spirit of Murray's work~\cite{Mur15}.

Finally, in \fullref{sec:compact} we show that the contracting
boundary of an infinite, finitely generated group is non-empty and compact if and only if the group
is hyperbolic.

\bigskip
We thank the referee for a careful reading of our paper.

\section{Preliminaries}\label{sec:prelim}
Let $X$ be a metric space with metric $d$.  For $Z \subset X$, define:
\begin{itemize}
\item $N_rZ:=\{x\in X\mid
\exists z\in Z,\, d(z,x)<r\}$
\item $N^c_rZ:=\{x\in X\mid
\forall z\in Z,\, d(z,x)\geq r\}$
\item $\bar{N}_rZ:=\{x\in X\mid \exists
z\in Z,\, d(z,x)\leq r\}$
\item $\bar{N}^c_rZ:=\{x\in X\mid \forall
z\in Z,\, d(z,x)> r\}$
\end{itemize}

For $\cL\geq 1$ and $\cA\geq 0$, a map $\phi\from (X,d_X)\to (X',d_{X'})$ is an \emph{$(\cL,\cA)$--quasi-isometric
embedding} if for all $x,\,y\in X$:
\[\frac{1}{\cL}d_X(x,y)-\cA \leq d_{X'}(\phi(x),\phi(y))\leq \cL d_X(x,y)+\cA\]
If, in addition, $\bar{N}_\cA\phi(X)={X'}$ then $\phi$ is an \emph{$(\cL,\cA)$--quasi-isometry}.
A \emph{quasi-isometry inverse} $\bar{\phi}$ of a quasi-isometry
$\phi\from X\to {X'}$ is a quasi-isometry $\bar{\phi}\from {X'} \to X$ such
that the compositions $\phi\circ\bar{\phi}$ and $\bar{\phi}\circ\phi$
are both bounded distance from the identity map on the respective space.

A \emph{geodesic} is an isometric embedding of an interval. 
A \emph{quasi-geodesic} is a quasi-isometric embedding of an interval.
If $\alpha\from I\to X$ is a quasi-geodesic, we often use $\alpha_t$
to denote $\alpha(t)$, and conflate $\alpha$ with its image in $X$.
When $I$ is of the form $[a,b]$ or $[a,\infty)$ we will assume, by
precomposing $\alpha$ with a translation of the domain, that $a=0$.
We use $\alpha+\beta$ and $\bar{\alpha}$ to denote concatenation and
reversal, respectively.

A metric space is \emph{geodesic} if every pair of points can be
connected by a geodesic.

A metric space if \emph{proper} if closed balls are compact. 

It is often convenient to improve quasi-geodesics to be continuous,
which can be accomplished by the following lemma.
\begin{lemma}[{Taming quasi-geodesics \cite[Lemma~III.H.1.11]{BriHae99}}]
  If $X$ is a geodesic metric space and $\gamma\from [a,b]\to X$ is an
  $(\cL,\cA)$--quasi-geodesic then there exists a continuous
  $(\cL,2(\cL+\cA))$--quasi-geodesic $\gamma'$ such that
  $\gamma_a=\gamma'_a$, $\gamma_b=\gamma'_b$ and the Hausdorff
  distance between $\gamma$ and $\gamma'$ is at most $\cL+\cA$.
\end{lemma}
\begin{proof}
  Define $\gamma'$ to agree with $\gamma$ at the endpoints and at
  integer points of $[a,b]$, and then connect the dots by geodesic
  interpolation.  
\end{proof}

A subspace $Z$ of a geodesic metric space $X$ is \emph{$\cA$--quasi-convex}
for some $A \geq 0$
if every geodesic connecting points in $Z$ is contained in
$\bar{N}_\cA Z$.

If $f$ and $g$ are functions then we say $f\preceq g$ if there exists 
a constant $C>0$ such that $f(x)\leq Cg(Cx+c)+C$ for all $x$. 
If $f\preceq g$ and $g\preceq f$ then we write $f\asymp g$.

\bigskip

We will give a detailed account of the contracting property in the
next section, but let us first take a moment to recall alternate
characterizations, which will prove useful later in the paper.

A subspace $Z$ of a metric space $X$ is \emph{$\mu$--Morse} 
for some $\mu:[1,\infty)\times[0,\infty)\rightarrow\R$ if for every
$\cL\geq 1$ and every $\cA\geq 0$, every
$(\cL,\cA)$--quasi-geodesic with endpoints in $Z$ is
contained in $\bar{N}_{\mu(\cL,\cA)}Z$.
We say $Z$ is \emph{Morse} if there exists $\mu$ such that it is
$\mu$--Morse.
It is easy to see that the property of being Morse is invariant under
quasi-isometries. 
In particular, a subset of a finitely generated group $G$ is Morse in one
Cayley graph of $G$ if and only if it is Morse in every Cayley graph
of $G$. 
Thus, we can speak of a \emph{Morse subset} of $G$ without specifying
a finite generating set. 

A set $Z$ is called \emph{$t$--recurrent}\footnote{This characterization was
    introduced in \cite{DruMozSap10} with $t=1/3$ for $Z$ a quasi-geodesic. The idea is that a short curve must pass near the `middle
    third' of the subsegment of $Z$ connecting its endpoints. The
    property, again only for quasi-geodesics, but for variable $t$, is called `middle recurrence' in \cite{AouDurTay16}.}, for $t\in(0,1/2)$, if for
every $C\geq 1$ there exists $D\geq 0$ such that if $p$ is a path with
endpoints $x$ and $y$ on $Z$ such that the ratio of the length of $p$ to the
distance between its endpoints is at most $C$, then there exists a
point $z\in Z$ such that $d(p,z)\leq D$ and
$\min\{d(z,x),\,d(z,y)\}\geq td(x,y)$. 
The set $Z$ is called \emph{recurrent}
if it is $t$--recurrent for every $t\in (0,1/2)$.
\begin{theorem}\label{morseequivalent}
  Let $Z$ be a subset of a geodesic metric space $X$. The following
  are equivalent:
  \begin{enumerate}
\item $Z$ is Morse.\label{item:morse}
  \item $Z$ is contracting.\label{item:contracting}
\item $Z$ is recurrent.\label{item:recurrent}
\item There exists $t\in (0,1/2)$ such that $Z$ is $t$-recurrent.\label{item:trecurrent}
  \end{enumerate}
Moreover, each of the equivalences are `effective', in the sense that
the defining function of one property determines the defining
functions of each of the others.
\end{theorem}
\begin{proof}
  The equivalence of \eqref{item:morse} and \eqref{item:contracting} is
  proved in \cite{ArzCasGrub}. That \eqref{item:recurrent} implies
  \eqref{item:trecurrent} is obvious.
The implications `\eqref{item:contracting} implies \eqref{item:recurrent}' and
`\eqref{item:trecurrent} implies \eqref{item:morse}' are proved in
\cite{AouDurTay16} for the case that $Z$ is a quasi-geodesic, but
their proofs go through with minimal change for arbitrary subsets $Z$.
\end{proof}

\section{Contraction}\label{sec:contraction}
\begin{definition}
  We call a function $\rho$ \emph{sublinear} if it is non-decreasing,
  eventually non-negative, and $\lim_{r\to\infty}\rho(r)/r=0$.
\end{definition}
\begin{definition}\label{def:contracting}
  Let $X$ be a proper geodesic metric space.
Let $Z$ be a closed subset of $X$, and let $\pi_Z\from X\to
2^Z:x\mapsto\{z\in Z\mid d(x,z)=d(x,Z)\}$ be
closest point projection to $Z$.
Then, for a sublinear function $\rho$, we say that $Z$ is
\emph{$\rho$--contracting} if for all $x$ and $y$ in $X$:
\[d(x,y)\leq d(x,Z)\implies \diam \pi_Z(x)\cup\pi_Z(y)\leq\rho(d(x,Z))\]

We say $Z$ is \emph{contracting} if there exists a sublinear function
$\rho$ such that $Z$ is $\rho$--contracting.
We say a collection of subsets $\{Z_i\}_{i\in\mathcal{I}}$ is
\emph{uniformly contracting} if there exists a sublinear function
$\rho$ such that for every $i\in\mathcal{I}$ the set $Z_i$ is $\rho$--contracting. 
\end{definition}
We shorten $\pi_Z$ to $\pi$ when $Z$ is clear from context.

Let us stress that the closest point projection map is set-valued, and
there is no bound on the diameter of image sets other than that
implied by the definition.

In a tree every convex subset is $\rho$--contracting where $\rho$ is
identically 0. 
More generally, in a hyperbolic space a set is contracting if and only
if it is quasi-convex.
In fact, in this case more is true: the contraction function is bounded in terms of
the hyperbolicity and quasi-convexity constants.
We call a set \emph{strongly contracting} if it is contracting with
bounded contraction function.

The more general \fullref{def:contracting} was introduced by
Arzhantseva, Cashen, Gruber, and Hume to characterize Morse
geodesics in small cancellation groups \cite{ArzCasGrua}.

The concept of strong contraction (sometimes simply called
`contraction' in the literature) has been studied
before, notably by Minsky \cite{Min96} to
describe axes of pseudo-Anosov mapping classes in Teichm\"uller space,
by Bestvina and Fujiwara \cite{BesFuj09} to describe axes of
rank-one isometries of CAT(0) spaces (see also Sultan \cite{Sul14}),
and by Algom-Kfir \cite{Alg11} to describe axes of fully irreducible
free group automorphisms acting on Outer Space.

Masur and Minsky \cite{MasMin99} introduced a different notion of
contraction that requires the existence of constants $A$ and $B$ such that:
\[d(x,y)\leq d(x,Z)/A\implies \diam \pi_Z(x)\cup\pi_Z(y)\leq B\]
This is satisfied, for example, by axes of pseudo-Anosov elements in
the mapping class group (as opposed to Teichm\"uller space).
Some authors refer to this property as `contraction', eg \cite{Beh06,
  DucRaf09, AbbBehDur17}.
It is not hard to show that this version implies the
version in \fullref{def:contracting} with the contraction function $\rho$ being logarithmic.

We now recall some further results about contracting sets in a geodesic metric space $X$.

\begin{lemma}[{\cite[Lemma~6.3]{ArzCasGrub}}]\label{Hausdorff}
Given a sublinear function $\rho$ and a constant $C\geq 0$ there
exists a sublinear function $\rho'\asymp \rho$ such that if $Z\subset X$ and
$Z'\subset X$ have Hausdorff distance at most $C$ and $Z$ is
$\rho$--contracting then $Z'$ is $\rho'$--contracting.
\end{lemma}

\begin{theorem}[{Geodesic Image Theorem \cite[Theorem~7.1]{ArzCasGrub}}]\label{GIT}
For $Z \subset X$, 
there exists a sublinear function $\rho$ so that $Z$ is
$\rho$--contracting if and only if there exists a sublinear function
$\rho'$ and a constant $\kappa_\rho$ so that
for every geodesic segment $\gamma$,
with endpoints denoted $x$ and $y$, if $d(\gamma, Z)\geq \kappa_\rho$ then
$\diam\pi (\gamma)\leq \rho'(\max\{d(x,Z),\,d(y,Z)\})$.
Moreover $\rho'$ and $\kappa_\rho$ depend only on $\rho$ and vice-versa, with $\rho'\asymp\rho$. 
\end{theorem}

An easy consequence is that there exists a $\kappa'_\rho$ such that if
$\gamma$ is a geodesic segment with endpoints at distance at most
$\kappa_\rho$ from a $\rho$--contracting set $Z$ then $\gamma\subset\bar{N}_{\kappa'_\rho}(Z)$.

The following is a special case of \cite[Proposition~8.1]{ArzCasGrub}.
\begin{lemma}\label{quasiconvexunion}
  Given a sublinear function $\rho$ and a constant $C\geq 0$ there
  exists a constant $B$ such that if $\alpha$ and $\beta$ are
  $\rho$--contracting geodesics such that their 
  initial points $\alpha_0$ and $\beta_0$ satisfy
  $d(\alpha_0,\beta_0)=d(\alpha,\beta)\leq C$ then $\alpha\cup\beta$
  is $B$--quasi-convex.
\end{lemma}

The next two lemmas are easy-to-state generalizations of results that
are known for strong contraction. 
The proofs are rather tedious, due to the weak hypotheses, so we
postpone them until after \fullref{lem:almosttriangle}.

\begin{lemma}\label{lem:subsegment}
Given a sublinear function $\rho$ there is a sublinear function
$\rho'\asymp\rho$ such that every subsegment of a $\rho$--contracting
geodesic is $\rho'$--contracting.
\end{lemma}

\begin{lemma}\label{lem:combination}
  Given a sublinear function $\rho$ there is a sublinear function
$\rho'\asymp\rho$ such that if $\alpha$ and $\beta$ are
$\rho$--contracting geodesic rays or segments such that
$\gamma:=\bar{\alpha}+\beta$ is geodesic, then $\gamma$ is $\rho'$--contracting.
\end{lemma}

Given $C\geq 0$ a \emph{geodesic $C$--almost triangle} is a trio of
geodesics $\alpha^i\from [a_i,b_i]\to X$, for $i\in\{0,\,1,\,2\}$ and
$a_i\leq 0\leq b_i\in\mathbb{R}\cup\{-\infty,\infty\}$, such that
 for each $i\in\{0,\,1,\,2\}$, with
 scripts taken modulo 3, we have:
 \begin{itemize}
 \item $b_i<\infty$ if and only if $a_{i+1}>-\infty$.
   \item If $b_i$ and $a_{i+1}$ are finite then
     $d(\alpha^i_{b_i},\alpha^{i+1}_{a_{i+1}})\leq C$.
     \item If $b_i$ and $a_{i+1}$ are not finite then $\alpha^i_{[0,\infty)}$ and
       $\bar{\alpha}^{i+1}_{[0,\infty)}=\alpha^{i+1}_{(-\infty,0]}$
       are asymptotic.
 \end{itemize}

\begin{lemma}\label{lem:almosttriangle}
Given a sublinear function $\rho$ and constant $C\geq 0$ there is a sublinear
function $\rho'\asymp\rho$ such that if $\alpha$, $\beta$, and
$\gamma$ are a geodesic $C$--almost triangle and $\alpha$ and $\beta$
are $\rho$--contracting then $\gamma$ is $\rho'$--contracting.
\end{lemma}
\begin{proof}
First suppose $\alpha$, $\beta$, and $\gamma$ are segments.
  By \fullref{quasiconvexunion}, there exists a $B$ depending only on
  $\rho$ and $C$ such that $\alpha\cup\beta$ is $B$--quasi--convex.
Thus, we can replace $\alpha\cup\beta$ by a single geodesic segment
$\delta$ whose endpoints are $C$--close to the endpoints of $\gamma$.
Furthermore, $\delta$ is a union of two subsegments, 
one of which has endpoints within distance $B$ of $\alpha$, and the
other of which has endpoints within distance $B$ of $\beta$.
Consequently, by \fullref{GIT} there exists $B'$ so that these
two subsegments are
$B'$--Hausdorff equivalent to subsegments of $\alpha$ and of $\beta$,
respectively. 
Applying \fullref{lem:subsegment}, \fullref{Hausdorff}, and
\fullref{lem:combination}, there is a $\rho''\asymp\rho$ depending on
$\rho$ and $B'$ such that
$\delta$ is $\rho''$--contracting.
\fullref{GIT} implies that since $\gamma$ and $\delta$ are close at
their endpoints, they stay close along their entire lengths, so their
Hausdorff distance is determined by $\rho''$ and $C$, hence by $\rho$
and $C$. 
Applying \fullref{Hausdorff} again, we conclude $\gamma$ is
$\rho'$--contracting with $\rho'\asymp \rho''\asymp\rho$ depending only on
$\rho$ and $C$.  

In the case of an ideal triangle, where not all three sides are
segments, replace $C$ by $\max\{C,\kappa_\rho\}$. 
\fullref{GIT} implies that if, say, $\gamma$ and $\bar{\alpha}$ have
asymptotic tails then the set of points $\gamma$ that come
$\kappa_\rho$--close to $\alpha$ is unbounded. 
Truncate the triangle at such a $C$--close pair of points.
Doing the same for other ideal vertices, we get a $C$--almost triangle
to which we can apply the previous argument and conclude that a
subsegment of $\gamma$ is $\rho'$--contracting. 
Since $\gamma$ comes $\kappa_\rho$--close to ${\alpha}$ on an
unbounded set, we can repeat the argument for larger and larger almost
triangles approximating $\alpha$, $\beta$, $\gamma$, and find that
every subsegment of $\gamma$ is contained in a $\rho'$--contracting subsegment,
which implies that $\gamma$ itself is $\rho'$--contracting.
\end{proof}

\begin{definition}
  If $Z$ is a subset of $\mathbb{R}$ define the \emph{interval of
    $Z$}, $\invl(Z)$, to be the smallest closed interval containing
  $Z$. 
If $\gamma\from I\to X$ is a geodesic and $Z$ is a subset of
  $\gamma$ let $\invl(Z):=\gamma(\invl(\gamma^{-1}(Z)))$.
\end{definition}

\begin{proof}[Proof of \fullref{lem:subsegment}]
  Let $\gamma\from I\to X$ be a $\rho$--contracting geodesic. 
Let $J:=[j_0,j_1]$  be  a subinterval of $I$. 
Let  $\rho''\asymp \rho$ be the function given by \fullref{GIT}, and
let $\kappa'_\rho$ be the constant defined there.
We claim it
suffices to take $\rho'(r):=2(2\kappa'_\rho+\rho''(2r)+\rho(2r))$.

First we show that if $\pi_{\gamma_I}(x)$ misses $\gamma_J$ then
$\pi_{\gamma_J}(x)$ is relatively close to one of the endpoints of
$\gamma_J$.
This is automatic if $\diam \gamma_J\leq\rho(d(x,\gamma_J))$,
so assume not. 
With this assumption, $\pi_{\gamma_I}(x)$ cannot contain points on both sides of
$\gamma_J$, that is, if $\gamma^{-1}(\pi_{\gamma_I}(x))$ contains a
point less than $j_0$ then it does not also contain one greater than
$j_1$, and vice versa.
Suppose that $\gamma^{-1}(\pi_{\gamma_I}(x))$ is contained in $(-\infty,j_0)$.
Let $\beta$ be a geodesic from $x$ to a point $y$ in $\pi_{\gamma_J}(x)$.
There exists a first time $s$ such that $d(\beta_s,\gamma_I)=\kappa_\rho$.
By \fullref{GIT},
$\diam \pi_{\gamma_I}(\beta|_{[0,s]})\leq \rho''(d(x,\gamma_I))$.
Suppose that $\gamma_{j_0}\in\invl
(\pi_{\gamma_I}(\beta|_{[0,s]}))$. 
Then there is a first time $s'\in[0,s]$, such that $\pi_I(\beta_{s'})$
contains a point in $\gamma_{[j_0,\infty)}$.
By the assumption on the diameter of $\gamma_J$, we actually have $\pi_{\gamma_I}(\beta_{s'})\cap \gamma_J \neq \emptyset$,
so $y \in \pi_{\gamma_J}(\beta_{s'}) \subset \pi_{\gamma_I}(\beta_{s'})
\subset \pi_{\gamma_I}(\beta|_{[0,s]})$ and
$\diam \gamma_{j_0}\cup\pi_{\gamma_J}(x)\leq  \rho''(d(x,\gamma_I))$.
Otherwise, if $\gamma_{j_0}\notin\invl
(\pi_{\gamma_I}(\beta|_{[0,s]}))$, then let $t>s$ be the first time such that $\gamma_{j_0}\in\invl
\pi_{\gamma_I}(\beta|_{[s,t]})$. 
Again, $y\in\pi_{\gamma_I}(\beta_t)$. 
Since the points of $\beta$ after $\beta_s$, are contained in $\bar{N}_{\kappa'_\rho}\gamma$, 
for all small $E>0$ we have $\diam
\pi_{\gamma_I}\beta_{t-E}\cup\pi_{\gamma_I}\beta_t\leq
E +2\kappa'_\rho$. 
Therefore, $d(\gamma_{j_0},y)\leq
d(y,\pi_{\gamma_I}(\beta_{t-E}))\leq E+2\kappa'_\rho$, for all sufficiently small $E$.
We conclude:
\begin{equation}
  \label{eq:3}
 \diam \gamma_{j_0}\cup\pi_{\gamma_J}(x)\leq \max\{2\kappa'_\rho,\rho''(d(x,\gamma_I))\} 
\end{equation}

Now suppose $x$ and $y$ are points such that $d(x,y)\leq
d(x,\gamma_J)$. 
Note that $d(y,\gamma_J)\leq 2d(x,\gamma_J)$.
We must show $\diam \pi_{\gamma_J}(x)\cup
\pi_{\gamma_J}(y)$ is bounded by a sublinear function of
$d(x,\gamma_J)$. 
There are several cases, depending on whether
$\pi_{\gamma_I}(x)$ and $\pi_{\gamma_I}(y)$ hit $\gamma_J$.

\casen{1}{$\gamma_J\cap\pi_{\gamma_I}(x)\neq \emptyset$ and
  $\gamma_J\cap\pi_{\gamma_I}(y)\neq \emptyset$} 
In this case $\pi_{\gamma_J}(x)\subset \pi_{\gamma_I}(x)$, and
likewise for $y$, so: 
\[\diam \pi_{\gamma_J}(x)\cup
\pi_{\gamma_J}(y)\leq \diam \pi_{\gamma_I}(x)\cup
\pi_{\gamma_I}(y)\leq\rho(d(x,\gamma_I))=\rho(d(x,\gamma_J))\] 

\casen{2}{$\gamma^{-1}(\pi_{\gamma_I}(x))<j_0$ and
  $\gamma^{-1}(\pi_{\gamma_I}(y))<j_0$}
By (\ref{eq:3}) twice:
\begin{align*}
  \diam\pi_{\gamma_J}(x)\cup\pi_{\gamma_J}(y)&\leq
2\max\{2\kappa'_\rho,\rho''(d(x,\gamma_I)),\rho''(d(y,\gamma_I))\}\\&\leq
2\max\{2\kappa'_\rho,\rho''(2d(x,\gamma_J))\}
\end{align*}

\casen{3}{$\gamma^{-1}(\pi_{\gamma_I}(y))<j_0$ and
  $\gamma_J\cap\pi_{\gamma_I}(x)\neq \emptyset$}
In this case $\pi_{\gamma_J}(x)\subset \pi_{\gamma_I}(x)$ and
  $d(x,y)\leq d(x,\gamma_J)=d(x,\gamma_I)$, so $\diam
  \pi_{\gamma_I}(y)\cup\pi_{\gamma_J}(x)\leq \rho(d(x,\gamma_J))$.
By hypothesis,
$\gamma_{j_0}\in\invl(\pi_{\gamma_I}(y)\cup\pi_{\gamma_J}(x))$, and by
(\ref{eq:3}): $ \diam \gamma_{j_0}\cup\pi_{\gamma_J}(y)\leq
\max\{2\kappa'_\rho,\rho''(2d(x,\gamma_J))\}$.
Thus: 
\[\diam \pi_{\gamma_J}(x)\cup\pi_{\gamma_J}(y)\leq \max\{\rho(d(x,\gamma_J)) ,2\kappa'_\rho,\rho''(2d(x,\gamma_J))\}\]

\casen{4}{$\gamma^{-1}(\pi_{\gamma_I}(x))<j_0$ and
  $\pi_{\gamma_I}(y) \cap \gamma_{[j_0,\infty)} \neq \emptyset$} 
If $j_1-j_0\leq 2\rho(2d(x,\gamma_J))$ then there is nothing more to 
prove, so assume not. 
Let $\beta$ be a geodesic from $x$ to $y$.
For all $z\in\beta$:
\[d(x,z)+d(z,y)=d(x,y)\leq d(x,\gamma_J)\leq d(x,z)+d(z,\gamma_J)\]
This implies $d(z,y)\leq d(z,\gamma_J)$.
Let $z$ be the first point on $\beta$ such that $\gamma^{-1}(\pi_{\gamma_I}(z))$ contains a point
greater than or equal to $j_0$.
By the hypothesis on $|J|$,
$\gamma^{-1}(\pi_{\gamma_I}(z))<j_1$. 
This means $\diam
\pi_{\gamma_J}(z)\cup\pi_{\gamma_J}(y)$ is controlled by one of the
previous cases, and it suffices to control $\diam \pi_{\gamma_J}(x)\cup\pi_{\gamma_J}(z)$.

We know from (\ref{eq:3}) that $\pi_{\gamma_J}(x)$ is
$\max\{2\kappa'_\rho,\rho''(d(x,\gamma_J))\}$-close to $\gamma_{j_0}$, so it suffices to control $\diam
\gamma_{j_0}\cup\pi_{\gamma_J}(z)$.
Take a point $w\neq z$ on $\beta$ before $z$ such that $d(z,w)\leq d(z,\gamma_I)$.
By hypothesis, $\gamma_{j_0}\in\invl
\pi_{\gamma_I}(w)\cup\pi_{\gamma_I}(z)$, but $\diam
\pi_{\gamma_I}(w)\cup\pi_{\gamma_I}(z)\leq
\rho(d(z,\gamma_I))=\rho(d(z,\gamma_J))\leq \rho(2d(x,\gamma_J))$.

\smallskip
Up to symmetric arguments, this exhausts all the cases. 
\end{proof}

\begin{proof}[Proof of \fullref{lem:combination}]
Let $\alpha$ and $\beta$ be $\rho$--contracting geodesic segments or
rays with $\alpha_0=\beta_0$ such that $\gamma:=\bar{\alpha}+\beta$ is
geodesic.

First suppose that $x$ is a point such that
$\pi_\gamma(x)\cap\alpha\neq\emptyset$ and
$\pi_\gamma(x)\cap\beta\neq\emptyset$.
Let $\delta$ be a geodesic from $x$ to $\alpha_0=\beta_0$.
Recall from \fullref{GIT} that once $\delta$ enters the
$\kappa_\rho$--neighborhood of either $\alpha$ or $\beta$ then it cannot
leave the $\kappa'_\rho$--neighborhood. 
Thus, $\delta$ intersects at most one of $\bar N_{\kappa_\rho}\alpha\setminus
N_{2\kappa'_\rho}\alpha_0$ or $\bar N_{\kappa_\rho}\beta\setminus N_{2\kappa'_\rho}\beta_0$.
Without loss of generality, suppose $\delta$ does not intersect $\bar N_{\kappa_\rho}\beta\setminus
N_{2\kappa'_\rho}\beta_0$.
Let $t$ be the first time such that $d(\delta_t,\beta)=\kappa_\rho$.
Then $d(\delta_t,\beta_0)\leq 2\kappa'_\rho$ and, by \fullref{GIT}, there
is a sublinear $\rho''\asymp\rho$ such that $\diam
\pi_\beta(\delta|_{[0,t]})\leq \rho''(d(x,\beta))=\rho''(d(x,\gamma))$.
In particular, this means $\diam \pi_\beta(x)\cup\beta_0\leq\diam
\pi_\beta(\delta)\leq 4\kappa'_\rho+\rho''(d(x,\gamma))$.
Now let $\delta'$ be a geodesic from $x$ to a point $x'\in\pi_\beta(x)$,
and project $\delta'$ to $\alpha$.
Since $\bar{\alpha}+\beta$ is geodesic, $\diam \pi_\alpha(x)\cup\alpha_0\leq
\diam\pi_\alpha\delta'\leq \rho''(\max\{d(x,\alpha),d(x',\alpha)\})$ by
\fullref{GIT}.
We have already established that $d(x',\alpha)\leq 4\kappa'_\rho+\rho''(d(x,\gamma))$.
Since $\rho''$ grows sublinearly, 
$d(x,\alpha)>4\kappa'_\rho+\rho''(d(x,\gamma))$ except for $d(x,\alpha)$
less than some bound depending only on $\rho$ and $\rho''$. 
We conclude that there is a sublinear function $\rho'''\asymp\rho$
depending only on $\rho$ such that
$\diam\pi_\alpha(x)\cup\alpha_0\leq\rho'''(d(x,\gamma))$ and
$\diam\pi_\beta(x)\cup\beta_0\leq\rho'''(d(x,\gamma))$, hence $\diam\pi_\gamma(x)\leq 2\rho'''(d(x,\gamma))$.

Now suppose $x,\,y\in X$ are points such that $d(x,y)\leq
d(x,\gamma)$.
There are several cases according to where $\pi_\gamma(x)$ and
$\pi_\gamma(y)$ lie.

\casen{1}{$\pi_\gamma(x)\cap\alpha\neq\emptyset\neq\pi_\gamma(y)\cap\alpha$}
Then $d(x,y)\leq d(x,\gamma)=d(x,\alpha)$, so contraction for $\alpha$
implies $\diam \pi_\alpha(x)\cup\pi_\alpha(y)\leq
\rho(d(x,\alpha))=\rho(d(x,\gamma))$. 
There are four sub-cases to check, according to whether
$\pi_\gamma(x)$ and $\pi_\gamma(y)$ hit $\beta$. These are easy to
check, with the worst bound being
$\diam\pi_\gamma(x)\cup\pi_\gamma(y)\leq \rho(d(x,\gamma))+2\rho'''(2d(x,\gamma))$.

\casen{2}{$\pi_\gamma(x)\cap\beta=\emptyset=\pi_\gamma(y)\cap\alpha$}
Let $\delta$ be a geodesic from $x$ to $y$.
Let $w$ be the first point on $\delta$ such that
$\pi_\gamma(w)\cap\beta\neq\emptyset$.
Then $d(w,\alpha)=d(w,\beta)=d(w,\gamma)\leq 2d(x,\gamma)$ and
$d(y,\beta)=d(y,\gamma)\leq 2d(x,\gamma)$.
We can apply that $\alpha$ is $\rho$-contracting to the pair $x,w$ since
$d(x,w) \leq d(x,y) \leq d(x,\gamma) = d(x,\alpha)$.
Likewise, we can apply that $\beta$ is $\rho$-contracting to $w,y$ since
$d(x,w)+d(w,y) = d(x,y) \leq d(x,\gamma) \leq d(x,w)+d(w,\gamma)$
so $d(w,y) \leq d(w,\gamma)$.
We conclude:
\begin{align*}
  \diam\pi_\gamma(x)\cup\pi_\gamma(y)&\leq\diam\pi_\alpha(x)\cup\pi_\alpha(w)+\diam\pi_\gamma(w)\\&\qquad+\diam\pi_\beta(w)\cup\pi_\beta(y)\\
&\leq  \rho(d(x,\alpha))+2\rho'''(d(w,\gamma))+\rho(d(w,\beta))\\
&\leq   \rho(d(x,\gamma))+2\rho'''(2d(x,\gamma))+\rho(2d(x,\gamma))
\end{align*}

By symmetry these two cases cover all possibilities, so it suffices to
define $\rho'(r):=2\rho(2r)+2\rho'''(2r)$.
\end{proof}

\section{Contraction and Quasi-geodesics}\label{sec:contraction2}
In this section we explore the behavior of a quasi-geodesic ray based
at a point in a contracting set $Z$. 
The main conclusion is that such a ray can stay close to $Z$ for an
arbitrarily long time, but once it exceeds a certain threshold
distance depending on the quasi-geodesic constants and the contraction
function then the ray must escape $Z$ at a definite linear
rate.

\begin{definition}\label{def:kappa}
  Given a sublinear function $\rho$ and constants $\cL\geq 1$ and
  $\cA\geq 0$, define:
  \[\kappa(\rho,\cL,\cA):=\max\{3\cA,3\cL^2,1+\inf\{R>0\mid \forall
  r\geq R,\, 3\cL^2\rho(r)\leq r\}\}\]
Define: \[\kappa'(\rho,\cL,\cA):= (\cL^2+2)(2\kappa(\rho,\cL,\cA)+\cA)\]
\end{definition}

\begin{remark}
  For the rest of the paper $\kappa$ and $\kappa'$ always refer to the
  functions defined in \fullref{def:kappa}. We use them frequently and
  without further reference.
\end{remark}

This definition implies that for $r\geq \kappa(\rho,\cL,\cA)$ we have:
 \begin{equation}
    \label{eq:1}
    r-\cL^2\rho(r)-\cA\geq \frac{1}{3}r\geq \cL^2\rho(r)
  \end{equation}

  An inspection of the proof of \cite[Theorem 7.1]{ArzCasGrub} gives that $\kappa(\rho,1,0)\geq \kappa_\rho$ and $\kappa'(\rho,1,0)\geq \kappa'_\rho$, so the results of
the previous section still hold using $\kappa(\rho,1,0)$ and
$\kappa'(\rho,1,0)$.  
Enlarging the constants lets us give unified proofs for geodesics and quasi-geodesics.

\begin{theorem}[Quasi-geodesic Image Theorem]\label{QGIT}
Let $Z\subset X$ be $\rho$--contracting. Let $\beta\from [0,T]\to X$ be a continuous
$(\cL,\cA)$--quasi-geodesic segment.  
If $d(\beta,Z)\geq \kappa(\rho,\cL,\cA)$ then: 
\[\diam\pi(\beta_0)\cup\pi(\beta_T) \leq \frac{L^2+1}{L^2}\left(A+d(\beta_T,Z)\right)+\frac{L^2-1}{L^2}d(\beta_0,Z)+2\rho(d(\beta_0,Z))\]  
\end{theorem}
The proof generalizes the proof of the Geodesic Image Theorem to work
for quasi-geodesics.
We typically apply the result when $d(\beta_T,Z)=\kappa(\rho,L,A)$, in
which case the theorem says that for fixed $\rho$, $\cL$, and $\cA$ the projection diameter of $\beta$ is
bounded in terms of $d(\beta_0,Z)$.
In particular, when $\beta$ is geodesic, or, more generally, when
$\cL=1$, the bound is sublinear in $d(\beta_0,Z)$, and we recover a
version of the Geodesic Image Theorem.
With a little more work we can prove this stronger statement for
quasi-geodesics as well.
Although we do not need it in this paper, the stronger version may be of independent interest, so we include a proof at the
end of this section (see \fullref{QGIT2}). 
\begin{proof}[Proof of \fullref{QGIT}]
  Let $t_0:=0$.  For each $i \in \mathbb{N}$ in turn, let $t_{i+1}$ be the first time such that
$d(\beta_{t_i},\beta_{t_{i+1}})=d(\beta_{t_i},Z)$, or set $t_{i+1}=T$
if no such time exists.
Let $j$ be the first index such that $d(\beta_{t_j},\beta_T)\leq d(\beta_{t_j},Z)$.

\begin{align*}
T&=T-t_j+\sum_{i=0}^{j-1}(t_{i+1}-t_i)\\
&\geq
  \frac{1}{\cL}\left(d(\beta_{t_j},\beta_T)-d(\beta_{t_j},Z)+\sum_{i=0}^j(d(\beta_{t_i},Z)-\cA)\right)\\
&\geq \frac{1}{\cL}\left(-d(\beta_T,Z)+\sum_{i=0}^j(d(\beta_{t_i},Z)-\cA)\right)\\
\end{align*}

On the other hand:
\begin{align*}
  \frac{T}{\cL}-\cA&\leq d(\beta_0,\beta_T)\\
&\leq d(\beta_0,Z)+\diam \pi(\beta_0)\cup\pi(\beta_T)+d(Z,\beta_T)\\
&\leq d(\beta_0,Z)+d(\beta_T,Z)+\sum_{i=0}^j\rho(d(\beta_{t_i},Z))
\end{align*}
Combining these gives:
  \begin{multline*}
\sum_{i=1}^j \left( d(\beta_{t_i},Z)-\cL^2\rho(d(\beta_{t_i},Z))-\cA \right)\\
\leq d(\beta_T,Z)+\cL^2\left(\cA+d(\beta_0,Z)+d(\beta_T,Z)\right)\\\qquad-\left(d(\beta_{0},Z)-\cL^2\rho(d(\beta_{0},Z))-\cA\right)    
  \end{multline*}
By (\ref{eq:1}), the left-hand side is at least
$L^2\sum_{i=1}^j\rho(d(\beta_{t_i},Z))$, so:
\begin{align*}
  \diam\pi(\beta_0)&\cup\pi(\beta_T)\leq
                                     \sum_{i=0}^j\rho(d(\beta_{t_i},Z))\\
&\leq \frac{L^2+1}{L^2}\left(A+d(\beta_T,Z)\right)+\frac{L^2-1}{L^2}d(\beta_0,Z)+2\rho(d(\beta_0,Z))\qedhere
\end{align*}
\end{proof}

\begin{corollary}\label{cor:asymp}
  Let $Z$ be $\rho$--contracting and let $\beta$ be a continuous
  $(\cL,\cA)$--quasi-geodesic ray with $d(\beta_0, Z)\leq \kappa(\rho,\cL,\cA)$.
There are two possibilities:
\begin{enumerate}
\item 
The set $\{t\mid d(\beta_t,Z)\leq \kappa(\rho,\cL,\cA)\}$ is unbounded and $\beta$ is contained in the
$\kappa'(\rho,\cL,\cA)$--neighborhood of $Z$.\label{item:asymptotic}
\item
There exists a last time $T_0$ such
that $d(\beta_{T_0},Z)=\kappa(\rho,\cL,\cA)$ and:
\begin{equation}
  \label{eq:2} \forall t, \quad
d(\beta_t,Z)\geq \frac{1}{2L}(t-T_0)-2(\cA+\kappa(\rho,\cL,\cA))\tag{$\star$}
\end{equation}
\end{enumerate}
\end{corollary}
\begin{proof}
Let $\kappa:=\kappa(\rho,\cL,\cA)$.
Let $[a,b]$ be a maximal interval such
 that $d(\beta_t,Z)\geq \kappa$ for $t\in[a,b]$ and
 $d(\beta_{a},Z)=d(\beta_{b},Z)=\kappa$.

For $t\in [a,b]$ we have $d(\beta_t,Z)\leq \kappa+\cL\cdot(b-a)/2
+\cA$.
Since $\beta$ is quasi-geodesic:
\[(b-a)\leq \cL(A+d(\beta_a,\beta_b))\leq \cL(A+2\kappa+\diam
 \pi(\beta_a)\cup\pi(\beta_b))\]

 \fullref{QGIT} implies:
 \begin{align*}
\diam 
 \pi(\beta_a)\cup\pi(\beta_b)&\leq\frac{L^2+1}{L^2}\left(A+\kappa\right)+\frac{L^2-1}{L^2}\kappa+\frac{2\kappa}{3L^2}\\
&=\frac{L^2+1}{L^2}A+\frac{6L^2+2}{3L^2}\kappa
 \end{align*}
Putting these estimates together yields:
\[d(\beta_t,Z)< (\cL^2+2)(2\kappa+\cA)=\kappa'(\rho,\cL,\cA)\]
Thus, once $\beta$ leaves the
$\kappa'(\rho,\cL,\cA)$--neighborhood of $Z$ it can never return
to the $\kappa(\rho,\cL,\cA)$--neighborhood of $Z$.
If $\{t\mid d(\beta_t,Z)\leq \kappa\}$ is unbounded then $\beta$ never
leaves the
$\kappa'(\rho,\cL,\cA)$--neighborhood of $Z$.

Suppose now that there does exist some last time $T_0$ such that
$d(\beta_{T_0},Z)=\kappa$.
Any segment $\beta_{[T_0,t]}$ stays outside
$\nbhd_{\kappa}Z$, so apply \fullref{QGIT} to see:
\begin{align*}
  \frac{t-T_0}{\cL}-\cA&\leq d(\beta_t,\beta_{T_0})\\
&\leq d(\beta_t,Z)+\diam\pi(\beta_t)\cup\pi(\beta_{T_0})+\kappa\\
&\leq \frac{6\cL^2-1}{3\cL^2}d(\beta_t,Z)+\frac{\cL^2+1}{\cL^2}(\cA+\kappa)+\kappa
\end{align*}
Thus:
\[d(\beta_t,Z)\geq\frac{3L}{6L^2-1}(t-T_0)-\frac{6L^2+3}{6L^2-1}(\cA+\kappa)\qedhere\]
\end{proof}

\begin{lemma}\label{lem:hausdorff}
Suppose $\alpha$ is a continuous, $\rho$--contracting
$(\cL,\cA)$--quasi-geodesic and $\beta$ is a continuous
$(\cL,\cA)$--quasi-geodesic ray such that $d(\alpha_0,\beta_0)\leq
\kappa(\rho,\cL,\cA)$.
If there are $r,\,s\in [0,\infty)$ such that $d(\alpha_r,\beta_s)\leq
\kappa(\rho,\cL,\cA)$ then $d_{Haus}(\alpha_{[0,r]},\beta_{[0,s]})\leq
\kappa'(\rho,\cL,\cA)$.
If $\alpha_{[0,\infty)}$ and $\beta_{[0,\infty)}$ are asymptotic then
their Hausdorff distance is at most $\kappa'(\rho,\cL,\cA)$.
\end{lemma}
\begin{proof}
 
\fullref{cor:asymp} \eqref{item:asymptotic} reduces the asymptotic
case to the bounded case and shows that $\beta$ is contained in
  $\bar{N}_{\kappa'(\rho,\cL,\cA)}\alpha$. 

For the other direction, suppose that $(a,b)$ is a maximal open subinterval of the domain of $\alpha$ such
that $\alpha_{(a,b)}\cap\pi_\alpha(\beta\cap\bar{N}_{\kappa(\rho,\cL,\cA)}\alpha)=\emptyset$. 
For subsegments of $\beta$ contained in
$\bar{N}_{\kappa(\rho,\cL,\cA)}\alpha$ the projection to $\alpha$ has
jumps of size at most $2\kappa(\rho,\cL,\cA)$.
For subsegments of $\beta$ outside
$\bar{N}_{\kappa(\rho,\cL,\cA)}\alpha$ the largest possible gap in the
projection is bounded by \fullref{QGIT} by:
\begin{align}  \label{eq:14}
  d(\alpha_a,\alpha_b)&\leq\frac{\cL^2+1}{\cL^2}(\cA+\kappa(\rho,\cL,\cA))+\frac{\cL^2-1}{\cL^2}\kappa(\rho,\cL,\cA)+2\rho(\kappa(\rho,\cL,\cA))\\\notag
&\leq
\frac{\cL^2+1}{\cL^2}(\cA+\kappa(\rho,\cL,\cA))+\frac{\cL^2-1}{\cL^2}\kappa(\rho,\cL,\cA)+\frac{2\kappa(\rho,\cL,\cA)}{3\cL^2}\\\notag
&=\frac{\cL^2+1}{\cL^2}\cA+\frac{6\cL^2+2}{6\cL^2}\cdot 2\kappa(\rho,\cL,\cA) 
\end{align}
This is greater than $2\kappa(\rho,\cL,\cA)$.

In either case, for $c\in (a,b)$ we have:
\begin{align}  \label{eq:15a}
  d(\alpha_c,\beta)&\leq
\kappa(\rho,\cL,\cA)+\min\{d(\alpha_a,\alpha_c),\,d(\alpha_b,\alpha_c)\}\\\notag
&\leq \kappa(\rho,\cL,\cA)+\cA+\cL\frac{b-a}{2}\\\notag
&\leq \kappa(\rho,\cL,\cA)+\cA+\cL\frac{\cL\cA+\cL d(\alpha_a,\alpha_b)}{2}
\end{align}
 Substitute \eqref{eq:14} into \eqref{eq:15a} and observe that the
 resulting bound is less than $\kappa'(\rho,\cL,\cA)$, which was
 defined to be $(\cL^2+2)(\cA+2\kappa(\rho,\cL,\cA))$.
\end{proof}

\begin{lemma}\label{cor:uniformcontraction}
If $\alpha$ is a $\rho$--contracting geodesic ray and $\beta$ is a
continuous $(\cL,\cA)$--quasi-geodesic ray asymptotic to
$\alpha$ with $\alpha_0=\beta_0$ then $\beta$ is $\rho'$--contracting
where $\rho'\asymp \rho$ depends only on $\rho$, $\cL$, and $\cA$. 
\end{lemma}
\begin{proof}
\fullref{lem:hausdorff} says the Hausdorff distance between $\alpha$ and
$\beta$ is bounded in terms of $\rho$, $L$, and $A$, so the claim follows from \fullref{Hausdorff}.
\end{proof} 

The next lemma gives the key divagation estimate, which gives us lower
bounds on fellow-travelling distance.
\begin{lemma}\label{keylemma}
  Let $\alpha$ be a $\rho$-contracting geodesic ray, and let $\beta$
  be a continuous
  $(\cL,\cA)$--quasi-geodesic ray with
  $\alpha_0=\beta_0=\bp$.
Given some $R$ and $J$, suppose there exists a point $x\in\alpha$ with
$d(x,\bp)\geq R$ and $d(x,\beta)\leq J$.
Let $y$ be the last point on the subsegment of $\alpha$ between $\bp$
and $x$ such that
$d(y,\beta)=\kappa(\rho,\cL,\cA)$.
There is a constant $\theconstantformerlyknownasLambda\leq 2$ and a function $\lambda(\phi,p,q)$
defined for sublinear $\phi$,
$p\geq 1$, and $q\geq 0$ such that $\lambda$ is monotonically
increasing in $p$ and $q$ and:
\[d(x,y) \leq \theconstantformerlyknownasLambda J + \lambda(\rho,\cL,\cA)\]
Thus:
\[d(\bp,y)\geq R-\theconstantformerlyknownasLambda J-\lambda(\rho,\cL,\cA)\]
\end{lemma}
\begin{proof}
If $d(x,\beta)\leq \kappa(\rho,\cL,\cA)$ then $y=x$ and we are done.
Otherwise, let $a$ be the last time such that
$\beta_a$ is $\kappa(\rho,\cL,\cA)$--close to $\alpha$ between $\bp$ and $x$,
and let $y' \in \alpha$ be the last point of $\alpha$ with
$d(\beta_a,y')=\kappa(\rho,\cL,\cA)$.
Note $d(y,x)\leq d(y',x)$.

Now let $b$ be the first
time such that $d(\beta_b,x)=J$.
The subsegment $\beta_{[a,b]}$ stays outside
$N_{\kappa(\rho,\cL,\cA)}\alpha$.
Pick a geodesic from $\beta_b$ to $x$ and let $w$ be the first point
such that $d(w,\alpha)=\kappa(\rho,1,0)$.
Pick $z\in\pi(\beta_b)$ and $v\in\pi(w)$, and let $W:=d(\beta_b,w)$, $Y:=d(y', z)$,
$Z:=d(z,v)$, and $X:=d(v,x)$, see \fullref{fig:keylemma}.

\begin{figure}[h]
  \centering
\labellist
\tiny
\pinlabel $\geq R$ [t] at 169 0
\pinlabel $Y$ [t] at 132 46
\pinlabel $Z$ [t] at 231 46
\pinlabel $X$ [t] at 289 46
\pinlabel $\leq J$ [bl] at 280 110
\pinlabel $W$ [l] at 245 98
\pinlabel $\kappa(\rho,\cL,\cA)$ [r] at 0 72
\pinlabel $\kappa(\rho,1,0)$ [l] at 335 63
\small
\pinlabel $y'$ [t] at 60 53
\pinlabel $y$ [t] at 90 50
\pinlabel $z$ [tr] at 210 44
\pinlabel $v$ [br] at 233 70
\pinlabel $w$ [bl] at 265 70
\pinlabel $\beta_a$ [br] at 64 88
\pinlabel $\beta_b$ [br] at 208 145
\pinlabel $\bp$ [r] at 1 58
\pinlabel $x$ [l] at 340 49
\pinlabel $\alpha$ [t] at 35 59
\pinlabel $\beta$ [br] at 119 99
\endlabellist
  \includegraphics[width=.6\textwidth]{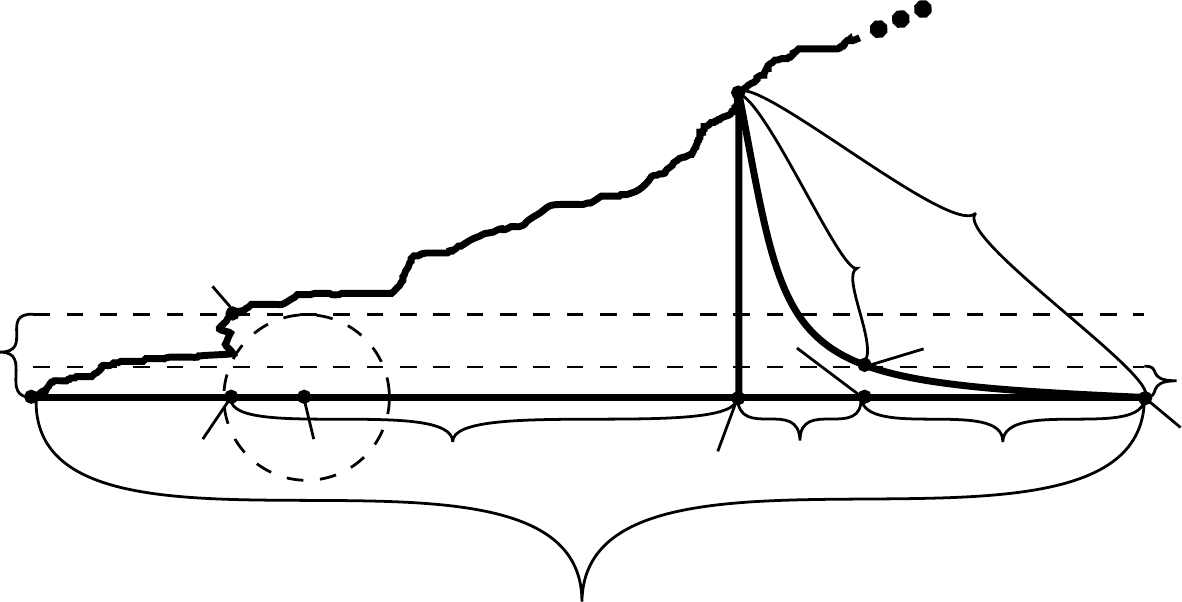}
  \caption{Setup for \fullref{keylemma}}
  \label{fig:keylemma}
\end{figure}

We have $W\geq d(\beta_b,\alpha)-Z-\kappa(\rho,1,0)$, and $X\leq
J-W+\kappa(\rho,1,0)$, so:
\begin{align*}
  d(y',x)&\leq X+Y+Z\\
&\leq Y+Z+J-d(\beta_b,\alpha)+Z+2\kappa(\rho,1,0)
\end{align*}
Apply \fullref{QGIT} to the subsegment of $\gamma$ between $\gamma_a$ to
$\gamma_b$ to bound $Y$. 
Apply \fullref{QGIT} to the subsegment of the chosen geodesic from
$\beta_b$ to $x$ between $\beta_b$ and $w$ to bound $Z$. (Note in the
latter case that we are applying \fullref{QGIT} to a geodesic, so use $\cL=1$
and $\cA=0$ for this case.)
Combining these bounds for $Y$ and $Z$ with the bound on $d(x,y')$
above yields:
\[d(y',x)\leq J+6\kappa(\rho,1,0)+\frac{\cL^2+1}{\cL^2}(\cA+\kappa(\rho,\cL,\cA))-\frac{1}{L^2}d(\beta_b,\alpha)+6\rho(d(\beta_b,\alpha))\]   
Now use the facts that $\rho(d(\beta_b,\alpha))\leq
\frac{d(\beta_b,\alpha)}{3L^2}$ and $d(\beta_b,\alpha)\leq J$ to
achieve:
\begin{align*}
  d(y,x)\leq d(y',x)&\leq
\frac{L^2+1}{L^2}J+6\kappa(\rho,1,0)+\frac{\cL^2+1}{\cL^2}(\cA+\kappa(\rho,\cL,\cA))\\
&\leq 2J+6\kappa(\rho,1,0)+2(\cA+\kappa(\rho,\cL,\cA))
\end{align*}

Set
$\theconstantformerlyknownasLambda:=2$ and $\lambda(\phi,p,q):=6\kappa(\phi,1,0)+2(q+\kappa(\phi,p,q))$.
\end{proof}

Here is an application of \fullref{keylemma} that we will use in
\fullref{sec:dynamics}.

\begin{lemma}\label{lem:nonasymptoticmakesquasigeodesic}
Given a sublinear function $\rho$ and constants $\cL\geq 1$,
$\cA\geq 0$ there exist constants $\cL'\geq 1$ and $\cA'\geq 0$
such that if $\alpha$ is a $\rho$-contracting geodesic ray or segment and $\beta$
  is a continuous
  $(\cL,\cA)$--quasi-geodesic ray not asymptotic to $\alpha$ with
  $\alpha_0=\beta_0=\bp$, then we obtain a continuous
  $(\cL',\cA')$--quasi-geodesic by following $\alpha$ backward until
  $\alpha_{s_0}$, 
  then following a geodesic from $\alpha_{s_0}$ to $\beta_{t_0}$,
  then following $\beta$, where $\beta_{t_0}$ is the last
  point of $\beta$ at distance $\kappa(\rho,\cL,\cA)$ from $\alpha$,
  and where $\alpha_{s_0}$ is the last point of $\alpha$ at distance
  $\kappa(\rho,\cL,\cA)$ from $\beta_{t_0}$.
\end{lemma}
\begin{proof}
Define $\kappa:=\kappa(\rho,\cL,\cA)$ and $\theconstantformerlyknownasLambda$ and
$\lambda:=\lambda(\rho,\cL,\cA)$ from \fullref{keylemma}. Recall
$\theconstantformerlyknownasLambda\leq 2\leq 2\cL$.
It suffices to take 
$\cA':=\left(\frac{(4\cL+1)\kappa+\lambda}{4\cL}+A\right)$
and $\cL':=4\cL$.
Since we have constructed a concatenation of three quasi-geodesic
segments, it suffices to check that points on different 
segments are not too close together.
Since $\cA'>\cA+\kappa$ we may ignore the short middle
segment.
Thus, we need to check for
$s\geq s_0$ and $t\geq t_0$ that
$d(\alpha_s,\beta_t)\geq\frac{s-s_0+t-t_0+\kappa}{\cL'}-\cA'$.

For such $s$ and $t$, let $x:=\alpha_s$, $y:=\alpha_{s_0}$, and
$z:=\beta_t$.
By \fullref{keylemma}, $s-s_0=d(x,y)\leq \theconstantformerlyknownasLambda d(x,z)+\lambda<2\cL d(x,z)+\lambda$.
Choose some point $z'\in\pi_{\alpha}(z)$.
By \fullref{cor:asymp} (\ref{eq:2}) we have $d(z,x)\geq d(z,z')\geq \frac{t-t_0}{2\cL}-2(\cA+\kappa)$.
Now average these two lower bounds for $d(x,z)$:
\begin{align*}d(\alpha_s,\beta_t)&=d(x,z)\geq \frac{1}{2}\left(
  \frac{s-s_0}{2\cL}-\frac{\lambda}{2\cL}+\frac{t-t_0}{2\cL}-2(A+\kappa)\right)\\
&\geq \frac{s-s_0+t-t_0+\kappa}{4\cL}-\left(\frac{\lambda}{4\cL}+\frac{4\cL+1}{4\cL}\kappa+A\right)\qedhere
\end{align*}
\end{proof}

To close this section we give the stronger formulation of the
Quasi-geodesic Image Theorem:
\begin{theorem}\label{QGIT2}
Given a sublinear function $\rho$ and constants $\cL\geq 1$ and
$\cA\geq 0$ there is a sublinear function $\rho'$ such that if $Z$ is
 $\rho$--contracting and $\beta\from [0,T]\to X$ is a continuous
$(\cL,\cA)$--quasi-geodesic segment with  $d(\beta,Z)=d(\beta_T,Z)=
\kappa(\rho,\cL,\cA)$ then $\diam\pi(\beta_0)\cup\pi(\beta_T)\leq\rho'(d(\beta_0,Z))$. 
\end{theorem}
\begin{proof}
  Define $\rho'(r):=\sup_\beta \diam \pi(\beta_0)\cup\pi(\beta_T)$
  where the supremum is taken over all continuous
  $(\cL,\cA)$--quasi-geodesic segments $\beta$ such that
  $d(\beta,Z)=\kappa(\rho,\cL,\cA)$ is realized at one endpoint of $\beta$ and
  the other endpoint is at distance at most $r$ from $Z$.
Suppose that $\rho'$ is not sublinear, so suppose $\limsup_{r\to\infty}\rho'(r)/r=2\epsilon>0$. 
Then there exists a sequence $(r_i)\to\infty$ such that for each $i$
there exists a continuous $(\cL,\cA)$--quasi-geodesic segment
$\beta^{(i)}\from [0,T_i]\to X$ with
$d(\beta^{(i)}_{T_i},Z)=\kappa(\rho,\cL,\cA)$ and
$d(\beta^{(i)}_0,Z)\leq r_i$ and
$\diam\pi(\beta^{(i)}_0)\cup\pi(\beta^{(i)}_{T_i})\geq \epsilon r_i$,
so that $\diam\pi(\beta^{(i)}_0)\cup\pi(\beta^{(i)}_{T_i})\geq
\epsilon d(\beta^{(i)}_0,Z)$.

For $n\in\mathbb{N}$ define $\kappa_n$ large enough so that for all
$r\geq \kappa_n$ we have $r-\cL^2\rho(r)-\cA\geq \frac{1}{3}r\geq n\cL^2\rho(r)$
(recalling (\ref{eq:1}), $\kappa_1=\kappa(\rho,\cL,\cA)$).
The proof of \fullref{QGIT} shows that if a continuous
$(\cL,\cA)$--quasi-geodesic segment stays outside the
$\kappa_n$--neighborhood of $Z$ then:
\begin{align}
\diam\pi(\beta_0)\cup\pi(\beta_T)&
\leq \frac{L^2+1}{nL^2}(A+d(\beta_T,Z))+\frac{L^2-1}{nL^2}d(\beta_0,Z)+\frac{n+1}{n}\rho(d(\beta_0,Z)) \notag \\
& \leq \frac{1}{n}\left(2A+2 d(\beta_T,Z)+d(\beta_0,Z)\right)  
  \label{eq:9}
\end{align}

For $\epsilon>0$ as above, choose $n\in\mathbb{N}$ large enough that
$n\epsilon>2$.
For all sufficiently large $i$ we have that $\diam\pi(\beta^{(i)}_0)\cup\pi(\beta^{(i)}_{T_i})>2A+2\kappa_1+\kappa_n$.
By (\ref{eq:9}) for $n=1$, we have $d(\beta^{(i)}_0,Z)>\kappa_n$. 
Let $s_i>0$ be the first time such that $d(\beta^{(i)}_{s_i},Z)=\kappa_n$.

\begin{align*}
  \epsilon d(\beta^{(i)}_0,Z)&\leq\diam\pi(\beta^{(i)}_{T_i})\cup\pi(\beta^{(i)}_0)\\
&\leq \diam\pi(\beta^{(i)}_{T_i})\cup\pi(\beta^{(i)}_{s_i})+
  \diam\pi(\beta^{(i)}_{s_i})\cup\pi(\beta^{(i)}_{0})\\
&\leq
  \left(2A+2\kappa_1+\kappa_n\right)+\frac{1}{n}\left(2A+2\kappa_n+d(\beta^{(i)}_0,Z)\right)\\
&\leq \left(2A+2\kappa_1+\kappa_n\right)+\frac{\epsilon}{2}\left(2A+2\kappa_n+d(\beta^{(i)}_0,Z)\right)\\
\end{align*}

Solving for $d(\beta^{(i)}_0,Z)$, we find that it is
bounded, independent of $i$.
By (\ref{eq:9}) for $n=1$, this would bound
$\diam\pi(\beta^{(i)}_0)\cup\pi(\beta^{(i)}_{T_i})$, independent of
$i$, whereas we have
assumed $\diam\pi(\beta^{(i)}_0)\cup\pi(\beta^{(i)}_{T_i})\geq
\epsilon r_i\to\infty$.
This is a contradiction, so we conclude $\lim_{r\to\infty} \rho'(r)/r=0$.
\end{proof}

\section{The contracting boundary and the \thistopology}\label{sec:topology}
\begin{definition}
  Let $X$ be a proper geodesic metric space with basepoint $\bp$.
Define $\cbdry X$ to be the set of contracting quasi-geodesic rays based at $\bp$
modulo Hausdorff equivalence.
\end{definition}

\begin{lemma}
For each $\zeta\in\cbdry X$:
\begin{itemize}
\item  The set of contracting geodesic rays in 
$\zeta$ is non-empty.
\item There is a sublinear function: \[\rho_\zeta(r):=\sup_{\alpha,x,y}\diam\pi_\alpha(x)\cup\pi_\alpha(y)\]
Here the supremum is taken over geodesics $\alpha\in\zeta$ and
points $x$ and $y$ such that $d(x,y)\leq d(x,\alpha)\leq r$.
\item Every geodesic in $\zeta$ is
$\rho_\zeta$--contracting.
\end{itemize}

\end{lemma}
\begin{proof}
By definition, $\zeta$ is an equivalence class of contracting
quasi-geodesic rays, so there exists some $\rho'$--contracting $(\cL,\cA)$--quasi-geodesic ray
  $\beta\in\zeta$ based at $\bp$.
Since $X$ is proper, a sequence of geodesic segments connecting $\bp$
to $\beta_i$ for $i\in\mathbb{N}$ has a subsequence that converges to a geodesic
$\alpha'$.
By \fullref{GIT},
all of these geodesic segments, hence $\alpha'$
as well, are contained in a bounded neighborhood of $\beta$, with
bound depending only on $\rho'$, so there do exist geodesics asymptotic to $\beta$.
Furthermore, \fullref{cor:asymp} implies that geodesic rays asymptotic
to $\beta$ have uniformly bounded Hausdorff distance from $\beta$,
with bound depending on $\rho'$, $\cL$, and $\cA$.
By \fullref{Hausdorff}, all such geodesics are $\rho''$--contracting for
some $\rho''\asymp\rho'$ depending on $\rho'$, $\cL$, and
$\cA$.

The function $\rho_\zeta$ is non-decreasing and bounds projection
diameters by definition. The fact that there exists a sublinear
function $\rho''$ such that all geodesics in $\zeta$ are
$\rho''$--contracting implies $\rho_\zeta\leq\rho''$, so $\rho_\zeta$
is also sublinear.
\end{proof}

\begin{definition}\label{def:nbhds}
Let $X$ be a proper geodesic metric space.
Take $\zeta\in\cbdry X$.
Fix a geodesic ray $\alpha^\zeta\in\zeta$.
 For each $r\geq 1$ define $U(\zeta,r)$ to be the set of
 points $\eta\in\cbdry X$ such that for all $\cL\geq 1$ and
 $\cA\geq 0$ and every continuous
 $(\cL,\cA)$--quasi-geodesic ray $\beta\in\eta$ we have
 $d(\beta,\alpha^\zeta\cap \nbhd^c_{r}\bp)\leq \kappa(\rho_\zeta,\cL,\cA)$.
\end{definition}

Informally, $\eta\in U(\zeta,r)$ means that inside the ball of radius $r$ about the basepoint quasi-geodesics in $\eta$ fellow-travel $\alpha^\zeta$ just as
closely as quasi-geodesics in $\zeta$ do.
Alternatively, quasi-geodesics in $\eta$ do not escape from
$\alpha^\zeta$ until after they leave the ball of radius $r$ about the basepoint.

\begin{definition}
Define the  \emph{\thistopology} on $\cbdry X$ by:
\[\fq:=\{U\subset \cbdry X\mid \forall \zeta\in U,\, \exists r\geq 1,\,
U(\zeta,r)\subset U\}\]
The contracting boundary equipped
with this topology is denoted $\cbdry^\fq X$.
\end{definition}

We do not assume that the sets $U(\zeta,r)$ are open in the topology
$\fq$.
Indeed, from the definition it is not even clear that $U(\zeta,r)$ is
a neighborhood of $\zeta$, but we will show that this is the case.

\begin{proposition}\label{prop:topology}
$\fq$ is a topology on $\cbdry X$, and
  for each $\zeta\in\cbdry X$ the collection $\{U(\zeta,n)\mid n\in \mathbb{N}\}$ is a neighborhood
  basis at $\zeta$.
\end{proposition}

\begin{corollary}
  $\cbdry^\fq X$ is first countable.
\end{corollary}

\begin{observation}\label{obs}
  Suppose $\eta\notin U(\zeta,r)$. By definition, for some $\cL$ and
  $\cA$ there exists a continuous $(\cL,\cA)$--quasi-geodesic
  $\beta\in\eta$ such that $d(\beta,\alpha^\zeta\cap \nbhd^c_{r}\bp)>
  \kappa(\rho_\zeta,\cL,\cA)$.
Since $\bp\in\beta$, this is not possible if
$\kappa(\rho_\zeta,\cL,\cA)\geq r$.
Thus, in light of \fullref{def:kappa}, the quasi-geodesic $\beta$ witnessing $\eta\notin U(\zeta,r)$
must be an $(\cL,\cA)$--quasi-geodesic with $\cL^2<r/3$ and $\cA<r/3$.
\end{observation}

The proof of \fullref{prop:topology} depends on two lemmas.
The first is a recombination result for quasi-geodesics.
Its key feature is that the quasi-geodesic constants of the result
depend only on the quasi-geodesic constants of the input, not on the
contraction function.

\begin{lemma}[Tail wagging]\label{tailwag}
Let $\rho$ be a sublinear function. Let $\cL\geq 1$ and $\cA\geq 0$. 
Let $T\geq 11\kappa'(\rho,\cL,\cA)$ and $S\geq T+6\kappa'(\rho,\cL,\cA)+6\kappa'(\rho,1,0)$.
Suppose $\alpha$ is a $\rho$--contracting geodesic ray based at $\bp$,
$\gamma$ is a continuous $(\cL,\cA)$--quasi-geodesic ray based at
$\bp$ such that $d(\gamma,\alpha_{[T,\infty)})\leq
\kappa(\rho,\cL,\cA)$, and $\beta$ is a geodesic ray based at $\bp$
such that $d(\beta,\alpha_{[S,\infty)})\leq \kappa(\rho,1,0)$.
Then there are continuous $(2\cL+1,\cA)$--quasi-geodesic rays that
agree with $\gamma$ until a point within distance
$11\kappa'(\rho,\cL,\cA)$ of $\alpha_T$ and share tails with $\alpha$
and $\beta$, respectively. 
\end{lemma}
\begin{proof}
  We construct the quasi-geodesic ray that shares a tail with
  $\beta$. The construction for the $\alpha$--tail is similar, but with
  easier estimates. 

Let $T'':=T-3\kappa'(\rho,1,0)-\kappa'(\rho,\cL,\cA)$.
Let $T'$ be the first time at which $\gamma$ comes within distance
$\kappa'(\rho,\cL,\cA)$ of $\alpha_{[T'',\infty)}$. 
Let $S'$ be such that $d(\beta_{S'},\alpha_{[S,\infty)})\leq \kappa(\rho,1,0)$. 
Let $t_0\leq T'$ and $r_0\geq S'$ be times such that
$d(\gamma_{t_0},\beta_{r_0})=d(\gamma_{[0,T']},\beta_{[S',\infty)})$,
and let $\delta$ be a geodesic from $\gamma_{t_0}$ to $\beta_{r_0}$.
There are times $b$, $c$, $b'$, and $c'$ such that
$d(\gamma_{t_0},\alpha_b),\,d(\gamma_{T'},\alpha_c)\leq\kappa'(\rho,\cL,\cA)$,
$d(\beta_{b'},\alpha_b),\,d(\beta_{c'},\alpha_c)\leq\kappa'(\rho,1,0)$.
For any  $t\leq t_0$ there exist $a$ and $a'$ such that
$d(\gamma_t,\alpha_a)\leq\kappa'(\rho,\cL,\cA)$ and
$d(\alpha_a,\beta_{a'})\leq \kappa(\rho,1,0)$.
See \fullref{fig:tailwag}.
\begin{figure}[h]
  \centering
\labellist
\pinlabel $\alpha$ [b] at 227 60
\pinlabel $\beta$ [bl] at 221 29 
\pinlabel $\gamma$ [br] at 232 46
\pinlabel $\delta$ [b] at 177 14
\small
\pinlabel $\bp$ [r] at 1 41
\pinlabel $\alpha_a$ [b] at 50 49
\pinlabel $\alpha_b$ [b] at 94 53
\pinlabel $\alpha_c$ [b] at 113 53
\pinlabel $\alpha_{T''}$ [b] at 104 53
\pinlabel $\alpha_T$ [b] at 132 53
\pinlabel $\alpha_S$ [b] at 212 57
\pinlabel $\beta_{S'}$ [tr] at 210 41
\pinlabel $\beta_{r_0}$ [l] at 224 12
\pinlabel $\beta_r$ [l] at 221 0
\pinlabel $\beta_{a'}$ [t] at 53 43
\pinlabel $\beta_{b'}$ [t] at 97 46
\pinlabel $\beta_{c'}$ [t] at 113 47
\pinlabel $\gamma_t$ [t] at 52 27
\pinlabel $\gamma_{t_0}$ [t] at 97 18
\pinlabel $\gamma_{T'}$ [tl] at 100 32 
\endlabellist
\includegraphics[width=\textwidth]{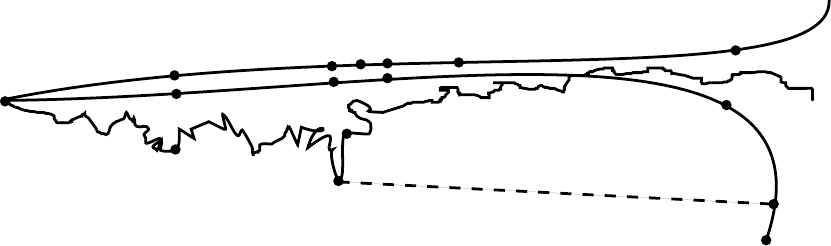}
  \caption{Wagging the tail of $\gamma$.}
  \label{fig:tailwag}
\end{figure}

The desired quasi-geodesic ray is
$\gamma_{[0,t_0]}+\delta+\beta_{[r_0,\infty)}$.

First, we verify $d(\gamma_{t_0},\alpha_T)\leq
11\kappa'(\rho,\cL,\cA)$.
The definitions of $t_0$ and $r_0$ demand $d(\gamma_{t_0},\beta_{r_0})\leq
d(\gamma_{T'},\beta_{S'})$.
The left-hand side is at least
$S'-b'-(\kappa'(\rho,\cL,\cA)+\kappa'(\rho,1,0))$, while the
right-hand side is no more than
$S'-c'+(\kappa'(\rho,\cL,\cA)+\kappa'(\rho,1,0))$, so $c'-b'\leq 2 (\kappa'(\rho,\cL,\cA)+\kappa'(\rho,1,0))$.
Since $T''\leq c\leq T''+2\kappa'(\rho,\cL,\cA)$ we have $d(\alpha_c,\alpha_T) \leq \kappa'(\rho,\cL,\cA)+3\kappa'(\rho,1,0)$.
Together, these allow us to estimate:
\begin{align*}
  d(\gamma_{t_0},\alpha_T)&\leq
                            d(\gamma_{t_0},\beta_{b'})+d(\beta_{b'},\beta_{c'})+d(\beta_{c'},\alpha_{c})+d(\alpha_c,\alpha_T)\\
&\leq
  (\kappa'(\rho,\cL,\cA)+\kappa'(\rho,1,0))+c'-b'+\kappa'(\rho,1,0) \\ & \qquad +(\kappa'(\rho,\cL,\cA)+3\kappa'(\rho,1,0))\\
&\leq 7\kappa'(\rho,1,0)+4\kappa'(\rho,\cL,\cA)\leq 11\kappa'(\rho,\cL,\cA)
\end{align*}

Next we verify the quasi-geodesic constants. 
Since we have  a concatenation of quasi-geodesics, we only need to
check that points on different pieces are not closer than
they ought to be with respect to the parameterization. 

First we claim $\gamma_{[0,t_0]}+\delta$ is an
$(\cL',\cA)$--quasi-geodesic for $\cL':=2\cL+1$.
This is true for $\gamma_{[0,t_0]}$ and $\delta$ individually.
Suppose there are $0\leq t< t_0$ and $0<u\leq|\delta|$ such that
$d(\gamma_t,\delta_u)<\frac{t_0-t+u}{\cL'}-\cA$.
Now, $d(\delta_u,\gamma_t)\geq d(\delta_u,\gamma_{t_0})=u$, which
implies $u<\frac{\cL'}{\cL'-1}(\frac{t_0-t}{\cL'}-\cA)$.
But then:
\begin{align*}
  \frac{t_0-t}{\cL}-\cA&\leq d(\gamma_t,\gamma_{t_0})\leq
                                  d(\gamma_t,\delta_u)+d(\delta_u,\gamma_{t_0})\\
&\leq \left(\frac{t_0-t+u}{\cL'}-\cA\right)+u
\end{align*}
Plugging in the value for $\cL'$ and the bound for $u$
yields a contradiction.

The same argument shows $\delta+\beta_{[r_0,\infty)}$ is a $(3,0)$--quasi-geodesic.

Now consider points $\gamma_t$ and $\beta_r$ for $t\leq t_0$
and $r\geq r_0$. 

\begin{align*}
 d(\gamma_t,\beta_r)&\geq
                      r-a'-(\kappa'(\rho,\cL,\cA)+\kappa'(\rho,1,0))\\
&=r-r_0+r_0-b'+b'-a'-(\kappa'(\rho,\cL,\cA)+\kappa'(\rho,1,0))\\
&\geq r-r_0+d(\gamma_{t_0},\beta_{r_0})+d(\gamma_t,\gamma_{t_0})-4
  (\kappa'(\rho,\cL,\cA)+\kappa'(\rho,1,0))\\
&\geq\frac{t_0-t+r-r_0+|\delta|}{2\cL+1}-\cA+|\delta|\frac{2\cL}{2\cL+1}-4 (\kappa'(\rho,\cL,\cA)+\kappa'(\rho,1,0))
\end{align*}
Thus, the ray we have constructed is a $(2\cL+1,\cA)$--quasi-geodesic,
since $|\delta|\geq 6 (\kappa'(\rho,\cL,\cA)+\kappa'(\rho,1,0))$, as
we now verify:
\begin{align*}
  |\delta|&\geq r_0-b'-(\kappa'(\rho,\cL,\cA)+\kappa'(\rho,1,0))\\
&\geq r_0-b-(\kappa'(\rho,\cL,\cA)+2\kappa'(\rho,1,0))\\
&\geq S'-T''-(\kappa'(\rho,\cL,\cA)+2\kappa'(\rho,1,0))\\
&\geq S-T''- (\kappa'(\rho,\cL,\cA)+3\kappa'(\rho,1,0))\\
&\geq 6 (\kappa'(\rho,\cL,\cA)+\kappa'(\rho,1,0))\qedhere
\end{align*}
\end{proof}

\begin{lemma}\label{lem:annulus}
  For every  sublinear function $\rho$ and $r\geq 1$ there exists a
  number $\psi(\rho,r)>r$ such that for every $R\geq \psi(\rho,r)$ and every
  $\zeta\in\cbdry X$ such that
  $\rho_\zeta\leq\rho$ we have that for every $\eta\in U(\zeta,R)$ there
  exists an $R'$ such that $U(\eta,R')\subset U(\zeta,r)$.
\end{lemma}
\begin{proof}
It suffices to take
$\psi(\rho,r):=r+\theconstantformerlyknownasLambda\kappa(\rho,2\sqrt{r/3}+1,r/3)+\lambda(\rho,\sqrt{r/3},r/3)$,
where $\theconstantformerlyknownasLambda$ and $\lambda$ are as in \fullref{keylemma}.

Suppose $R\geq \psi(\rho,r)$ and 
  $\zeta$ is a point in $\cbdry X$ such that
  $\rho_\zeta\leq\rho$.
Suppose that $\eta\in U(\zeta,R)$ with $\eta\neq\zeta$.
Let $T_0$ be the last time such that
$d(\alpha^\eta_{T_0},\alpha^\zeta)=\kappa(\rho_\zeta,1,0)$.
Set:
\[R':=T_0+2\kappa'(\rho_\zeta,2\sqrt{r/3}+1,r/3)+4\kappa(\rho_\zeta,1,0)+28\kappa'(\rho_\eta,\sqrt{r/3},r/3)+6\kappa'(\rho_\eta,1,0)\]

Suppose that there exists a point $\xi\in U(\eta,R')$ such that
$\xi\notin U(\zeta,r)$.
The latter implies there exists an $\cL\geq 1$ and $\cA\geq 0$ and a
continuous $(\cL,\cA)$--quasi-geodesic $\gamma\in\xi$ such that $d(\gamma,N_r^c\bp\cap\alpha^\zeta)>\kappa(\rho_\zeta,\cL,\cA)$.
By \fullref{obs}, we have $\cL^2,\,\cA<r/3$.
Set $\alpha:=\alpha^\eta$, $\beta:=\alpha^\xi$,
$T:=T_0+4\kappa(\rho_\zeta,1,0)+22\kappa'(\rho_\eta,\sqrt{r/3},r/3)+2\kappa'(\rho_\zeta,2\sqrt{r/3}+1,r/3)$,
and $S:=R'\geq T+6\kappa'(\rho_\eta,\cL,\cA)+6\kappa'(\rho_\eta,1,0)$.
Apply \fullref{tailwag} to $\alpha$, $\beta$, $\gamma$, $T$, and $S$ to produce a continuous
$(2\cL+1,\cA)$--quasi-geodesic $\delta\in\eta$ that agrees with
$\gamma$ at least until a point $z$ in the ball of radius
$11\kappa'(\rho_\eta,\cL,\cA)$ about $\alpha^\eta_T$.

By \fullref{cor:asymp} \eqref{eq:2} we have
$d(\alpha_T,\alpha^\zeta)\geq (T-T_0)/2-2\kappa(\rho_\zeta,1,0)$,
which implies $d(z,\alpha^\zeta)\geq \kappa'(\rho_\zeta,2\cL+1,\cA)$, so
by point $z$ the ray
$\delta$
has already escaped $\alpha^\zeta$ and can never return to its
$\kappa(\rho_\zeta,2\cL+1,\cA)$--neighborhood.
Therefore, the only points of $\delta$ in the
$\kappa(\rho_\zeta,2\cL+1,\cA)$--neighborhood of $\alpha^\zeta$ are
those that were contributed by $\gamma$. 
By construction, $\gamma$ does not come
$\kappa(\rho_\zeta,\cL,\cA)$--close to $\alpha^\zeta$ outside the ball
of radius $r$.
By applying \fullref{keylemma}, we see that $\delta$ is a witness to 
$\eta\notin U(\zeta,R)$.
This is a contradiction, so  $U(\eta,R')\subset U(\zeta,r)$.
\end{proof}

\begin{proof}[Proof of \fullref{prop:topology}]
 For every $\zeta\in\cbdry X$ and $1\leq r<r'$ we have $\zeta\in
  U(\zeta,r')\subset U(\zeta,r)$.
 The nesting is immediate from \fullref{def:nbhds}, and $\zeta\in
  U(\zeta,r)$ by \fullref{cor:asymp}.
Now it is easy to see that $\fq$ is a topology.
That a set of the form $U(\zeta,r)$ is a neighborhood of
$\zeta$ in this topology follows from 
showing that the set \[ U:=\{\eta\in U(\zeta,r) \mid \exists R_\eta,\,
U(\eta,R_\eta)\subset U(\zeta,r)\} \] is open, since then $\zeta \in U \subset U(\zeta,r)$.
Now if $\eta \in U$ then there exists $R_\eta$ so 
that $U(\eta,R_\eta)\subset U(\zeta,r)$.
\fullref{lem:annulus} says that
for all $\xi\in U(\eta,\psi(\rho_\eta,R_\eta))$ there exists $R'$ with
$U(\xi,R')\subset U(\eta,R_\eta) \subset U(\zeta,r)$. 
Therefore $U(\eta, \psi(\rho_\eta,R_\eta))\subset U$ and so $U$ is open.
\end{proof}

From this proof we observe the following consequence.
\begin{corollary}\label{cor:annulus}
  For every $\zeta\in\cbdry X$ and $r\geq 1$ there exists an open set $U$
  such that $U(\zeta,\psi(\rho_\zeta,r))\subset U\subset U(\zeta,r)$.
\end{corollary}

\begin{proposition}\label{prop:choices}
  The topology $\fq$ does not depend on the choice of basepoint or on
  the choices of
  the representative geodesic rays for each point in $\cbdry X$.
\end{proposition}
\begin{proof}
Let $\mathcal{C}$ be the set of contracting quasi-geodesic rays based
at $\bp$ and let $\mathcal{C}'$ be the set of contracting
quasi-geodesic rays based at $\bp'$. 
  There is a map $\phi\from \mathcal{C}\to\mathcal{C}'$ by prefixing
  $\gamma\in\mathcal{C}$ with a chosen geodesic segment from $\bp'$ to $\bp$.
The map $\phi$ clearly induces a bijection $\cbdry\phi$ between
contracting boundaries of $X$ with respect to different basepoints,
and the inverse map can be achieved by simply prefixing quasi-geodesic
rays by a geodesic from $\bp$ to $\bp'$. We check that $\cbdry\phi$ is
an open map.
For $\zeta\in\cbdry^\fq X$ and $r\geq 1$ we show for sufficiently
large $R$ that $U'(\cbdry\phi(\zeta),R)\subset \cbdry\phi(U(\zeta,r))$,
where $U'(\cbdry\phi(\zeta),R)$ denotes the appropriate neighborhood of $\cbdry\phi(\zeta)$
defined with $\bp'$ as basepoint.

Let $\alpha:=\alpha^\zeta$ be the reference geodesic for $\zeta$ based
at $\bp$, and let $\alpha'$ be the reference geodesic for
$\cbdry\phi(\zeta)$ based at $\bp'$. 
Then $\alpha'$ is bounded Hausdorff distance from $\alpha$. 
Suppose $\alpha$ is $\rho$--contracting and $\alpha'$ is $\rho'$--contracting.
\fullref{GIT} implies that $\alpha'$ eventually comes within distance $\kappa(\rho,1,0)$ of $\alpha$, and 
\fullref{QGIT} implies that this first happens at 
some time no later than
$d(\bp',\alpha)+3\kappa(\rho,1,0)+2\rho(d(\bp',\alpha))$.  After
that time $\alpha'$ remains in the $\kappa'(\rho,1,0)$--neighborhood of
$\alpha$. 
Assume $R>d(\bp',\alpha)+3\kappa(\rho,1,0)+2\rho(d(\bp',\alpha))$.

Assume further that $R>r+2d(\bp,\bp')$ and 
suppose
$\eta\in U'(\cbdry\phi(\zeta),R)$.
Let $\gamma\in\cbdry\phi^{-1}(\eta)$ be an arbitrary continuous
$(\cL,\cA)$--quasi-geodesic.
Our goal is to show that if $R$ is chosen sufficiently large with
respect to $\rho$, $\rho'$, and $r$, then such a $\gamma$ must come
within distance $\kappa(\rho,\cL,\cA)$
of $\alpha$ outside $N_r\bp$.
We then conclude $\cbdry\phi^{-1}(U'(\cbdry\phi(\zeta),R))\subset
U(\zeta,r)$.
By \fullref{obs}, it suffices to consider the case $\cL^2,\,\cA< r/3$.

Now, $\gamma':=\phi(\gamma)\in\eta$ is a continuous
$(\cL,\cA+2d(\bp,\bp'))$--quasi-geodesic.
Since $\eta\in U'(\cbdry\phi(\zeta),R)$ there exists a point
$x'\in\alpha'$ such that $d(\gamma',x')\leq
\kappa(\rho',\cL,\cA+2d(\bp,\bp'))$ and $d(x',\bp')\geq R$.
The first restriction on $R$ implies there is a point $x\in\alpha$
such that $d(x,x')\leq \kappa'(\rho,1,0)$, so $d(\gamma',x)\leq
\kappa'(\rho,1,0)+\kappa(\rho',\cL,\cA+2d(\bp,\bp'))$.
We also have $d(x,\bp)\geq R-\kappa'(\rho,1,0)-d(\bp,\bp')$.
Assuming further that $R>2d(\bp,\bp')+2\kappa'(\rho,1,0)+\kappa(\rho',\cL,\cA+2d(\bp,\bp'))$, we have that the point of $\gamma'$ close to $x$ is actually a point
of $\gamma$.
Let $y$ be the last point of $\alpha$ at distance
$\kappa(\rho,\cL,\cA)$ from $\gamma$ (see \fullref{fig:changeofbasepoint}), and  apply \fullref{keylemma}  to
find:
\begin{multline*}
d(\bp,y)\geq
R-\kappa'(\rho,1,0)-d(\bp,\bp')\\\quad-\theconstantformerlyknownasLambda(\kappa'(\rho,1,0)+\kappa(\rho',\sqrt{r/3},r/3+2d(\bp,\bp')))-\lambda(\rho,\sqrt{r/3},r/3)
\end{multline*}

\begin{figure}[h]
  \centering
\labellist
\small
\pinlabel $\bp$ [r] at 1 18
\pinlabel $\bp'$ [r] at 50 125
\pinlabel $\zeta$ [l] at 307 22
\pinlabel $\eta$ [b] at 226 127
\pinlabel $\gamma$ [l] at 228 86
\pinlabel $\gamma'$ [r] at 222 86
\pinlabel $\alpha$ [t] at 273 16
\pinlabel $\alpha'$ [b] at 273 25
\pinlabel $x$ [t] at 187 5
\pinlabel $x'$ [t] at 166 9
\pinlabel $y$ [t] at 137 5 
\tiny
\pinlabel $\leq \kappa(\rho',\cL,\cA+2d(\bp,\bp'))$ [l] at 225 43
\pinlabel $\leq \kappa'(\rho,1,0)$ [l] at 229 0
\pinlabel $\geq R$ [l] at 139 99
\endlabellist
  \includegraphics[width=.6\textwidth]{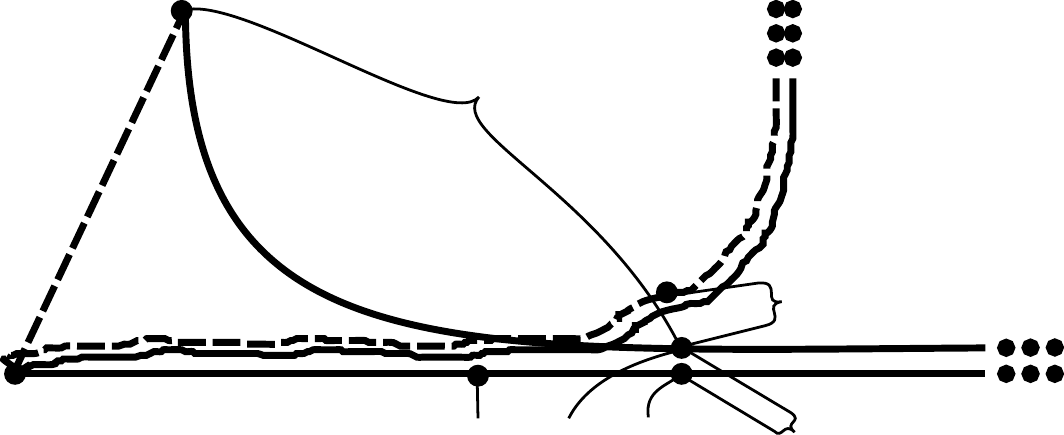}
  \caption{Change of basepoint}
  \label{fig:changeofbasepoint}
\end{figure}

Assuming that $R$ was chosen large enough to guarantee the right-hand side is at
least $r$, we have that $\gamma$ comes within distance $\kappa(\rho,\cL,\cA)$
of $\alpha$ outside $N_r\bp$. 
\end{proof}

\begin{proposition}\label{hausdorff}
  $\cbdry^\fq X$ is Hausdorff.
\end{proposition}
\begin{proof}
Let $\zeta$ and $\eta$ be distinct points in $\cbdry X$.
Let $\alpha:=\alpha^\zeta$ and $\beta:=\alpha^\eta$ be representative geodesic
rays. 
Let $R$ be large enough that the $\kappa'(\rho_\zeta,1,0)$--neighborhood of
$\alpha_{[R,\infty)}$ is disjoint from the $\kappa'(\rho_\eta,1,0)$--neighborhood of
$\beta_{[R,\infty)}$. Such an $R$ exists by \fullref{cor:asymp}.

Choose $\xi\in U(\zeta,R)$.
Let $\gamma\in\xi$ be a geodesic ray. 
Since $\xi\in U(\zeta,R)$ there exists a point $x\in\alpha$ and
$y\in\gamma$ with
$d(x,\bp)\geq R$ and $d(x,y)\leq \kappa(\rho_\zeta,1,0)$.
By construction $d(y,\beta)>\kappa'(\rho_\eta,1,0)$, so, by
\fullref{cor:asymp}, the final visit of $\gamma$  to the
$\kappa(\rho_\eta,1,0)$--neighborhood of $\beta$ must have occurred inside
the ball of radius $R$ about $\bp$. Thus, $\xi\notin U(\eta,R)$.
\end{proof}

\begin{proposition}\label{regular}
   $\cbdry^\fq X$ is regular.
\end{proposition}
\begin{proof}
Suppose $C\subset \cbdry^\fq X$ is closed and $\zeta\in C^c$. 
Then $C^c$ is a neighborhood of $\zeta$, so there exists $r'$ such
that for all $r\geq r'$ we have $U(\zeta,r)\subset C^c$. Suppose:
\begin{equation}
   \label{eq:7}
   \forall \zeta\in\cbdry^\fq X,\, \exists r'\geq 1,\, \forall r\geq r',\, \exists
   R>r,\, \closure{U(\zeta,R)}\subset U(\zeta,r)
 \end{equation}
Then there exists an $R>r$ such that
$\closure{U(\zeta,R)}\subset U(\zeta,r)\subset C^c$, so $C$ is
contained in an open set $\closure{U(\zeta,R)}^c$ that is disjoint
from $U(\zeta,R)$.
By \fullref{prop:topology}, $U(\zeta,R)$ is a neighborhood of $\zeta$,
so it contains an open set $U$ that contains $\zeta$.
The disjoint open sets $U$ and $\closure{U(\zeta,R)}^c$ separate
$\zeta$ and $C$, so (\ref{eq:7}) implies regularity.

The proof of (\ref{eq:7}) is similar to the proof of
\fullref{lem:annulus}: suppose given $r$ and $\zeta$ there is no 
$R$ satisfying the claim. Then there exists a point $\eta\in\closure{U(\zeta,R)}\cap
U(\zeta,r)^c$.
Now $\eta\in\closure{U(\zeta,R)}$
implies that for all $n\in\mathbb{N}$ there exists $\xi_n\in
U(\zeta,R)\cap U(\eta,n)$, while $\eta\notin U(\zeta,r)$ implies there
exist $\cL^2,\,\cA<r/3$ and a continuous $(\cL,\cA)$--quasi-geodesic
$\gamma\in\eta$ such that $d(\gamma,N^c_r\bp\cap\alpha^\zeta)>\kappa(\rho_\zeta,\cL,\cA)$.
For sufficiently large $n$ we wag the tail of $\gamma$ by
\fullref{tailwag} to produce a continuous
$(2\cL+1,\cA)$--quasi-geodesic $\delta\in\xi_n$ that agrees with
$\gamma$ on a long initial segment.
If $R$ is large enough this sets up a contradiction between the fact
that $\xi_n\in U(\zeta,R)$ and the fact that $\gamma$ witnesses
$\eta\notin U(\zeta,r)$, so for large enough $R$ we have
$\closure{U(\zeta,R)}\subset U(\zeta,r)$, as desired.
\end{proof}

\bigskip

Generally in this paper we will work directly with the topology on the contracting
boundary. 
However, it is worth mentioning that this object that we have called a
`boundary' really is a topological boundary.

\begin{definition}
  A \emph{bordification} of a Hausdorff topological space $X$ is a Hausdorff space
  $\hat{X}$ containing $X$ as an open, dense subset. 
\end{definition}

The contracting boundary of a proper geodesic metric space provides a
bordification of $X$ by $\hat{X}:=X\cup\cbdry X$ as follows. 
For $x\in X$ take a neighborhood basis for $x$ to be metric balls
about $x$. 
For $\zeta\in\cbdry X$ take a neighborhood basis for $\zeta$ to be
sets $\hat{U}(\zeta,r)$ consisting of $U(\zeta,r)$ and points $x\in
X$ such that we have $d(\gamma,N_r^c\bp\cap\alpha^\zeta)\leq \kappa(\rho_\zeta,\cL,\cA)$ for every $\cL\geq 1$, $\cA\geq 0$, and continuous
$(\cL,\cA)$--quasi-geodesic segment $\gamma$ with endpoints $\bp$ and
$x$.

\begin{proposition}
  $\hat{X}:=X\cup\cbdry X$ topologized as above defines a first
  countable bordification of $X$
  such that the induced topology on $\cbdry X$ is the topology of
  fellow-travelling quasi-geodesics. 
\end{proposition}
\begin{proof}
  A similar argument to that of \fullref{prop:topology} shows we have defined
  a neighborhood basis in a topology for each point in $\hat{X}$, and
  the topology agrees with the metric topology on $X$ and topology
  $\fq$ on $\cbdry X$ by construction. 
That $\hat{X}$ is Hausdorff follows from \fullref{hausdorff}.
$X$ is clearly open in $\hat{X}$. 
To see that $X$ is dense, consider
$\zeta\in\cbdry X$, which, by definition, is an equivalence class of contracting
quasi-geodesic rays. 
For any quasi-geodesic ray $\gamma\in\zeta$ we
have that $(\gamma_n)_{n\in\mathbb{N}}$ is a sequence in $X$
converging to $\gamma$, because the subsegments $\gamma_{[0,n]}$ are
uniformly contracting. 
\end{proof}

\begin{definition}\label{def:limitset}
  If $G$ is a finitely generated group acting properly discontinuously
  on a proper geodesic metric space $X$ with basepoint $\bp$ we
  define the \emph{limit set} $\Lambda(G):=\closure{G\bp}\setminus G\bp$ of $G$ to be the topological frontier in
  $\hat{X}$ of the orbit $G\bp$ of the basepoint.
\end{definition}

\section{Quasi-isometry invariance}\label{sec:invariance}
In this section we prove quasi-isometry invariance of the \thistopology.
\begin{theorem}\label{thm:continuity}
Suppose $\phi\from X\to X'$ is a quasi-isometric embedding between proper geodesic
metric spaces.
If $\phi$ takes contracting quasi-geodesics to contracting quasi-geodesics then it
induces an injection $\cbdry \phi\from
\cbdry^\fq X\to\cbdry^\fq X'$ that is an open mapping onto its image with the subspace topology.
If $\phi(X)$ is a contracting subset of $X'$ then $\cbdry \phi$ is
continuous.
\end{theorem}
We will see in \fullref{morsepreserving} that if $\phi(X)$ is
contracting then $\phi$ does indeed take contracting quasi-geodesics
to contracting quasi-geodesics, so we get the following corollary of \fullref{thm:continuity}.
\begin{corollary}\label{thm:qiinvariance}
If $\phi\from X\to X'$ is a quasi-isometric embedding between proper geodesic
metric spaces and $\phi(X)$ is contracting in $X'$ then $\cbdry\phi$
is an embedding.
In particular, if $\phi$ is a quasi-isometry then $\cbdry\phi$ is a homeomorphism.
\end{corollary}

\begin{remark}
  Cordes \cite{Cor15} proves a version of \fullref{thm:continuity} and
  \fullref{thm:qiinvariance}
  for the Morse boundary. The construction of the injective map is
  exactly the same. 
For continuity, 
he defines a map between contracting boundaries to be \emph{Morse-preserving} if for each
$\mu$ there is a $\mu'$ such that the map takes boundary points with a
$\mu$--Morse representative to boundary points with a $\mu'$--Morse representative, and shows that if
$\phi$ is a quasi-isometric embedding that induces a Morse-preserving
map $\cbdry\phi$ on the contracting boundary then $\cbdry\phi$ is continuous in the
direct limit topology. 

Similarly, let us say that $\phi$ is \emph{Morse-controlled} if for
each $\mu$ there exists $\mu'$ such that $\phi$ takes $\mu$--Morse
geodesics to $\mu'$--Morse geodesics.
A Morse-controlled quasi-isometric embedding induces a
Morse-preserving boundary map.
We will see in \fullref{morsepreserving}  that the
hypothesis that $\phi(X)$ is a contracting set
implies that $\phi$ is Morse-controlled. 
\end{remark}

Cayley graphs of a fixed group with respect to different finite
generating sets are quasi-isometric, so \fullref{thm:qiinvariance}
allows us to define the contracting boundary of a finitely generated
group, independent of a choice of generating set.
\begin{definition}
  If $G$ is a finitely generated group define $\cbdry^\fq G$ to be
  $\cbdry^\fq X$ where $X$ is any Cayley graph of $G$ with respect to
  a finite generating set.
\end{definition}

\smallskip

The hypothesis in \fullref{thm:continuity} that $\phi(X)$ is contracting already implies that
it is undistorted, so in fact we do not need to explicitly require
$\phi$ to be a quasi-isometric embedding. 
We can relax the hypotheses by only requiring $\phi$ to be coarse
Lipschitz and uniformly proper. 
This is illustrated by the following easy lemma. 

A map $\phi\from X \to X'$ between metric spaces is \emph{coarse
  Lipschitz} if there are constants $\cL\geq 1$ and $\cA\geq 0$ such
that $d(\phi(x),\phi(x'))\leq \cL d(x,x')+\cA$ for all $x,x'\in X$.
It is \emph{uniformly proper} if there exists a non-decreasing
function $\chi\from [0,\infty)\to
[0,\infty)$ such that $d(x,x')\leq \chi(d(\phi(x),\phi(x')))$ for
all $x,x'\in X$.
Note that if $X$ is geodesic and $\phi$ is coarse Lipschitz and
uniformly proper then $\chi(r)>0$ once $r>\cA$.

\begin{lemma}\label{cor:qiembedding}
If $\phi\from X \to X'$ is a coarse Lipschitz, uniformly proper map between geodesic metric spaces
and $Z\subset X$ has quasi-convex image in $X'$ then $\phi|_Z\from
Z\to X'$ is a quasi-isometric embedding.
\end{lemma}

We will prove
 a stronger statement than this in \fullref{lemma:uniformproper}.  

\begin{lemma}\label{morsepreserving}
Suppose $\phi\from X \to X'$ is a coarse Lipschitz, uniformly proper map between geodesic metric spaces
and $Z\subset X$.
If $\phi(X)$ is Morse and $Z$ is Morse then $\phi(Z)$ is Morse.
If $\phi(Z)$ is Morse then $Z$ is Morse.
Moreover, the Morse function of $Z$ determines the Morse function of
$\phi(Z)$, and vice versa, up to functions depending on $\phi$.
\end{lemma}
Before proving \fullref{morsepreserving} let us consider some examples to motivate the
hypotheses. If $X$ is a Euclidean plane, $X'$ is a line, $Z$ is a
geodesic in $X$, and $\phi$ is the composition of projection of $X$
onto $Z$ and an isometry between $Z$ and $X'$ then $\phi$ is Lipschitz
and $\phi(Z)$ is Morse, but $\phi$ is not proper and $Z$ is not
Morse. If $X$ is a line, $X'$ is a plane, $Z=X$, and $\phi$ is an
isometric embedding then $\phi$ is Lipschitz and uniformly proper and $Z$
is Morse, but $\phi(X)=\phi(Z)$ is not Morse.

In this paper we will only use the lemma in the case that $\phi(X)$ is Morse, and in this
case it is easy to prove that $Z$ is Morse when $\phi(Z)$ is. 
However, the more general statement might be of independent interest,
and requires only mild
generalizations of known results.
The proof of the first claim uses essentially the same argument as the
well-known result that quasi-convex subspaces are quasi-isometrically
embedded. The key technical point for this direction is made in 
\fullref{lemma:uniformproper} (in a more general form than needed for
\fullref{morsepreserving}, for later use).
The second claim is proved using the same strategy as used by Drutu, Mozes,
and Sapir
\cite[Lemma~3.25]{DruMozSap10}, who proved it in the case that
$X'$ is a finitely generated group, $\phi\from X\to X'$ is inclusion of
a finitely generated subgroup, and $Z$ is an infinite cyclic group.

\begin{lemma}\label{lemma:uniformproper}
  If $\phi\from X\to X'$ is a coarse Lipschitz, uniformly
  proper map between geodesic metric spaces and $Z\subset X$ has Morse
  image in
  $X'$ then for every $L\geq 1$ and $A\geq 0$ there exist $L'\geq 1$,
  $A'\geq 0$, $D'\geq 0$, and $D\geq 0$ such that for every $(L,A)$--quasi-geodesic
  $\gamma$ in
  $X'$ with endpoints on $\phi(Z)$ there is an
  $(\cL',\cA')$--quasi-geodesic $\delta$ in $X$ with endpoints in $Z$
  such that:
  \begin{itemize}
\item $\delta\subset N_{D'}(Z)$
\item $\gamma\subset N_D(\phi(\delta))$
 \item $\gamma$ and $\phi(\delta)$
  have the same endpoints. 
  \end{itemize}
\end{lemma}
The proof, briefly, is to project $\gamma$ to $\phi(Z)$ and then pull
the image back to $X$.
\begin{proof}
Suppose $\phi$ is $\chi$--uniformly proper,
$(\cL_\phi,\cA_\phi)$--coarse-Lipschitz, and $\phi(Z)$ is $\mu$--Morse. 
Suppose the domain of $\gamma$ is $[0,T]$.
For $z\in \{0,T\}$ choose $\delta_z\in Z$ such that $\phi(\delta_z)=\gamma_z$.
For $z\in \mathbb{Z}\cap (0,T)$ choose $\delta_z\in Z$ such that $d(\phi(\delta_z),\gamma_z)\leq \mu(L,A)$.
Complete $\delta$ to a map on $[0,T]$ by connecting the dots by
geodesic interpolation in $X$. 
For $D:=L/2+A+\mu(L,A)$
we have $\gamma\subset \bar{N}_D(\phi(\delta))$.
 Since the reparameterized geodesic
segments used to build $\delta$ have endpoints on $Z$ and length at
most $\chi(\cL+\cA +2\mu(\cL,\cA))$, by choosing
$D':=\chi(\cL+\cA+2\mu(\cL,\cA))/2$ we have 
$\delta\subset\bar{N}_{D'}(Z)$, and, furthermore, $\delta$ is $\chi(\cL+\cA +2\mu(\cL,\cA))$--Lipschitz.
For any $a \in [0,T]$ we have $d(\phi(\delta_a),\gamma_a) \leq \cL_\phi D'+\cA_\phi +D$.
Finally, for $a,b \in [0,T]$:
\begin{align*}
  \cL_\phi d(\delta_a,\delta_b)+\cA_\phi&\geq
                                         d(\phi(\delta_a),\phi(\delta_b))\\
&\geq d(\gamma_a,\gamma_b)-2(\cL_\phi D'+\cA_\phi +D)\\
&\geq \frac{|b-a|}{\cL}-\cA-2(\cL_\phi D'+\cA_\phi +D)
\end{align*}
Thus, $\delta$ is an $(\cL',\cA')$--quasi-geodesic for
$\cL':=\max\{\cL_\phi\cL, \chi(\cL+\cA +2\mu(\cL,\cA))\}$ and $\cA':=(3\cA_\phi+\cA+2\cL_\phi D'+2D)/\cL_\phi$.
\end{proof}
\fullref{cor:qiembedding} follows by the same argument applied to a
geodesic.
\begin{proof}[Proof of \fullref{morsepreserving}]
Suppose $\phi(X)$ and $Z$ are Morse. 
A quasi-geodesic $\gamma$ in $X'$ with endpoints on $\phi(Z)$ has
endpoints on $\phi(X)$, which is Morse.
Apply \fullref{lemma:uniformproper} to get a quasi-geodesic $\delta$
in $X$ such that $\phi(\delta)$ is coarsely equivalent to $\gamma$. 
We can, and do, choose $\delta$ so that it has endpoints on $Z$.
Since
$Z$ is Morse,  $\delta$ stays close to $Z$, so $\phi(\delta)$ stays
close to $\phi(Z)$, so $\gamma$ is close to $\phi(Z)$. Thus, $\phi(Z)$
is Morse.

Now suppose $\phi(Z)$ is Morse.
By \fullref{cor:qiembedding}, $\phi$ restricted to $Z$ is a
quasi-isometric embedding. 
Suppose
that it is an $(\cL,\cA)$--quasi-isometric embedding. (These constants
are at least the coarse Lipschitz constants, so we will also assume
$\phi$ is $(\cL,\cA)$--coarse Lipschitz on all of $X$.)
Suppose $\phi(Z)$ is $\mu$--Morse.
The Morse property implies that $\phi(Z)$ is $(2\mu(1,0)+1)$--coarsely
connected, so $Z$ is $E$--coarsely connected for $E:=\cL(2\mu(1,0)+1+\cA)$.
Let $t:=\frac{1}{6\cL^2}$. 
If $Z$ has diameter at most $\frac{1+2t}{1-2t}E$ then it is
$\mu'$--Morse for $\mu'$ the function $\frac{1+2t}{1-2t}E$,
which depends only on $\cL$, $\cA$, and $\mu$.
In this case we are done. Otherwise we prove that $Z$ is
$t$--recurrent and apply \fullref{morseequivalent}.

We fix $C\geq 1$ and produce the corresponding $D$ from the definition
of recurrence.

Since the diameter of $Z$ is bigger than $\frac{1+2t}{1-2t}E$, the fact
that $Z$ is $E$--coarsely connected implies that for every $a,\, b\in Z$
there exists a point $c\in Z$ such that $td(a,b)+E\geq
\min\{d(a,c),\,d(b,c)\}\geq td(a,b)$:
if $d(a,b)\geq\frac{1}{1-2t}E$ then $c$ may be found on a coarse path from $a$ to $b$,
otherwise $c$ may be found on a coarse path joining $a$ to one of two points separated by more than $\frac{1+2t}{1-2t}E$.
Such a point $c$ is within distance $td(a,b)+E$ of every path with
endpoints $a$ and $b$, so for any fixed $K\geq
0$ we may restrict our attention to the case $d(a,b)>K$ by assuming $D$ is at least $tK+E$.

Suppose $p$ is path in $X$
with endpoints $a$ and $b$ on $Z$ such that $p$ has length $|p|$ at most
$Cd(a,b)$ and $d(a,b)> 8\cL(\cA+1)$.
Subdivide $p$ into $\lceil|p|\rceil$ many subsegments, all but
possibly the last of
which has length 1. 
Denote the endpoints of these subsegments
$a=x_0,x_1,\dots, x_{\lceil |p|\rceil}=b$.
Let $q$ be a path in $X'$ obtained by connecting each $\phi(x_i)$ to
$\phi(x_{i+1})$ by a geodesic. 
Then $q$ is a path of length at most $(\cL+\cA)\lceil |p|\rceil \leq (\cL+\cA)\frac{9}{8}|p|$ that coincides
with $\phi(p)$ on $\phi(\{x_0,\dots,x_{\lceil |p|\rceil}\})$. 
Since $\phi$ is an $(\cL,\cA)$--quasi-geodesic embedding of $Z$ we have
that the distance between the endpoints of $q$ is at least $d(a,b)/\cL
-\cA \geq \frac{7}{8L}d(a,b)$, so that $|q|<C'd(\phi(a),\phi(b))$ for $C'=9C\cL(\cL+\cA)/7$.
Since $\phi(Z)$ is Morse it is recurrent, so given $t':=1/3$ and $C'$ as above there is a $D'\geq
0$ and $z\in Z$ such that $\min\{d(\phi(z),\phi(a)),d(\phi(z),\phi(b))\}\geq d(\phi(a),\phi(b))/3$ and $d(\phi(z),q)\leq D'$.
Thus, there is some $i$ such that $d(\phi(x_i),\phi(z))\leq D'+(\cL+\cA)/2$. 
If $\phi$ is $\chi$--uniformly proper then $d(x_i,z)\leq
D:=\max\{\chi(D'+(\cL+\cA)/2),\, 8t\cL(\cA+1)+E\}$. 
It remains to check that $z$ is sufficiently far from the endpoints of
$p$. 
This follows easily from our choice of $t$, the distance bound between $\phi(z)$ and
the endpoints of $q$, and the assumption $d(a,b)>8\cL\cA$, by using the fact that $\phi|_Z$ is an
$(\cL,\cA)$--quasi-isometric embedding.
\end{proof}

\begin{corollary}\label{morsetransitive}
If $G$ is a finitely generated group and $Z$ is a subset of a finitely
generated subgroup $H$ of $G$ such that $Z$ is Morse in $G$ then $Z$ is Morse
in $H$.
If $G$ is a finitely generated group and $H$ is a Morse subgroup of
$G$ then every Morse subset $Z$ of $H$ is also Morse in $G$.
\end{corollary}

\begin{proof}[{Proof of \fullref{thm:continuity}}]
Since the topology is basepoint invariant we choose $\bp\in X$ and let
$\bp':=\phi(\bp)\in X'$. 

Suppose $\phi$ is an $(\cL,\cA)$--quasi-isometric embedding, and
suppose $\bar{\phi}\from \phi(X)\to X$ is an $(\cL,\cA)$--quasi-isometry inverse to $\phi$.
We assume $\sup_{x\in X'} d(\phi\circ\bar{\phi}(x),x)\leq \cA$.

The quasi-isometric embedding $\phi$ induces an injective map between equivalence classes of
quasi-geodesic rays based at $\bp$ and equivalence classes of
quasi-geodesic rays based at $\bp'$.
The hypothesis that $\phi$ sends contracting quasi-geodesics to
contracting quasi-geodesics implies that it takes equivalence classes of contracting quasi-geodesic
rays to equivalence classes of contracting quasi-geodesic rays, so $\phi$ induces an injection $\cbdry\phi\from \cbdry
X\to\cbdry X'$.

\paragraph{\itshape Continuity:}
Assume $\phi(X)$ is $\rho$--contracting.
By \fullref{morsepreserving}, $\phi$ sends contracting quasi-geodesic
rays to contracting quasi-geodesic rays, so we have an injective map
$\cbdry \phi$ as above.
We claim that:
\begin{equation}
  \label{eq:15}
  \forall \zeta\in\cbdry\phi(\cbdry X),\,\forall r>1,\,\exists
  R'>1,\,\forall R\geq R',\, U((\cbdry\phi)^{-1}(\zeta),R)\subset 
(\cbdry\phi )^{-1}(U(\zeta,r))
\end{equation}

Given the claim, 
let $U'$ be an open set in $\cbdry^\fq X'$.
For each $\zeta\in U'\cap \cbdry\phi(\cbdry X)$ there exists an
$r_\zeta$ such that $U(\zeta,r_\zeta)\subset U'$.
Apply \eqref{eq:15} to get an $R_\zeta$, and choose an open
neighborhood  of $(\cbdry\phi)^{-1}(\zeta)$ contained in
$U((\cbdry\phi)^{-1}(\zeta),R_\zeta)$.
Let $U$ be the union of
these open sets for all $\zeta\in U'\cap \cbdry\phi(\cbdry X)$.
Then $U$ is an open set and \eqref{eq:15} implies $U=(\cbdry\phi)^{-1}(U')$.

To prove the claim we play our usual game of supposing the converse,
deriving a bound on $R$, and then choosing $R$ to be larger than that
bound.
The key point is that all of the constants involved are bounded in
terms of $\zeta$, $(\cbdry\phi)^{-1}(\zeta)$, $r$, and $\rho$.

Suppose for given $\zeta\in\cbdry\phi(\cbdry X)$ and $r>1$ there exists an $R>1$ and a point  $\eta\in
U((\cbdry\phi)^{-1}(\zeta),R)$ such that
$\eta\notin(\cbdry\phi)^{-1}(U(\zeta,r))$.
The latter implies there exists a continuous
$(\cL',\cA')$--quasi-geodesic $\gamma\in\cbdry\phi(\eta)$ witnessing
$\cbdry\phi(\eta)\notin U(\zeta,r)$.
By \fullref{obs}, the quasi-geodesic constants of $\gamma$ are bounded
in terms of $r$. 
We must adjust $\gamma$ to get it into the domain of $\bar{\phi}$.
Since $\cbdry\phi(\eta)$ is in the image of $\cbdry\phi$, the quasi-geodesic $\gamma$ is asymptotic to a quasi-geodesic
contained in $\phi(X)$, so $\gamma$ is contained in a bounded
neighborhood of $\phi(X)$.
Since $\phi(X)$ as a whole is $\rho$--contracting, we can replace $\gamma$ by a
projection $\gamma'$ of $\gamma$ to $\phi(X)$ as \fullref{lemma:uniformproper}.
The Hausdorff distance between $\gamma$ and $\gamma'$ is
bounded\footnote{If we had only assumed $\phi$ to be Morse-controlled
  this bound would depend on the Morse/contraction function of $\eta$,
which can be arbitrarily bad, even for $\eta$ in a small  neighborhood of $\zeta$.} in
terms of $\rho$ and the quasi-geodesic constants of $\gamma$, hence by
$r$, and the additive quasi-geodesic constant of $\gamma'$ increases
by at most twice the Hausdorff distance.

Tame $\bar{\phi}(\gamma')$ to get a continuous quasi-geodesic
$\hat{\gamma}\in\eta$.
The Hausdorff distance between them and the quasi-geodesic constants $(\cL'',\cA'')$
of $\hat{\gamma}$ are bounded in terms of the quasi-isometry constants
of $\gamma'$ and $\phi$.

\begin{figure}[h]
  \centering
\labellist
\small
\pinlabel $\zeta$ [l] at 440 80
\pinlabel $\alpha':=\alpha^\zeta$ [l] at 425 101
\pinlabel $\eta$ [b] at 134 173
\pinlabel $\cbdry\phi(\eta)$ [b] at 380 171
\pinlabel $(\cbdry\phi)^{-1}(\zeta)$ [l] at 185 83
\pinlabel $x$ [tr] at 124 56
\pinlabel $\bp$ [r] at 3 11
\pinlabel $\bp'$ [r] at 242 11
\pinlabel $\alpha:=\alpha^{(\cbdry\phi)^{-1}(\zeta)}$ [l] at 164 115
\pinlabel $\bar{\phi}(\gamma')$ [r] at 122 118
\pinlabel $\hat{\gamma}$ [l] at 130 118
\pinlabel $\gamma$ [l] at 381 118
\pinlabel $\phi(x)$ [l] at 421 20
\pinlabel $\hat{\alpha}$ [t] at 301 0
\pinlabel $\phi$ [t] at 210 47
\pinlabel $\phi(\alpha)$ [t] at 330 1
\pinlabel $x''$ [b] at 386 71
\tiny
\pinlabel $\leq J_1$ [t] at 379 20 
\pinlabel $\leq J_2$ [tl] at 425 50
\pinlabel $\geq R$ [t] at 79 5
\pinlabel $\leq J_0$ [t] at 170 14
\endlabellist
  \includegraphics[width=.8\textwidth]{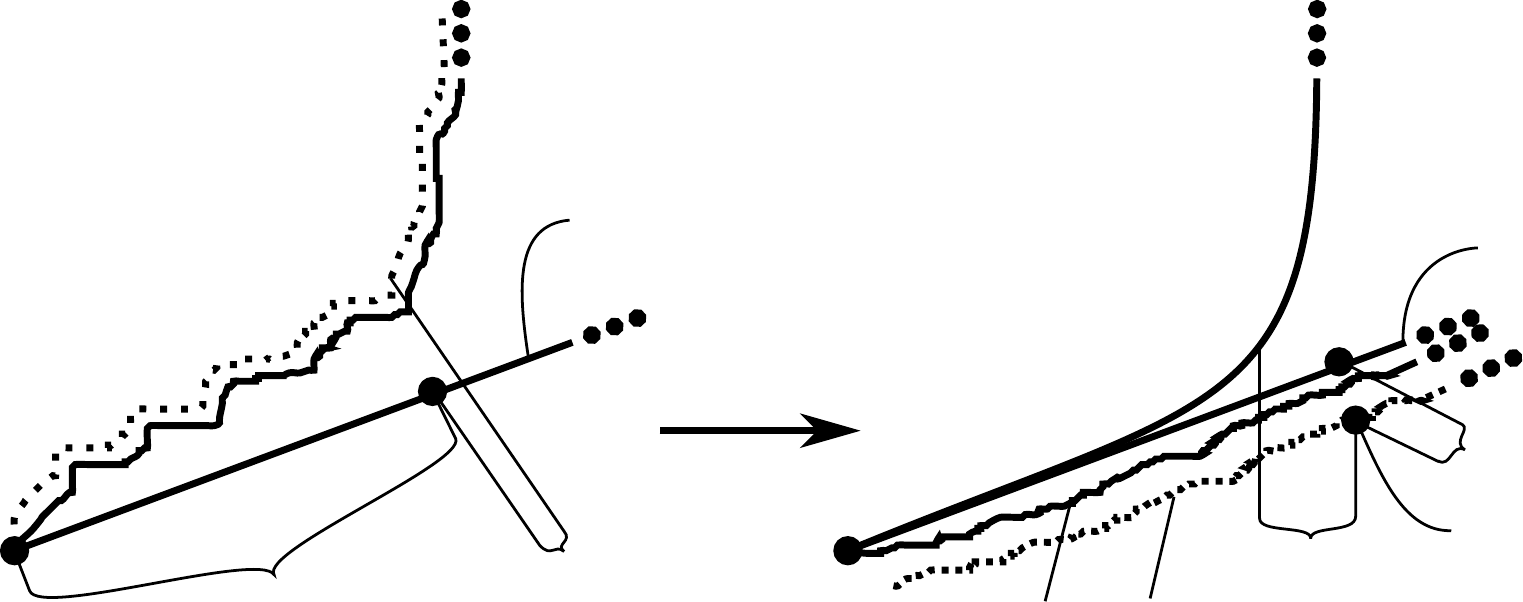}
  \caption{Setup for \fullref{thm:continuity}}
  \label{fig:qiinvariance}
\end{figure}
Let $\alpha:=\alpha^{(\cbdry\phi)^{-1}(\zeta)}$.
Since $\hat{\gamma}\in\eta\in
U((\cbdry\phi)^{-1}(\zeta),R)$, 
there exists $x\in\alpha$ such that $d(\bp,x)\geq R$ and
$d(x,\hat{\gamma})\leq \kappa(\rho_{(\cbdry \phi)^{-1}(\zeta)},
\cL'',\cA'')$.
By the argument of the previous paragraph, $\kappa(\rho_{(\cbdry \phi)^{-1}(\zeta)},
\cL'',\cA'')$ can be bounded in terms of $\cL$, $\cA$, $r$, $\rho$,
$\rho_\zeta$, and $\rho_{(\cbdry \phi)^{-1}(\zeta)}$.
This bound plus the Hausdorff distance to $\bar{\phi}(\gamma')$ give a bound $d(\bar{\phi}(\gamma'),x)\leq J_0$.
Push forward by $\phi$ to get 
$d(\gamma,\phi(x))\leq J_1:=\cL J_0+2\cA+d_{\mathrm{Haus}}(\gamma,\gamma')$.
We also know $d(\phi(x),\bp')\geq R/\cL - \cA$.

The quasi-isometric embedding $\phi$ sends the geodesic $\alpha$ to an
$(\cL,\cA)$--quasi-geodesic $\phi(\alpha)$  asymptotic to $\alpha':=\alpha^\zeta$
with $\phi(\alpha)_0=\phi(\bp)=\bp'$.
Tame $\phi(\alpha)$ to produce a continuous
$(\cL,2\cL+2\cA)$--quasi-geodesic $\hat{\alpha}$ at Hausdorff distance at
most $\cL+\cA$ from $\phi(\alpha)$.
Since $\hat{\alpha}\in \zeta$ we have that $\hat{\alpha}$ is contained in
the $\kappa'(\rho_\zeta,\cL,2\cL+2\cA)$--neighborhood of
$\alpha'$, so $\phi(\alpha)$ is contained in the
$J_2$--neighborhood
of $\alpha'$ for $J_2:=\kappa'(\rho_\zeta,\cL,2\cL+2\cA)+\cL+\cA$.
In particular, $d(\phi(x),\alpha')\leq J_2$.
Let $x''$ be the closest point of $\alpha'$ to $\phi(x)$, so that $d(\gamma,x'')\leq J_1+J_2$ and $d(\bp',x'')\geq R/\cL-\cA-J_2$.
By \fullref{keylemma}, since $\gamma$ is an
$(\cL',\cA')$--quasi-geodesic, if $y$ is the last point of
$\alpha'$ such that $d(\gamma,y)=\kappa(\rho_\zeta,\cL',\cA')$ then:
\begin{equation}\label{eq:8}
d(\bp',y)\geq R/\cL-\cA-J_2-\theconstantformerlyknownasLambda(J_1+J_2)-\lambda(\rho_\zeta,\cL',\cA')
\end{equation}

Everything except $R$ in (\ref{eq:8}) can be bounded in
terms of  $\cL$, $\cA$, $r$, $\rho$, $\rho_\zeta$, and $\rho_{(\cbdry \phi)^{-1}(\zeta)}$, so, given $r$ and $\zeta$ we can choose $R$ large enough to guarantee
$d(\bp',y)>r$.
For such an $R$, we have $\cbdry\phi(\eta)\in U(\zeta,r)$ for every
$\eta\in U((\cbdry\phi)^{-1}(\zeta),R)$.
This finishes the proof of claim \eqref{eq:15}, so we conclude
$\cbdry\phi$ is continuous if $\phi(X)$ is contracting.

\paragraph{\itshape Open mapping:}
The image of $\bar{\phi}$ is coarsely dense in $X$, so it is
contracting. 
  Thus, we can apply the argument of the proof of continuity above to
  $\bar{\phi}$ to get the following analogue of \eqref{eq:15}, noting
  that $\cbdry\bar{\phi}=(\cbdry\phi)^{-1}$:
\begin{equation}
  \label{eq:16}
  \forall \zeta\in\cbdry X,\,\forall r>1,\,\exists R'>1,\,\forall
R\geq
R',\,\cbdry\phi(U(\zeta,r))\supset U(\cbdry\phi(\zeta),R)\cap\cbdry\phi(\cbdry X)
\end{equation}

Let  $U$ be an open set in $\cbdry^\fq X$.
For every $\zeta\in U$ there exists $r_\zeta$ such that
$U(\zeta,r_\zeta)\subset U$. 
Apply \eqref{eq:16} to get $R_\zeta$, and let $U_\zeta$ be an open
neighborhood of $\cbdry\phi(\zeta)$ contained in $U(\cbdry\phi(\zeta),R_\zeta)$.
Then $U':=\bigcup_{\zeta\in
  U}U_\zeta$ is an open set in $\cbdry^\fq X'$ containing
$\cbdry\phi(U)$.
The choices of the $R_\zeta$, by \eqref{eq:16}, imply
that $U'\cap\cbdry\phi(\cbdry^\fq X)=\cbdry\phi(U)$.
\end{proof}

\bigskip

One reason it may be convenient to weaken the stated quasi-isometric
embedding hypothesis is that the orbit map of a properly discontinuous group action of a finitely
generated group on a proper geodesic metric space is always coarse Lipschitz and uniformly proper, so we get the following consequences of \fullref{thm:qiinvariance}.

\begin{proposition}\label{prop:limset}
  Suppose $G$ acts properly discontinuously on a proper geodesic metric
  space $X$.
Suppose the orbit map $\phi\from g\mapsto g\bp$ takes contracting
quasi-geodesics to contracting quasi-geodesics and has quasi-convex
image. Then:
  \begin{itemize}
  \item  $G$ is  finitely generated.
\item The orbit map  $\phi\from g\mapsto g\bp$ is a quasi-isometric
  embedding.
\item The orbit map
  induces an injection $\cbdry\phi\from \cbdry^\fq G\to\cbdry^\fq X$
  that is an open mapping onto its image, which is $\limit(G)$ (recall
  \fullref{def:limitset}).
In particular, if $\limit(G)$ is compact then so is $\cbdry^\fq G$.
  \end{itemize}
If $\phi(G)$ is contracting in $X$ then the above are true and
$\cbdry\phi$ is an embedding, so $\cbdry^\fq G$ is homeomorphic to $\limit(G)$.
\end{proposition}
\begin{proof}
Suppose $G\bp$ is $Q$--quasi-convex.
A standard argument shows that $G$ is finitely generated and $\phi$ is
a quasi-isometric embedding. 

Since $\phi$ takes contracting quasi-geodesics to contracting
quasi-geodesics it induces an open injection $\cbdry\phi$ onto its image by
\fullref{thm:continuity}.

A point in $\zeta\in\cbdry^\fq G$ is sent to the equivalence class
of the contracting quasi-geodesic ray $\phi(\alpha^\zeta)$. 
The sequence $(\phi(\alpha^\zeta_n))_{n\in\mathbb{N}}$
converges to $\cbdry\phi(\zeta)$ in $\hat{X}$, so the image of $\cbdry
\phi$ is contained in $\limit(G)$. 

Conversely, suppose $\zeta\in\limit(G)$. 
Then there is a sequence
$(g_n\bp)_{n\in\mathbb{N}}$ converging in $\hat{X}$ to
$\zeta\in\cbdry^\fq X$. 
By passing to a subsequence, we may assume $g_n\bp\in
\hat{U}(\zeta,n)$ for all $n$.
The definition of $\hat{U}(\zeta,n)$ implies that for any chosen
geodesic $\gamma^n$ from $\bp$ to $g_n\bp$ there are points $x_n$
on $\alpha^\zeta$ and $y_n$ on $\gamma^n$ such that $d(x_n,y_n)\leq\kappa:=
\kappa(\rho_\zeta,1,0)$ and $d(\bp,x_n)\geq n$. 
Since $G\bp$ is quasi-convex, there exists $g'_n\in G$ such that
$d(y_n,g'_n\bp)\leq Q$. 
Thus, the set $G\bp\cap
N_{\kappa+Q}\alpha^\zeta$ is unbounded. 
If $g\bp\in N_{\kappa+Q}\alpha^\zeta$ then an application of
\fullref{keylemma} implies that a geodesic from $\bp$ to $g\bp$ has
an initial segment that is $\kappa'(\rho_\zeta,1,0)$--Hausdorff
equivalent to an initial segment of $\alpha^\zeta$, and the length of
these initial segments is $d(\bp,g\bp)$ minus a constant depending on
$\zeta$ and $Q$, but not $g$. 
Since we can take $d(\bp,g\bp)$ arbitrarily large, and since every
geodesic from $\bp$ to $g\bp$ is contained in the $Q$--neighborhood
of $G\bp$, we conclude that $\alpha^\zeta$ is 
contained in a bounded neighborhood of $G\bp$. 
Now project $\alpha^\zeta$ to $G\bp$ and pull back to
$G$ to get a contracting quasi-geodesic ray whose $\phi$--image is
asymptotic to $\alpha^\zeta$, which shows
$\zeta\in\cbdry\phi(\cbdry^\fq G)$.

If $\phi(G)$ is contracting then $\phi$ does indeed take contracting
quasi-geodesics to contracting quasi-geodesics, by
\fullref{morsepreserving}, and have quasi-convex image, so the
previous claims are true and \fullref{thm:qiinvariance} says
$\cbdry\phi$ is an embedding.
\end{proof}

\begin{corollary}\label{cor:contractingsubgroup}
  If $H$ is a subgroup of a finitely generated group
  $G$ and $H$ is contracting in $G$ then $H$ is finitely generated and the inclusion $\iota\from
  H\to G$ induces an embedding $\cbdry\iota \from\cbdry^\fq H\to\cbdry^\fq G$.
\end{corollary}
Properly speaking, we ought to require that $H$ is a contracting subset of
the Cayley graph of $G$ with respect to some specified generating set,
but it follows from \fullref{morseequivalent} that the property of
being a contracting subset does not depend on the choice of metric
within a quasi-isometry class. 
\begin{corollary}\label{cor:he}
  If $H$ is a  hyperbolically embedded subgroup (in the
  sense of \cite{DahGuiOsi17}) in a
  finitely generated group $G$ then $\cbdry\iota \from\cbdry^\fq
  H\to\cbdry^\fq G$ is an embedding.
A special case is that of a peripheral subgroup of a relatively
hyperbolic group.
\end{corollary}
\begin{proof}
  Sisto \cite{Sis16} shows hyperbolically embedded
  subgroups are Morse, hence contracting.
Peripheral subgroups of relatively hyperbolic groups are a motivating
example for the definition of hyperbolically embedded subgroups in
\cite{DahGuiOsi17}, but in this special case the fact they are Morse
was already shown by Drutu and Sapir \cite{DruSap05}.
\end{proof}

Together with \fullref{thm:qiinvariance}, \fullref{morsetransitive} implies:
\begin{corollary}
If $G$ is a finitely generated group and $Z$ is a Morse subset of $G$
then $\cbdry^\fq Z$ embeds into $\cbdry^\fq H$ for every finitely generated subgroup $H$ of $G$ containing $Z$.
In particular, if $\cbdry^\fq Z$ is non-empty then so is $\cbdry^\fq
H$, and if $\cbdry^\fq Z$ contains a non-trivial connected component
then so does $\cbdry^\fq H$.
\end{corollary}

\section{Comparison to other topologies}\label{sec:comparison}
\begin{definition}\label{def:Vnbhds}
Let $X$ be a proper geodesic metric space.
Take $\zeta\in\cbdry X$.
Fix a geodesic ray $\alpha\in\zeta$.
 For each $r\geq 1$ define $V(\zeta,r)$ to be the set of
 points $\eta\in\cbdry X$ such that for every geodesic ray $\beta\in\eta$ we have
 $d(\beta,\alpha\cap \nbhd^c_{r}\bp)\leq \kappa(\rho_\zeta,1,0)$.
\end{definition}
The same argument as \fullref{prop:topology} shows that
$\{V(\zeta,n)\mid n\in\mathbb{N}\}$ gives a neighborhood basis at $\zeta$ for a topology
$\fg$ on
$\cbdry X$. 
We call $\fg$ the \emph{topology of fellow-travelling geodesics}.
It is immediate from the definitions that  $\fq$ is a refinement of 
$\fg$. 
The topology $\fg$ need not be preserved by quasi-isometries of $X$
\cite{Cas16gromovboundary}. 
It is an open question whether $\fg$ is
preserved by quasi-isometries when $X$ is the Cayley graph of a
finitely generated group.

One might also try  to take $V'(\zeta,r)$ to be the set of
 points $\eta\in\cbdry X$ such that for \emph{some} geodesic ray $\beta\in\eta$ we have
 $d(\beta,\alpha\cap \nbhd^c_{r}\bp)\leq \kappa(\rho_\zeta,1,0)$.
Let $\fg'$ denote the resulting topology.
Beware that in general $\{V'(\zeta,r)\mid r\geq 1\}$ is only a filter base
converging to $\zeta$, not necessarily a neighborhood base of $\zeta$
in $\fg'$; the sets $V'(\zeta,r)$ might not be neighborhoods of $\zeta$. 

\begin{lemma}\label{lem:uniformlycontractingmeansallthetopologiesarethesame}
 Let  $X$ be a proper geodesic metric space.
Let $\bdry_\rho X=\{\zeta\in\cbdry X\mid \rho_\zeta\leq\rho\}$, i.e.\ $\zeta \in \bdry_\rho X$ if all geodesics $\alpha \in \zeta$ are $\rho$--contracting.
The topologies on $\bdry_\rho X$ generated by taking, for each
$\zeta\in\bdry_\rho X$ and $r\geq 1$, the sets $U(\zeta,r)\cap \bdry_\rho
X$, $V(\zeta,r)\cap \bdry_\rho X$, or $V'(\zeta,r)\cap \bdry_\rho X$,
are equivalent. 
\end{lemma}
\begin{proof}
  For each $\zeta$ and $r$ we have $U(\zeta,r)\cap \bdry_\rho
X\subset V(\zeta,r)\cap \bdry_\rho X\subset V'(\zeta,r)\cap \bdry_\rho
X$ by definition.

Given that points in $V'(\zeta,r)\cap \bdry_\rho
X$ and $U(\zeta,r)\cap \bdry_\rho
X$ are uniformly contracting, a straightforward application of
\fullref{keylemma} shows that for all $\zeta$ and $r$, for all
sufficiently large $R$ we have $V'(\zeta,R)\cap \bdry_\rho
X\subset U(\zeta,r)\cap \bdry_\rho
X$.
Also since points of $V'(\zeta,R)\cap \bdry_\rho
X$ are uniformly contracting, these do, in fact, give a neighborhood
basis at $\zeta$ for the induced topology, as in \fullref{prop:topology}.
\end{proof}

\begin{proposition}\label{lem:uniformlyasymptotic}
 Let $X$ be a proper geodesic metric space.
If $X$ is hyperbolic then $\cbdry^{\fq}X\cong \cbdry^{\fg}X\cong
\cbdry^{\fg'}X$, and these are homeomorphic to the Gromov boundary.
If $X$ is  CAT(0) then $\cbdry^{\fg}X\cong \cbdry^{\fg'}X$, and these
are homeomorphic to the subset of the visual boundary of $X$
consisting of endpoints of contracting geodesic rays, topologized as a
subspace of the visual boundary.
\end{proposition}
\begin{proof}
 For a description of a neighborhood basis for points in the Gromov
  or visual boundary see \cite[III.H.3.6]{BriHae99} and
  \cite[II.8.6]{BriHae99}, respectively.
Note that these are equivalent to the neighborhood bases for $\fg'$.

The claim for hyperbolic spaces follows from
\fullref{lem:uniformlycontractingmeansallthetopologiesarethesame},
because geodesics in a hyperbolic space are uniformly contracting.

If $X$ is CAT(0) then $\cbdry^{\fg}X\cong\cbdry^{\fg'}X$ because there
is a unique geodesic ray in each asymptotic equivalence class.
\end{proof}

More generally, $\cbdry^{\fg}X\cong\cbdry^{\fg'}X$ if $X$ is a proper
geodesic metric space with the property that every geodesic ray in $X$
is either not contracting or has contraction function bounded by a
constant. 
This follows by the same argument as in \cite{Cas16gromovboundary}.

\bigskip
Next, we recall the \emph{direct limit topology}, $\dl$, on $\cbdry X$
of Charney
and Sultan \cite{ChaSul15} and Cordes \cite{Cor15}. 

For a given contraction function $\rho$ consider the set $\bdry_\rho
X$ of points $\zeta$ in $\bdry_c X$ such that one can take
$\rho_\zeta\leq \rho$, as in \fullref{lem:uniformlycontractingmeansallthetopologiesarethesame}.
The topologies $\fq$, $\fg$, and $\fg'$ on $\bdry_\rho
X$ coincide by \fullref{lem:uniformlycontractingmeansallthetopologiesarethesame}.
For $\rho\leq\rho'$ the inclusion $\bdry_\rho X\into \bdry_{\rho'}X$ is
continuous, and $\cbdry X$, as a set, is the direct limit of this system of
inclusions over all contraction functions.

Let $\dl$ be the direct limit topology on $\cbdry X$, that is, the
finest topology on $\cbdry X$ such that all of the inclusion maps
$\bdry_\rho X\into \cbdry X$ are continuous.

\begin{proposition}\label{dlrefinesfq}
  $\dl$ is a refinement of $\fq$.
\end{proposition}
\begin{proof}
 The universal property of the direct limit topology says that a map
 from the direct limit is continuous if and only if the precomposition with
 each inclusion map is continuous.
Thus, it suffices to show the inclusion $\bdry_\rho X\into \cbdry
^\fq X$ is continuous.
This is clear from
\fullref{lem:uniformlycontractingmeansallthetopologiesarethesame},
since we can take the topology on $\bdry_\rho X$ to be the subspace
topology induced from $\bdry_\rho X\into \cbdry
^\fq X$.
\end{proof}

\begin{lemma}\label{lem:equivalenttoCSC} 
  $\cbdry^\dl X$ is homeomorphic to Cordes's Morse boundary.
\end{lemma}
\begin{proof}
Cordes considers Morse geodesic rays, and defines the Morse boundary
to be the set of asymptotic equivalence classes of Morse geodesic rays
based at $\bp$, topologized by taking the direct limit topology of the
system of uniformly Morse subsets.
By \fullref{morseequivalent}, a collection of uniformly Morse rays is contained in a
collection of uniformly contracting rays, and vice versa. 
It follows as in \cite[Remark~3.4]{Cor15} that the direct limit topology over
uniformly Morse points and the direct limit topology over uniformly
contracting points agree on $\cbdry X$.
\end{proof}

As with the other topologies, if $X$ is hyperbolic then $\cbdry^\dl X$
is homeomorphic to the Gromov boundary.
Thus, if $X$ is a proper geodesic hyperbolic metric space then all of the above
topologies yield a compact contracting boundary. 
Conversely, Murray \cite{Mur15} showed if $X$ is a complete CAT(0) space admitting a properly discontinuous,
cocompact, isometric group action, and if $\cbdry^\dl
X$ is compact and non-empty, then $X$ is hyperbolic.
Work of Cordes and Durham \cite{CorDur16} shows that if the
contracting boundary, with topology $\dl$, of a finitely generated group is non-empty and
compact then the group  is hyperbolic.
We will prove  this for  $\fq$ in \fullref{sec:compact}.

\bigskip

We have shown that all of the topologies we consider agree for
hyperbolic groups. 
More generally, we could ask about relatively hyperbolic groups.
There are many ways to define relatively hyperbolic groups \cite{Far98,Bow12,Osi06,GroMan08,Dru09,DruSap05,MR2684983,Sis13projection}, all of
which are equivalent in our setting.
Let $G$ be a finitely generated group that is hyperbolic relative to
a collection of \emph{peripheral} subgroups $\mathcal{P}$.
Fix a finite generating set for $G$.
We again use $G$ to denote the Cayley graph of $G$ with respect to
this generating set.
Let $\tilde{G}$ be the \emph{cusped space} obtained by gluing a
combinatorial horoball onto each left coset of a peripheral subgroup,
as in \cite{GroMan08}.
The cusped space is hyperbolic, and its boundary $\bdry\tilde{G}$ is
the Bowditch boundary of $(G,\mathcal{P})$.
Points in the Bowditch boundary that are fixed by a conjugate of a
peripheral subgroup are known as \emph{parabolic points}, and the
remaining points are known as \emph{conical points}.
As described, $G$ sits as a
subgraph in $\tilde{G}$.

\begin{theorem}\label{thm:relhyp}
  If a finitely generated group $G$ is hyperbolic relative to
  $\mathcal{P}$, then the inclusion $\iota\from
  G\into \tilde{G}$ induces a continuous, $G$--equivariant map $\iota_*\from \cbdry^\fg
  G\to \bdry \tilde{G}$ that is injective at conical points.

For $\iota_*\from \cbdry^\fq
  G\to \bdry \tilde{G}$, the preimage of a parabolic point is the
  contracting boundary of its stabilizer subgroup embedded in $\cbdry^\fq
  G$ as in \fullref{cor:he}.

Let $q\from \cbdry^\fq G\to \cbdry^\fq G/\iota_*$ be the quotient map from $\cbdry^\fq G$ to its
$\iota_*$--decomposition space, that is, the quotient space of
$\cbdry^\fq G$ obtained by collapsing to a point the preimage of each point in
$\iota_*(\cbdry^\fq G)$.
If each peripheral subgroup is hyperbolic or has empty contracting
boundary then $\iota_*\circ q^{-1}$ is an embedding.
\end{theorem}
\fullref{thm:relhyp} for $\dl$, without the embedding result, was
first observed by
Tran \cite{Tra16}.
Recall from the introduction that the embedding statement is not true
for $\dl$ (cf \cite[Remark~8.13]{Tra16}).
\begin{corollary}
  If $G$ is a finitely generated group that is hyperbolic relative to
  subgroups with empty contracting boundaries then $\cbdry^\fq
  G=\cbdry^\fg G$.
\end{corollary}
Since the contracting boundary of  a hyperbolic group is the same as
the Gromov boundary, we also recover the following well-known result
(see \cite{Man15} and references therein). 
\begin{corollary}
  If $G$ is hyperbolic and hyperbolic relative to $\mathcal{P}$ then
  the Bowditch boundary of $(G,\mathcal{P})$ can be obtained from the
  Gromov boundary of $G$ by collapsing to a point each embedded Gromov
  boundary of a peripheral subgroup. 
\end{corollary}

The following example shows that the embedding statement of
\fullref{thm:relhyp} can fail when a peripheral subgroup is
non-hyperbolic with non-trivial contracting boundary.
\begin{example}\label{ex:doublefreeproduct}
Let $A:=\langle a,b \mid [a,b]=1\rangle$, $H:=A*\langle c\rangle$, and
$G:=H*\langle d\rangle$. Since $G$ is a free product of $H$ and a
hyperbolic group, $G$ is hyperbolic relative to $H$. 

A geodesic $\alpha$ in $G$ (or $H$) is contracting if and only if there is a
bound $B$ such that $\alpha$ spends at most time $B$ in any given
coset of $A$.

Consider the sequence $(a^nd^\infty)_{n\in\mathbb{N}}$ in $\cbdry^\fq
G$. 
We have $(\iota_*(a^nd^\infty))\to \iota_*(\cbdry^\fq H)$, which is a
parabolic point in 
$\bdry\tilde{G}$.
However, $(q(a^nd^\infty))$ does not converge in
$\cbdry^\fq G/\iota_*$.
To see this, note that every edge $e$ in the Cayley graph of $G$ with one
incident vertex in $A$ determines a clopen
subset $U_e$ of $\cbdry^\fq G$ consisting of all $\zeta\in\cbdry^\fq G$
such that $\alpha^\zeta$ crosses $e$.
Let $U$ be the union of the $U_e$ for every edge $e$ incident to $A$
and labelled by $c$ or $c^{-1}$. This is an open set containing
$\cbdry^\fq H$ such that 
$q^{-1}(q(U))=U$ and $a^nd^\infty\notin U$ for all $n\in\mathbb{N}$. 
Therefore, $q(U)$ is an open set in $\cbdry^\fq G/\iota_*$ containing
the point $q(\cbdry^\fq H)$ but not containing any 
$q(a^nd^\infty)$. 
\end{example}

Before proving the theorem let us recall some of the necessary
machinery for relatively hyperbolic groups.
Any bounded set meets finitely many cosets of
the peripherals, and projections of peripheral sets  to one another are uniformly
bounded.

Given an $(\cL,\cA)$--quasi-geodesic $\gamma$, Drutu and Sapir \cite{DruSap05}
define the \emph{saturation} $\mathrm{Sat}(\gamma)$ of $\gamma$ to be
the union of $\gamma$ and all cosets $gP$ of peripheral subgroups
$P\in\mathcal{P}$ such that $\gamma$ comes within distance
$M$ of $gP$,  where $M$ is a number depending on $\cL$ and $\cA$.
\cite[Lemma~4.25]{DruSap05} says there exists $\mu$ independent of
$\gamma$ such that the saturation of $\gamma$ is $\mu$--Morse.
It follows that the analogous saturation of $\gamma$ in $\tilde{G}$,
that is, the union of $\gamma$ and all horoballs sufficiently close to
$\gamma$, is also Morse.

Sisto \cite{Sis13projection} extends these results, showing, in
particular, that peripheral subgroups are strongly
contracting.\footnote{Sisto does not use the term `strongly
  contracting', but observe it is equivalent to the first two
  conditions of \cite[Definition~2.1]{Sis13projection}.}

The other key definition is that of a \emph{transition point} of
$\gamma$, as defined by Hruska \cite{MR2684983}.
The idea is that a point of $\gamma$ is \emph{deep} if it is contained
in a long subsegment of $\gamma$ that is contained in a neighborhood
of some $gP$, and a point is a transition point if it is not deep.
A quasi-geodesic in $\gamma$ is bounded Hausdorff
distance from a path of the form
$\beta_0+\alpha_1+\beta_1+\alpha_2+\cdots$ where the $\beta_i$ are
shortest paths connecting some $g_iP_i$ to some $g_{i+1}P_{i+1}$ and
the $\alpha_i$ are paths in $g_iP_i$.
The transition points are
the points close to the $\beta$--segments. 
In $\tilde{G}$ there is an obvious way to shorten such a path by
letting the $\alpha$--segments relax into the corresponding
horoballs. 
If the endpoints of $\alpha_i$ are $x$ and $y$,
this replaces $\alpha_i$ with a segment of length roughly $2\log_2
d_G(x,y)$ in $\tilde{G}$. 
This is
essentially all that happens: if $\gamma$ is a quasi-geodesic in $G$
then take a geodesic $\hat{\gamma}$ with the same endpoints as
$\gamma$ in the coned-off space $\hat{G}$
obtained by collapsing each coset of a peripheral subgroup.
Lift $\hat{g}$ to a nice $\alpha$-$\beta$ path in $G$ as above. 
The $\beta$--segments are coarsely well-defined, because the cosets of
peripheral subgroups are strongly contracting, and the union of the
$\beta$ segments is Hausdorff equivalent to the set of transition
points of $\gamma$.
Only the endpoints of the $\alpha$--segments are coarsely well
defined, but relaxing the $\alpha$-segments to geodesics in the
corresponding horoball yields a uniform quasi-geodesic in $\tilde{G}$ (see
\cite[Section~7]{Bow12}).
Since $\tilde{G}$ is hyperbolic, this is within bounded Hausdorff
distance of any $\tilde{G}$ geodesic $\tilde{\gamma}$ with the same endpoints as $\gamma$.
In particular, $\tilde{\gamma}$ comes boundedly close to the
transition points of $\gamma$.

\begin{proof}[{Proof of \fullref{thm:relhyp}}]
We omit $\iota$ from the notation and think of $G$ sitting as a
subgraph of $\tilde{G}$.
First we show that for $\zeta\in\cbdry G$ the sequence
$(\alpha_n^\zeta)_{n\in\mathbb{N}}$ converges to a point of $\bdry\tilde{G}$. 
Distances in $G$ give an upper bound for distances in $\tilde{G}$, so
all quasi-geodesics in $G$ asymptotic to $\alpha^\zeta$ also converge to this
point in $\bdry\tilde{G}$, which we define to be $\iota_*(\zeta)$.
Let $(x \cdot y):=\frac{1}{2}(d(\one,x)+d(\one, y)-d(x,y))$ denote the Gromov product of $x,y \in \tilde{G}$ with respect to the basepoint
$\one$ corresponding to the identity element of $G$. 
(See
\cite[Section~III.H.3]{BriHae99} for background on boundaries of
hyperbolic spaces.)

To see that the sequence $(\alpha^\zeta_n)_{n\in\mathbb{N}}$ does indeed converge,
there are two cases.
If $\alpha^\zeta$ has unbounded projection to
$gP$ for some $g\in G$ and $P\in\mathcal{P}$,
then a tail of $\alpha^\zeta$ is contained in a bounded neighborhood
of $gP$, but leaves every bounded subset of $gP$.
It follows that $(\alpha^\zeta_n)$
converges to the parabolic point in $\bdry \tilde{G}$ fixed by
$gPg^{-1}$ corresponding to the horoball attached to $gP$.
Furthermore, the projection of the tail of $\alpha^\zeta$ to $gP$ is a
contracting quasi-geodesic ray in $gP$ (by \fullref{morsetransitive}), so $P$ has non-trivial contracting
boundary.

The other case is that $\alpha^\zeta$ has bounded (not necessarily
uniformly!) projection to every $gP$.
Now, given any $r$ there are only finitely many horoballs in
$\tilde{G}$ that meet the $r$--neighborhood of $\one$. 
Since $\alpha^\zeta$ has bounded projection to each of these,
for sufficiently large $s$ none of these are in
$\mathrm{Sat}(\alpha^\zeta_{[s,\infty)})$.
Since $\mathrm{Sat}(\alpha^\zeta_{[s,\infty)})$ is $\mu$--Morse in
$\tilde{G}$ for some $\mu$ independent of $\alpha^\zeta$, for
any $m,\, n\geq s$, geodesics connecting $\alpha^\zeta_m$ and
$\alpha^\zeta_{n}$ in $\tilde{G}$ stay outside the $(r-\mu(1,0))$--ball
about $\one$. 
We conclude
$\lim_{m,\,n\to\infty}(\alpha^\zeta_m\cdot\alpha^\zeta_n)_{\tilde{G}}=\infty$,
so $(\alpha^\zeta_n)_{n\in\mathbb{N}}$ converges to a point in
$\bdry\tilde{G}$, which, in this case, is a conical point. 

If $\alpha^\zeta$ and $\alpha^\eta$ tend to the same conical point in
$\bdry\tilde{G}$ then the sets of transition points of $\alpha^\zeta$
and $\alpha^\eta$ are unbounded and at bounded Hausdorff distance from
one another in $G$. 
Since they are contracting geodesics in $G$ they
can only come close on unbounded sets if they are in fact asymptotic,
so $\iota_*$ is injective at conical points.

{\itshape Continuity:}
To show $\iota_*\from\cbdry^\fg G\to \bdry\tilde{G}$ is continuous we
show that for all $\zeta\in\cbdry G$ and all $r$ there exists an $R$
such that for all $\eta\in V(\zeta,R))$ we have
$(\iota_*(\zeta)\cdot\iota_*(\eta))_{\tilde{G}}>r$.

Recall that there is a bound $B$ such that a $\tilde{G}$ geodesic
comes $B$--close to the transition points of a $G$ geodesic with the
same endpoints.
There exists $B'$ so that $\diam \pi_{gP}(x)\leq B'$ for each $x \in G, g \in G, P \in \mathcal{P}$, and so that for any deep point $x$ of a geodesic along $gP$ we have $\diam \{x\}\cup \pi_{gP}(x) \leq B'$.
Finally, there exists a constant $B''$ depending on $B$ so that if $x, y \in G$ satisfy $d(\pi_{gP}(x),\pi_{gP}(y))\geq B''$ for some $g \in G, P \in \mathcal{P}$, then any geodesic from $x$ to $y$ has a deep component along $gP$ whose transition points at the ends are within $B''$ of $\pi_{gP}(x)$ and $\pi_{gP}(y)$, respectively.

Suppose $\zeta$ and $r$ are given.
If $\iota_*(\zeta)$ is conical then given any $r'\geq 0$ there is an $R'$ such
that for all $n\geq R'$ we have
$d_{\tilde{G}}(\one,\alpha^\zeta_n)>r'$.
Choose $R\geq R'$ such that $\alpha^\zeta_R$ is a transition point,
and moreover that any deep component along $\alpha^\zeta$ within $\kappa'(\rho_\zeta,1,0)+B'+B''$ of $\alpha^\zeta_R$ has distance at least $R'$ from $\one$.
If $\eta\in V(\zeta,R)$ then $\alpha^\eta$ comes
$\kappa(\rho_\zeta,1,0)$--close to $\alpha^\zeta_{[R,\infty)}$ in
$G$, so there is a point $\alpha^\eta_t$ that is 
$\kappa'(\rho_\zeta,1,0)$--close to $\alpha^\zeta_R$.
If $\alpha^\eta_t$ is a deep point of $\alpha^\eta$, let $g'P'$ be the corresponding coset.
If $d(\pi_{g'P'}(\one),\alpha^\eta_t)>J:=B''+2B'+2\kappa'(\rho,1,0))$ then the geodesic $\alpha^\zeta$ must also have a deep component along $g'P'$ with one endpoint $(\kappa'(\rho_\zeta,1,0)+B'+B'')$--close to $\alpha^\zeta_R$ and the other $z:= \alpha^\zeta_{s} \in \bar{N}_{B''}\pi_{g'P'}(\one)$; by assumption, $s \geq R'$.
If $\alpha^\eta_t$ is a transition point of $\alpha^\eta$, or if $d(\pi_{g'P'}(\one),\alpha^\eta_t)\leq J$, then $z:=\alpha^\zeta_R$ is $(J+B'')$--close to a transition point of $\alpha^\eta$.
In either case then, $z$ is a transition point of $\alpha^\zeta$ which is $(J+B'')$--close to a transition point of $\alpha^\eta$, and has $d(z,\one) \geq R'$.
Thus, by the choice of $B$, there are points $x\in[\one,\iota_*(\zeta)]$ and
$y\in[\one,\iota_*(\eta)]$ with $d_{\tilde{G}}(\one,x)\geq r'-B$ and
$d(x,y)\leq 2B+J+B''$.
This allows us to bound $(\iota_*(\zeta)\cdot\iota_*(\eta))$ below in
terms of these constants and the hyperbolicity constant for
$\tilde{G}$, and by choosing $r'$ large enough we guarantee $(\iota_*(\zeta)\cdot\iota_*(\eta))>r$.

Now suppose $\iota_*(\zeta)$ is parabolic. 
Then there is some $R_0\geq 0$, $M\geq 0$, $g\in G$, and $P\in\mathcal{P}$ such
that $\alpha^\zeta_{[R_0,\infty)}\subset N_MgP$.
For $R\gg R_0$, if $\eta\in V(\zeta,R)$ then $\alpha^\eta$ comes within
distance $\kappa(\rho_\zeta,1,0)$ of $\alpha^\zeta_{[R,\infty)}$.
If $\iota_*(\zeta)\neq\iota_*(\eta)$ then eventually $\alpha^\eta$
escapes from $gP$, so it has a transition point at $G$--distance
greater than $R-R_0-\kappa(\rho_\zeta,1,0)$ from $\alpha^\zeta_{R_0}$.
This implies $\diam \pi_{gP}([\one,\iota_*(\eta)])>R-R_0-C$, where $C$
depends on $M$, $B$, the contraction function of $gP$, and
$\kappa(\rho_\zeta,1,0)$. 
It follows from the geometry of the horoballs that for $\iota_*(\zeta)$ 
is the parabolic boundary point corresponding to $gP$ and $y\in
\bdry\tilde{G}$ we have $(\iota_*(\zeta)\cdot y)$ is roughly $d_{\tilde{G}}(\one, gP)+\log_2\diam
\pi_{gP}(\one)\cup\pi_{gP}(y)$, so by choosing $R>2^r+R_0+C$ we guarantee $(\iota_*(\zeta)\cdot\iota_*(\eta))>r+d_{\tilde{G}}(\one,gP) \geq r$.

{\itshape Embedding:} 
Suppose that $U'\subset \cbdry^\fq G/\iota_*$ is open. 
Define $U:=q^{-1}(U')$, which is open in $\cbdry^\fq G$. 
We claim that for each $p\in \iota_*(U)$ there exists an
$R_p>0$ such that for $p'\in \iota_*(U)$, if $(p\cdot
p')>R_p$ then $\iota_*^{-1}(p')\subset U$.
Given the claim, the proof concludes by choosing, for each
$p\in\iota_*(U)$, an open neighborhood $V_p$ of $p$ such that
$V_p\subset \{p'\in \bdry\tilde{G} \mid
(p\cdot p')>R_p\}$, and setting $V=\bigcup_{p\in\iota_*(U)}V_p$.
Then $V$ is open and 
$\iota_*(\cbdry^\fq
G)\cap V = \iota_*(U)$, so that  $\iota_*\circ q^{-1}(\cbdry^\fq
G/\iota_*)\cap V = \iota_*\circ q^{-1}(U')$.

First we prove the claim when $p$ is conical. 
In this case  there is a unique point
$\zeta\in\iota_*^{-1}(p)$, and since $U$ is open there exists
$r_\zeta>1$ such that $U(\zeta,r_\zeta)\subset U$.
Let $x$ be a transition point of $\alpha^\zeta$ and
choose $R$ such that $d_{\tilde{G}}(\one,x)\leq R$.
If the claim is false then there exists an $\eta\in \cbdry^\fq G$ such
that $(\iota_*(\zeta)\cdot\iota_*(\eta))>R$ and $\eta\notin U(\zeta,r_\zeta)$.
Since $\eta\notin U(\zeta,r_\zeta)$ , there exists an $\cL$ and $\cA$ and a continuous
$(\cL,\cA)$--quasi-geodesic $\gamma\in\eta$
such that the last point $y\in\alpha^\zeta$ such that
$d_G(y,\gamma)=\kappa(\rho_\zeta,\cL,\cA)$ satisfies $d(\one,y)<r_\zeta$.
By \fullref{obs}, we can take $\cL<\sqrt{r_\zeta/3}$ and $\cA<r_\zeta/3$.

By hyperbolicity, geodesics in $\tilde{G}$ tending to $\iota_*(\zeta)$
and $\iota_*(\eta)$ remain  boundedly close together for distance
approximately $(\iota_*(\zeta)\cdot\iota_*(\eta))>R$.
Since $x$ is a transition point of $\alpha^\zeta$ there is a $B$ such
that any  geodesic $[\one,\iota_*(\zeta)]$ comes $B$--close to $x$, so some point $z'$ in a
geodesic $[\one,\iota_*(\eta)]$ also comes boundedly close to $x$.
If the point $z'$ lies in a horoball along which $\gamma$ has a deep component, whose transition points at both ends are close to $[\one,\iota_*(\zeta)]$, then this deep component must be of bounded size else $x \in \alpha^\zeta$ would not be a transition point.
It follows that $\gamma$ must contain a point $z$ at bounded distance from $x$.
Since $x$ and $z$ are transition points, we also get a bound on
$d_G(x,z)$.
Then, by
applying \fullref{keylemma}, we get an upper bound on $d_G(\one,x)$
depending on  $r_\zeta$ and $\rho_\zeta$, but independent of $\gamma$ and
$\eta$.
However, if the set of transition points of $\alpha^\zeta$ is bounded
in $G$ then it is bounded in $\tilde{G}$, which would imply $\iota_*(\zeta)=p$ is parabolic, contrary to hypothesis. 

Now suppose $p$ is parabolic. 
By hypothesis, its
stabilizer $G_p$ is a hyperbolic group conjugate into $\mathcal{P}$.
Since the maps are $G$--equivariant we may assume $G_p\in
\mathcal{P}$.
We may assume that we have chosen a generating set for $G$ extending
one for $G_p$. 
Since $G_p$ is quasi-isometrically embedded, by \fullref{cor:qiembedding}, there exist $\cL_p\geq 1,\,\cA_p\geq 0$ with
$\frac{1}{\cL_p}d_{G_p}(x,y)-\cA_p\leq d(x,y)\leq d_{G_p}(x,y)$ for
all $x,\,y\in G_p$.
The contracting boundary $\cbdry^\fq G_p$ embeds into $\cbdry^\fq G$
and is compact---it is homeomorphic to the Gromov boundary 
of $G_p$. 
Geodesic rays in $G_p$ are uniformly contracting, by hyperbolicity, so
there exists a contracting function $\rho$ such that for all
$\zeta\in\cbdry^\fq G_p\subset\cbdry^\fq G$ we have that $\zeta$ is $\rho$--contracting.

We will verify the following fact at the end of this proof:
\begin{equation}
  \label{eq:6}
 \exists R'_p>1\,\forall \xi\in\cbdry^\fq G_p,\, U(\xi,R'_p)\subset U
\end{equation}
Assuming \eqref{eq:6}, let $\eta\in\iota_*^{-1}(p')$ and let
$\gamma\in\eta$ be a continuous $(\cL,\cA)$--quasi-geodesic for some
$\cL<\sqrt{R'_p/3},\,\cA< R'_p/3$.
Since $G_p$ is strongly contracting, there exist $C$ and $C'$ such
that the diameter of $\pi_{G_p}(\alpha^\eta\cap N_C^cG_p)$ is at most
$C'$. We define $\pi_{G_p}(\eta)$ to be this finite diameter set. 
Since $\gamma$ stays $\kappa'(\rho_\eta,\cL,\cA)$--close to
$\alpha^\eta$, strong contraction implies $\pi_{G_p}(\gamma\cap
N_{\kappa'(\rho_\eta,\cL,\cA)}^c G_p)\subset N_{C'}\pi_{G_p}(\eta)$.
We apply  \cite[Lemma~1.15]{Sis13projection} to any sufficiently
long\footnote{Long enough to leave the
  $\max\{D_0(\cL,\cA),\kappa'(\rho_\eta,\cL,\cA)\}$--neighborhood of
    $G_p$ where $D_0$ is as in \cite[Lemma~1.15]{Sis13projection}.} initial subsegment
of $\gamma$ to conclude  there is a function $K$, a
point $z\in\gamma$, and a point $x\in\pi_{G_p}(\eta)$ such that
$d(x,z)\leq K(\cL,\cA)$.

Since $G_p$ is a hyperbolic group there exists a constant $D$ such
that every point is within $D$ of a geodesic ray based at $\one$. 
Let $\xi\in\cbdry^\fq G_p$ be a point such that there is a
$G_p$--geodesic $[\one,\xi]$ containing a point $w$ with $d(w,x)\leq D$.
Since this $G_p$--geodesic is a $(\cL_p,\cA_p)$--quasi-geodesic in
$G$, there exists $y'\in\alpha^\xi$ such that
$d(y',w)\leq\kappa'(\rho,\cL_p,\cA_p)$.

We have $d(z,y')\leq K(\cL,\cA)+D+\kappa'(\rho,\cL_p,\cA_p)$ and:
\begin{align*}
  d(y',\one)&\geq d(x,\one)-D-\kappa'(\rho,\cL_p,\cA_p)\\
&\geq \frac{d_{G_p}(\one,\pi_{G_p}(\eta))}{\cL_p}-\cA_p-D-\kappa'(\rho,\cL_p,\cA_p)
\end{align*}
\fullref{keylemma} implies that, for $\theconstantformerlyknownasLambda$ and $\lambda$ as in the
lemma, $\gamma$ comes within distance
$\kappa(\rho,\cL,\cA)$ of $\alpha^\xi$ outside the ball of radius:
\begin{equation}
  \label{eq:10}
\frac{d_{G_p}(\one,\pi_{G_p}(\eta))}{\cL_p}-\cA_p-D-\kappa'(\rho,\cL_p,\cA_p)-\theconstantformerlyknownasLambda(K(\cL,\cA)+D+\kappa'(\rho,\cL_p,\cA_p))-\lambda(\rho,\cL,\cA)  
\end{equation}

Now, $d_{G_p}(\one,\pi_{G_p}(\eta))\asymp
2^{(p\cdot p')}>2^{R_p}$.
Since $\cL<\sqrt{R'_p/3},\,\cA< R'_p/3$, all the negative terms
are bounded in terms of $R'_p$, so we can guarantee \eqref{eq:10} is
greater than $R'_p$ by taking $R_p$ sufficiently large.
This means the quasi-geodesic $\gamma$ does not witness $\eta\notin U(\xi,R_p')$.
Since $\gamma$ was arbitrary, $\eta\in U(\xi,R_p')$, which, by
\eqref{eq:6}, is contained in $U$. 
Thus, $\iota_*^{-1}(p')\subset U$ when $(p\cdot p')>R_p$ for $R_p$
sufficiently large with respect to $R_p'$. 

\medskip

It remains to determine $R_p'$ and verify \eqref{eq:6}.
Define:
\[\theta(s):=s+\theconstantformerlyknownasLambda(\kappa'(\rho,\sqrt{s/3},s/3)+\kappa(\rho,1,0))+\lambda(\rho,\sqrt{s/3},s/3)\]

Since $U$ is open, for every $\zeta\in\cbdry^\fq G_p$ there exists
$r_\zeta$ such that $U(\zeta,r_\zeta)\subset U$. 
For each $\zeta\in\cbdry^\fq G_p$, let $U_\zeta$ be an open
neighborhood of $\zeta$ such that $U_\zeta\subset
U(\zeta,\theta(r_\zeta))$.
Then $\{U_\zeta\}_{\zeta\in\cbdry^\fq G_p}$ is an open cover of
$\cbdry^\fq G_p$, which is compact, so there exists a finite set
$F\subset \cbdry^\fq G_p$ such that $\cbdry^\fq G_p\subset
\bigcup_{\zeta\in F}U_\zeta\subset \bigcup_{\zeta\in
  F}U(\zeta,\theta(r_\zeta))$.
Define $r:=\max_{\zeta\in F}r_\zeta$ and:
\[R_p':=r+\kappa'(\rho,1,0)+\theconstantformerlyknownasLambda(\kappa(\rho,\sqrt{r/3},r/3)+\kappa'(\rho,1,0))+\lambda(\rho,\sqrt{r/3},r/3)\]

Suppose that $\xi\in\cbdry^\fq G_p$ and $\eta\in U(\xi,R_p')$. There
exists $\zeta\in F$ such that $\xi\in U_\zeta\subset
U(\zeta,\theta(r_\zeta))$.
Let $\gamma\in\eta$ be a continuous $(\cL,\cA)$--quasi-geodesic for
some $\cL<\sqrt{r_\zeta/3},\,\cA< r_\zeta/3$.
Since $\eta\in U(\xi,R_p')$, there exist $z\in\gamma$ and $x\in
\alpha^\xi$ such that $d(x,z)\leq \kappa(\rho,\cL,\cA)$ and
$d(x,\one)\geq R_p'$.

There are now two cases to consider.
First, suppose that there exists $y'\in\alpha^\zeta$ with $d(x,y)\leq \kappa'(\rho,1,0)$.
Then $d(y',\one)\geq R_p'-\kappa'(\rho,1,0)$.
By \fullref{keylemma}, $\gamma$ comes $\kappa(\rho,\cL,\cA)$ close to
$\alpha^\zeta$ outside the ball of radius:
\[R_p'-\kappa'(\rho,1,0)-\theconstantformerlyknownasLambda(\kappa'(\rho,1,0)+\kappa(\rho,\cL,\cA))-\lambda(\rho,\cL,\cA)\]
By definition of $R_p'$ and the conditions
$\cL<\sqrt{r_\zeta/3},\,\cA< r_\zeta/3$, this radius is at least
$r$, which is at least $r_\zeta$, so $\gamma$ does not witness
$\eta\notin U(\zeta,r_\zeta)$.

The second case, where the above $y'$ does not exist, is the case
that $x$ occurs after $\alpha^\xi$ has already escaped
$\alpha^\zeta$. 
In this case there exists $x'$ between $\one$ and $x$ on $\alpha^\xi$
and $y'\in\alpha^\zeta$ such that $d(x',y')\leq \kappa(\rho,1,0)$ and
$d(\one,y')\geq \theta(r_\zeta)$.
Moreover, by \fullref{lem:hausdorff}, there exists $z'\in\gamma$ such that
$d(z',x')\leq \kappa'(\rho,\cL,\cA)$.
By \fullref{keylemma}, $\gamma$ comes within distance
$\kappa(\rho,\cL,\cA)$ of $\alpha^\zeta$ outside the ball of radius:
\[\theta(r_\zeta)-\theconstantformerlyknownasLambda(\kappa(\rho,1,0)+\kappa'(\rho,\cL,\cA))-\lambda(\rho,\cL,\cA)\]
By definition of $\theta$ and the conditions
$\cL<\sqrt{r_\zeta/3},\,\cA< r_\zeta/3$, this radius is at least $r_\zeta$, so $\gamma$ does not witness
$\eta\notin U(\zeta,r_\zeta)$.
This verifies \eqref{eq:6}.
\end{proof}

\section{Metrizability for group boundaries}\label{sec:metrizability}
In this section let $G$ be a finitely generated group with nonempty
contracting boundary.
Consider the Cayley graph of $G$ with respect to some fixed finite generating
set, which is a proper geodesic metric space we again denote $G$, and
take the basepoint to be the vertex $\one$ corresponding to the
identity element of the group.

There is a natural action of $G$ on $\cbdry^\fq G$ by homeomorphisms
defined by sending $g\in G$ to the map that takes $\zeta\in\cbdry G$ to
the equivalence class of the quasi-geodesic that is the concatenation
of a geodesic from $\one$ to $g$ and the geodesic $g\alpha^\zeta$.

The following two results generalize results of Murray \cite{Mur15}
for the case of $\cbdry^\dl X$ when $X$ is CAT(0). See also
\cite{Ham09b}.

\begin{remark}
  If $\beta\from I\to G$ is a geodesic and $\beta_m$ is a vertex for some
  $m\in\mathbb{Z}\cap I$ then $\beta_n$ is a vertex for every
  $n\in\mathbb{Z}\cap I$.
  Vertices in the Cayley graph are in one-to-one correspondence with
  group elements.
  If $Z$ is a subset of the Cayley graph we use $\beta_nZ$ to denote
  the image of $Z$ under the action by the group element corresponding to the
  vertex $\beta_n$.

  We will always parameterize bi-infinite geodesics in $G$ so that
  integers go to vertices.
\end{remark}

\begin{proposition}\label{finiteorbit}
  $G$ is virtually (infinite) cyclic if and only if $G\act\cbdry^\fq G$ has a finite orbit.
\end{proposition}
\begin{proof}
  If $G$ is virtually cyclic then $|\cbdry G|=2$ and every orbit is
  finite.

Conversely, if $G$ has a finite orbit then it has a finite index
subgroup that fixes a point in $\cbdry G$. 
The inclusion of a finite index subgroup is a quasi-isometry, so we
may assume that $G$ fixes a point $\zeta\in\cbdry G$. 

Let $\alpha\in\zeta$ be geodesic and $\rho$--contracting. 
Let $\beta$ be an arbitrary geodesic ray or segment with $\beta_0=\one$.
Since $G\zeta=\zeta$, for all $n\in\mathbb{N}$ the geodesic rays
$\alpha$ and $\beta_n\alpha$ are asymptotic. 
By \fullref{GIT}, $\alpha$ and $\beta_n\alpha$ eventually stay within
distance $\kappa'_\rho$ of one another.
Truncate $\alpha$ and $\beta_n\alpha$ when their distance is
$\kappa'_\rho$. 
By \fullref{lem:subsegment}, these segments are contracting, and they
form a geodesic almost triangle with $\beta_{[0,n]}$, so, by
\fullref{lem:almosttriangle}, $\beta_{[0,n]}$ is $\rho'$--contracting
for some $\rho'\asymp \rho$ depending only on $\rho$.
Since this is true uniformly for all $n$, $\beta$ is
$\rho'$--contracting. 
Since $\beta$ was arbitrary and $G$ is homogeneous, every geodesic in
$G$ is uniformly contracting, which means $G$ is hyperbolic and
$\cbdry^\fq G$ is the Gromov boundary.
If $G$ is hyperbolic and not virtually cyclic then its boundary is
uncountable and every orbit is dense, hence infinite.
\end{proof}

\begin{proposition}\label{prop:countablebasis}
  Suppose $|\cbdry^\fq G|>2$, and fix a point $\eta\in\cbdry G$. 
For every $\zeta\in\cbdry G$ and
  every $r\geq 1$ there exists an $R'\geq 1$ such that for all $R_2\geq R_1\geq R'$
  there exist $g\in G$ such that $\zeta\in U(g\eta,R_2)\subset
  U(g\eta,R_1) \subset U(\zeta,r)$.
\end{proposition}
\begin{corollary}
  $\cbdry^\fq G$ is separable.
\end{corollary}  
\begin{corollary}
 If $G$ is not virtually cyclic then
  $G\act \cbdry^\fq G$ is minimal, that is, every orbit is dense.
\end{corollary}
\begin{remark}
For the corollaries we just need to know that we can push
$\eta$ into $U(\zeta,r)$ via the group action.
The stronger statement of \fullref{prop:countablebasis}
is used in \fullref{secondcountable} to upgrade first countable and separable to second countable.
The reason for having two parameters $R_1$ and $R_2$ is to be able to
apply \fullref{cor:annulus} in case $U(g\eta,R_1)$ is not an open set.
\end{remark}
\begin{proof}[Proof of \fullref{prop:countablebasis}]
 By \fullref{finiteorbit}, $G\act\cbdry^\fq G$ does not have a finite
 orbit, so there exists a $g'\in G$ with $\eta':=g'\eta\neq\eta$.
Let $\beta$ be a geodesic joining $\eta'$ and $\eta$.
It suffices to assume $\beta_0=\one$; otherwise, we could consider
 $\beta':=\beta^{-1}_0\beta$, which is a geodesic with $\beta'_0=\one$
and endpoints in $G\eta$.

Let $\alpha:=\alpha^\zeta$ be the geodesic representative of $\zeta$.
Choose $\rho$ so that $\alpha$, $\beta_{[0,\infty)}$, and $\bar{\beta}_{[0,-\infty)}$ are all $\rho$--contracting.

For each integral $t\gg 0$, at most one of $\alpha_t \beta_{[0,\infty)}$ and
$\alpha_t \bar{\beta}_{[0,-\infty)}$ remains in the closed $\kappa'(\rho,1,0)$--neighborhood
of $\alpha_{[0,t]}$ for distance greater than $2\kappa'(\rho,1,0)$,
otherwise we contradict the fact that $\alpha_t\beta$ is a geodesic. 
Define $g_t:=\alpha_t$ if $\alpha_t \beta_{[0,\infty)}$ does not remain in the closed $\kappa'(\rho,1,0)$--neighborhood
of $\alpha_{[0,t]}$ for distance greater than $2\kappa'(\rho,1,0)$.
Otherwise, define $g_t:=\alpha_tg'$.
For each $s\in\mathbb{N}$ consider a geodesic triangle with sides
$\alpha_{[0,t]}$, $g_t\beta_{[0,s]}$, and a geodesic $\delta^{s,t}$ joining
$\alpha_0$ to $g_t\beta_s$.
By \fullref{lem:subsegment}, the first two sides are uniformly
contracting, so $\delta^{s,t}$ is as well, by
\fullref{lem:almosttriangle}.
Since $G$ is proper, for each fixed $t$ a subsequence of the
$\delta^{s,t}$ converges to a contracting geodesic ray $\delta^t\in
g_t\eta$.
See \fullref{fig:82}.
\begin{figure}[h!]
  \centering
  \labellist
  \tiny
  \pinlabel $\beta$ [r] at 1 88
  \pinlabel $\beta_\infty=\eta$ [b] at 1 112
  \pinlabel $\beta_{-\infty}=g'\eta$ [t] at 1 2
  \pinlabel $\zeta$ [l] at 240 56
  \pinlabel $g_{t_1}\eta:=\alpha_{t_1}\eta$ [b] at 146 115
  \pinlabel $g_{t_2}\eta:=\alpha_{t_2}g'\eta$ [t] at 201 0
  \pinlabel $\alpha$ [b] at 228 57
  \pinlabel $\delta^{t_1}$ [r] at 146 106
  \pinlabel $\delta^{t_2}$ [r] at 201 11
  \pinlabel $\alpha_{t_1}\beta$ [l] at 75 16
  \pinlabel $\alpha_{t_2}\beta$ [bl] at 199 88
  \pinlabel $\delta^{1,t_1}$ [l] at 156 62
  \pinlabel $\delta^{2,t_1}$ [l] at 156 76
  \pinlabel $\delta^{3,t_1}$ [l] at 156 97
  \endlabellist
  \includegraphics{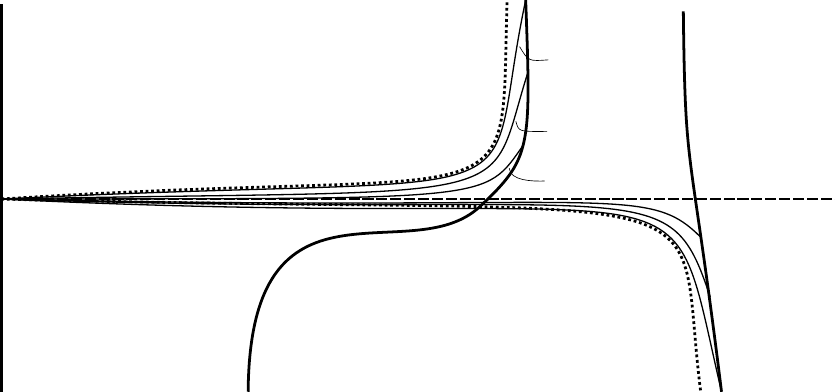}
  \caption{}\label{fig:82}
\end{figure}
Moreover, since the $\delta^{s,t}$ are uniformly contracting, the
contraction function for $\delta^t$ does not depend on $t$.
Now, for any given $t$ it is possible that $\delta^t$ does not
coincide with the  chosen representative $\alpha^{g_t\eta}$ of
$g_t\eta$, but they are asymptotic, so
\fullref{cor:uniformcontraction} tells us that uniform contraction for
the $\delta^t$ implies uniform contraction for the $\alpha^{g_t\eta}$.
Thus, there is a $\rho'$ independent of
$t$ such that $\alpha^{g_t\eta}$ is $\rho'$--contracting.
Furthermore, the  defining condition for $g_t$ guarantees that there
is a $C$ independent of $t$ such that the geodesic representative
$\alpha^{g_t\eta}$ comes within distance $\kappa(\rho,1,0)$ of $\alpha$
outside of $N_{t-C}\one$, which implies that
$\alpha^{g_t\eta}_{[0,t-C]}\subset \bar{N}_{2\kappa'(\rho,1,0)}\alpha_{[0,t-C]}$.

First we give a condition that implies $\zeta\in
U(g_t\eta,R)$. 
Suppose:
\begin{equation}
  \label{eq:4}
t\geq R+C +2\theconstantformerlyknownasLambda(\kappa'(\rho,\sqrt{R/3},R/3)+\kappa'(\rho,1,0))+\lambda(\rho',\sqrt{R/3},R/3)
\end{equation}
Suppose
that $\gamma\in\zeta$ is a continuous
$(\cL,\cA)$--quasi-geodesic. 
By \fullref{obs} is suffices to consider $\cL^2,\,\cA< R/3$.
By \fullref{cor:asymp}, $\gamma\subset \bar
N_{\kappa'(\rho,\cL,\cA)}\alpha$, so 
there is a point $\gamma_a$ that is $(2\kappa'(\rho,\cL,\cA)+2\kappa'(\rho,1,0))$--close to $\alpha^{g_t\eta}_{t-C}$.
By \fullref{keylemma}, $\gamma$ comes
$\kappa(\rho',\cL,\cA)$--close to $\alpha^{g_t\eta}$ outside the
ball around $\one$ of radius
$t-C-\theconstantformerlyknownasLambda(2\kappa'(\rho,\cL,\cA)+2\kappa'(\rho,1,0))-\lambda(\rho',\cL,\cA)$.
By (\ref{eq:4}), this is at least $R$. 
Since $\gamma\in\zeta$ was
arbitrary,  $\zeta\in U(g_t\eta,R)$.

Next, we give a condition that implies $U(g_t\eta,R)\subset
U(\zeta,r)$. Suppose:
\begin{equation}
  \label{eq:5}
t-C\geq R\geq r+\theconstantformerlyknownasLambda(2\kappa'(\rho,1,0)
+2\kappa'(\rho',\sqrt{r/3},r/3))+\lambda(\rho,\sqrt{r/3},r/3)
\end{equation}
Suppose that $\gamma$ is a continuous
$(\cL,\cA)$--quasi-geodesic such that $[\gamma]\in
U(g_t\eta,R)$.
By \fullref{obs}, it suffices to consider $\cL^2,\,\cA<r/3$.
By definition, $\gamma$ comes $\kappa(\rho',\cL,\cA)$ close to
$\alpha^{g_t\eta}$ outside $N_R\one$, so some point $\gamma_b$
is $2\kappa'(\rho',\cL,\cA)$--close to $\alpha^{g_t\eta}_R$,
which implies that 
$d(\gamma_b,\alpha_R) \leq
(2\kappa'(\rho',\cL,\cA)+2\kappa'(\rho,1,0))$.
Now apply \fullref{keylemma} to see that $\gamma$ comes
$\kappa(\rho,\cL,\cA)$--close to $\alpha$ outside the ball around
$\one$ of radius  $R-\theconstantformerlyknownasLambda(2\kappa'(\rho,1,0)
+2\kappa'(\rho',\cL,\cA))-\lambda(\rho,\cL,\cA)$, which is at least
$r$, by (\ref{eq:5}).
Thus, $U(g_t\eta,R)\subset U(\zeta,r)$.

Equipped with these two conditions, we finish the proof.
The contraction functions $\rho$ and $\rho'$ are determined by $\zeta$
and $\eta$. 
Given these and any $r\geq 1$, define $R'$ to be the right hand side of (\ref{eq:5}).
Given any $R_2\geq R_1\geq R'$, it suffices to define $g:=g_t$ for any
$t$ large enough to satisfy
both condition (\ref{eq:4}) for
$R=R_2$ and
condition (\ref{eq:5}) for $R=R_1$.
\end{proof}

\begin{proposition}\label{secondcountable}
  $\cbdry^\fq G$ is second countable.
\end{proposition}
\begin{proof}
  If $G$ is virtually cyclic then $\cbdry^\fq G$ is the discrete
  space with two points, and we are done. 
Otherwise, fix any $\eta\in\cbdry G$.
For each $g\in G$ and $n\in\mathbb{N}$ choose an open set $U_{g,n}$
such that $U(g\eta,\psi(\rho_{g\eta},n))\subset U_{g,n}\subset
U(g\eta,n)$ as in \fullref{cor:annulus}.

Let $U$ be a non-empty open set and let $\zeta$ be a point in $U$.
By definition of $\fq$, there exists an $r\geq 1$ such that
$U(\zeta,r)\subset U$.
Let $R'$ be the constant of \fullref{prop:countablebasis} for $\zeta$
and $r$, and let
$R'\leq R_1\in\mathbb{N}$.
As noted there, there exists a sublinear $\rho'$ such that the points $g_t\eta$ in the proof of
\fullref{prop:countablebasis} are all $\rho'$--contracting.
Define $R_2:=\psi(\rho',R_1)\geq \psi(\rho_{g_t\eta},R_1)$.
Combining \fullref{prop:countablebasis} and the definition of the sets
$U_{g,n}$, there exists $g\in G$ such that:
\[\zeta\in U(g\eta,R_2)\subset U_{g,R_1}\subset U(g\eta,R_1)\subset
U(\zeta,r)\subset U\]

Thus, $\mathcal{U}:=\{U_{g,n}\mid g\in G,\,n\in\mathbb{N}\}$ is a
countable basis for $\cbdry^\fq G$.
\end{proof}

\begin{corollary}\label{metrizable}
   $\cbdry^\fq G$ is metrizable.
\end{corollary}
\begin{proof}
$\cbdry^\fq G$ is second countable by \fullref{secondcountable},
regular by \fullref{regular}, and Hausdorff by \fullref{hausdorff}.
  The Urysohn Metrization Theorem says every second countable, regular,
  Hausdorff space is metrizable.
\end{proof}

It is an interesting open problem to describe, in terms of the
geometry of $G$, a metric on
$\cbdry^\fq G$ that is compatible with $\fq$.

\section{Dynamics}\label{sec:dynamics}
\begin{definition}
  An element $g\in G$ is \emph{contracting} if $\mathbb{Z}\to G :
  n\mapsto g^n$ is a quasi-isometric embedding whose image is a
  contracting set.
\end{definition}
We use $g^\infty$ and $g^{-\infty}$ to denote the equivalence classes
of the contracting quasi-geodesic rays based at $\one$ corresponding
to the non-negative powers of $g$ and non-positive powers, respectively.
These are distinct points in $\cbdry G$.

\begin{lemma}\label{lem:nicenesting}
  Given a contracting
  element $g\in G$, an $r\geq 1$, and a point $\zeta\in\cbdry^\fq
  G\setminus\{g^\infty,\,g^{-\infty}\}$ there exists an $R'\geq 1$ such that for every $R\geq
  R'$ and every $n\in\mathbb{N}$ we have $U(\zeta,R)\subset g^{-n}U(g^n\zeta,r)$.
\end{lemma}
\begin{proof}
Since $g$ is contracting there is a sublinear function $\rho$ such
that all geodesic segments joining powers of $g$ as well as geodesic rays
based at $\one$ going to $g^\infty$, $g^{-\infty}$, or $g^m\zeta$ for any
$m\in\mathbb{Z}$ are all $\rho$--contracting.

Consider a geodesic triangle with sides
$g^{-m}\alpha^{g^m\zeta}$,
$\alpha^\zeta$, and a geodesic from $\one$ to $g^{-m}$ for arbitrary
$m\in\mathbb{Z}$.
All sides are $\rho$--contracting, and such a triangle is $B$--thin for some $B$ independent of $m$. 
Thus, for sufficiently large $s'$, independent of $m$, the point
$\alpha^\zeta_{s'}$ is more than $B$--far from $\langle g\rangle$, hence $B$--close to
$g^{-m}\alpha^{g^m\zeta}$.
Since $\alpha^\zeta$ and $g^{-m}\alpha^{g^m\zeta}$ are asymptotic,
they eventually come $\kappa(\rho,1,0)$--close and then stay
$\kappa'(\rho,1,0)$--close thereafter.
\fullref{GIT} says the first time they come $\kappa(\rho,1,0)$ close occurs
no later than $s'+\rho'(B)$.
Take $R'\geq s'+\rho'(B)$, which guarantees
$d(\alpha^\zeta_s,g^{-m}\alpha^{g^m\zeta})\leq \kappa'(\rho,1,0)$ for all
$m\in\mathbb{Z}$ and all $s\geq R'$.

Let $\cL'$ and $\cA'$ be the constants of
\fullref{lem:nonasymptoticmakesquasigeodesic} for $\rho$, $\cL=\sqrt{r/3}$
and $\cA=r/3$, and let $\theconstantformerlyknownasLambda$ and $\lambda$ be as in \fullref{keylemma}.
Take $T:=1+r+\theconstantformerlyknownasLambda(\kappa(\rho,\cL',\cA')+\kappa'(\rho,1,0))+\lambda(\rho,1,0)$.
We require further that $R'$ is larger than $\cL'$ and $\cA'$ and large enough so that for all $s\geq R'$
we have
$d(\alpha^\zeta_s,\langle g\rangle)>T+\kappa'(\rho,1,0)$.

Suppose, for a contradiction, that there exist some $n\in\mathbb{N}$ and
$R\geq R'$ such that there exists a point $\eta\in
U(\zeta,R)\setminus g^{-n}U(g^n\zeta,r)$.
Since $\eta\notin g^{-n}U(g^n\zeta,r)$,  for some $\cL,\,\cA$ there exists a continuous $(\cL,\cA)$--quasi-geodesic
$\gamma\in g^n\eta$ that does not come $\kappa(\rho,\cL,\cA)$--close
to $\alpha^{g^n\zeta}$ outside $N_r\one$.
By \fullref{obs}, it suffices to consider the case that $\cL<\sqrt{r/3}$
and $\cA<r/3$.

As in \fullref{lem:nonasymptoticmakesquasigeodesic}, construct a continuous $(\cL',\cA')$--quasi-geodesic
$\delta$ that first follows a geodesic from $\one$ towards $g^{-n}$, then a
geodesic segment of length $\kappa(\rho,\cL,\cA)$, and then
follows a tail of $g^{-n}\gamma$.
Since it shares a tail with $g^{-n}\gamma$, we have $\delta\in \eta$.
Since $\eta\in U(\zeta,R)$, there is some $s\geq R$ such that 
$\delta$ comes within distance $\kappa(\rho,\cL',\cA')$ of 
$\alpha^\zeta_s$.
Our choice of $R'$ guarantees that
$d(\alpha^\zeta_s,g^{-n}\alpha^{g^n\zeta})\leq \kappa'(\rho,1,0)$ and
$d(\alpha^\zeta_s,\langle g\rangle)>T+\kappa'(\rho,1,0)$.
The latter implies that the point of $\delta$ close to $\alpha^\zeta_s$
is a point of $g^{-n}\gamma$, so there is a point of $g^{-n}\gamma$
that comes within distance  $\kappa(\rho,\cL',\cA')+ \kappa'(\rho,1,0)$ of a
point $x$ of $g^{-n}\alpha^{g^n\zeta}$ such that $d(x,\langle
g\rangle)\geq T$.
\fullref{keylemma} says that there is a point $y$ on
$g^{-n}\alpha^{g^n\zeta}$ such that $d(y,g^{-n}\gamma)=
\kappa(\rho,\cL,\cA)$ and $d(x,y)\leq \theconstantformerlyknownasLambda(\kappa(\rho,\cL',\cA')+
\kappa'(\rho,1,0))+\lambda(\rho,\cL,\cA)$.
The definition of $T$ implies $d(y,g^{-n})\geq d(y,\langle g\rangle)\geq d(x,\langle g\rangle)-d(x,y)>r$. 
But then $g^ny$ is a point of $\alpha^{g^n\zeta}$ with $d(g^ny,\one)>r$
and $d(g^ny,\gamma)=\kappa(\rho,\cL,\cA)$, contradicting the
definition of $\gamma$.
\end{proof}

\begin{lemma}\label{lem:pushalongcontracting}
 Given a contracting element $g\in G$, an $r\geq 1$, and a point
 $\zeta\in\cbdry^\fq G\setminus\{g^{-\infty}\}$ there exist constants
 $R'\geq 1$ and $N$ such that for all $R\geq R'$ and $n\geq N$ we have
 $g^nU(\zeta,R)\subset U(g^\infty,r)$.
\end{lemma}
\begin{proof}
The lemma is easy if $\zeta=g^\infty$, so assume not.
Since $g$ is contracting the geodesics $\alpha^{g^n\zeta}$ are
uniformly contracting. 
Let $\rho$ be a sublinear function such that $\alpha^{g^\infty}$,
$\alpha^{g^{-\infty}}$, and all $\alpha^{g^n\zeta}$ are
$\rho$--contracting. 
Since these geodesics are uniformly contracting, ideal geodesic
triangles with vertices $g^\infty$, $g^{-\infty}$, and $g^n\zeta$ are
uniformly thin, independent of $n$. 
Thus, for $N$ sufficiently large and for all $n\geq N$ we have that
$\alpha^{g^n\zeta}$ stays $\kappa'(\rho,1,0)$ close to $\alpha^{g^\infty}$
for distance greater than
$S':=1+r+\lambda(\rho,\sqrt{r/3},r/3)+\theconstantformerlyknownasLambda(\kappa'(\rho,\sqrt{r/3},r/3)+\kappa'(\rho,1,0))$,
where $\theconstantformerlyknownasLambda$ and $\lambda$ are as in \fullref{keylemma}.

Suppose that $\eta\in U(g^n\zeta,S)$ for some $n\geq N$ and $S\geq
S'$. 
Let $\gamma\in\eta$ be a continuous
$(\cL,\cA)$--quasi-geodesic for some $\cL^2,\,\cA\leq r/3$.
 By hypothesis, $\gamma$ comes $\kappa(\rho,\cL,\cA)$--close to
 $\alpha^{g^n\zeta}$ outside $N_S\one$. 
Therefore, $\gamma$ stays $\kappa'(\rho,\cL,\cA)$--close to
$\alpha^{g^n\zeta}$ in $N_S\one$. 
By our choice of $N$, this implies $\gamma$ stays
$(\kappa'(\rho,\cL,\cA)+\kappa'(\rho,1,0))$--close to
$\alpha^{g^\infty}$ in $N_S\one$. 
By our choice of $S$ and \fullref{keylemma}, $\gamma$ comes
$\kappa(\rho,\cL,\cA)$--close to $\alpha^{g^\infty}$ outside the
neighborhood of $\one$ of radius:
\begin{multline*}
  S-\theconstantformerlyknownasLambda 
(\kappa'(\rho,\cL,\cA)+\kappa'(\rho,1,0))-\lambda(\rho,\cL,\cA)\\
\geq S'-\theconstantformerlyknownasLambda
(\kappa'(\rho,\sqrt{r/3},r/3)+\kappa'(\rho,1,0))-\lambda(\rho,\sqrt{r/3},r/3)>r
\end{multline*}
Since $\gamma$ was arbitrary, $\eta\in U(g^\infty,r)$, thus
$U(g^n\zeta,S)\subset U(g^\infty,r)$.

By \fullref{lem:nicenesting}, given $g$, $S'$, and $\zeta$ there exists
an $R'$ such that for all $R\geq R'$ and every $n\in\mathbb{N}$ we
have $U(\zeta,R)\subset g^{-n}U(g^n\zeta,S')$.
Thus, for this $R'$ and $N$ as above we have, for all $n\geq N$ and $R\geq R'$, that
$g^nU(\zeta,R)\subset U(g^n\zeta,S')\subset U(g^\infty,r)$.
\end{proof}

\begin{theorem}[Weak North-South dynamics for contracting elements]\label{NS}
Let $g\in G$ be a contracting element.
For every open set $V$ containing $g^\infty$ and every compact set
$C\subset\cbdry^\fq G\setminus\{g^{-\infty}\}$ there exists an $N$
such that for all $n\geq N$ we have $g^nC\subset V$.
\end{theorem}
We remark that if \fullref{NS} were true for \emph{closed} sets 
and $G$ contained contracting elements without common powers then
we could play
  ping-pong to produce a free subgroup of $G$. 
Such a result cannot be true in this generality
  because there are Tarski Monsters, non-cyclic groups such that every
  proper subgroup is cyclic, such that every non-trivial
  element is Morse, hence, contracting \cite[Theorem 1.12]{OlsOsiSap09}.
\begin{proof}[Proof of \fullref{NS}]
Since $V$ is an open set containing $g^\infty$ there exists some $r>0$
such that $U(g^\infty,r)\subset V$.
  For this $r$ and for each $\zeta\in C$ there exist $R_\zeta$ and $N_\zeta$ 
  as in \fullref{lem:pushalongcontracting} such that for all $n\geq
  N_\zeta$ we have $g^nU(\zeta,R_\zeta)\subset U(g^\infty,r)$.
By \fullref{prop:topology}, $U(\zeta,R_\zeta)$ is a neighborhood of
$\zeta$, so there exists an open set $U'_\zeta$ such that $\zeta\in
U'_\zeta\subset U(\zeta,R_\zeta)$.
The collection $\{U'_\zeta\mid \zeta\in C\}$ 
is an open cover of $C$, which is compact, so there exists a
finite subset $C'$ of $C$ such that
$\{U'_\zeta\mid \zeta\in C'\}$ covers $C$.
Define $N:=\max_{\zeta\in C'}N_\zeta$.
For every $n\geq N$ we then have:
\begin{align*}
  g^nC&\subset g^n(\bigcup_{\zeta\in C'}U'_\zeta)\subset
        g^n(\bigcup_{\zeta\in C'}U(\zeta,R_\zeta))\\
&=\bigcup_{\zeta\in C'}g^nU(\zeta,R_\zeta)\subset U(g^\infty,r)\subset V\qedhere
\end{align*}

\end{proof}

\section{Compactness}\label{sec:compact}
In this section we characterize when the contracting boundary of a group is compact, \fullref{thm:compacthyperbolic}, and give a partial characterization of when the limit set of a group is compact, \fullref{prop:compactclosure}.
\begin{theorem}\label{thm:compacthyperbolic}
  Let $G$ be an infinite,  finitely generated group. Consider the Cayley graph of $G$, which we again denote $G$,  with respect to a fixed finite generating set. The following are equivalent:
  \begin{enumerate}
\item Geodesic rays in $G$ are uniformly contracting.\label{item:raysuniformlycontracting}
\item Geodesic segments in $G$ are uniformly contracting.\label{item:uniformlycontracting}
\item $G$ is hyperbolic.\label{item:hyperbolic}
\item $\cbdry^\dl G$ is non-empty and compact.\label{item:dlcompact}
\item $\cbdry^\fq G$ is non-empty and compact.\label{item:fqcompact}
  \item Every geodesic ray in $G$ is contracting.\label{item:allcontracting}
  \end{enumerate}
\end{theorem}
\begin{remark}
Work of Cordes and Durham \cite{CorDur16} implies `\eqref{item:dlcompact} implies \eqref{item:hyperbolic}'.
Roughly the same argument we use for `\eqref{item:raysuniformlycontracting} implies \eqref{item:uniformlycontracting}' is contained in the proof of \cite[Proposition~4.2]{CorDur16}.
More interestingly, they prove \cite[Lemma~4.1]{CorDur16}  that compact subsets of the Morse boundary (of a space) are uniformly Morse, which is a more general version of `\eqref{item:dlcompact} implies \eqref{item:raysuniformlycontracting}'.
We specifically designed the \thistopology to allow sequences with decaying contraction/Morse functions to converge when geometrically appropriate, so the corresponding statement cannot be true in our setting. 
In particular, the equivalence of \eqref{item:fqcompact} and \eqref{item:allcontracting} with \eqref{item:raysuniformlycontracting}-\eqref{item:dlcompact} does not follow from their result. 
\end{remark}
\begin{remark}
Previous attempts have been made to prove results similar to `\eqref{item:allcontracting} implies \eqref{item:raysuniformlycontracting}'. We point out a difficulty in the obvious approach.
Suppose that $(\gamma^n)_{n\in\mathbb{N}}$ is a sequence of geodesics with decaying Morse functions.
Let $\delta^n$ be paths witnessing the decaying Morse functions, by which we mean that there exist $\cL$ and $\cA$ such that for each $n$ the path $\delta^n$ is an $(\cL,\cA)$--quasi-geodesic with endpoints on $\gamma^n$, and that $\delta^n$ is not contained in ${N}_n\gamma^n$. 
Let $\beta^n$ denote the subsegment of $\gamma^n$ between the endpoints of $\delta^n$.
We may assume by translation that for all $n$ the basepoint $\one$ is approximately the midpoint of $\beta^n$.
By properness, there is a subsequence of $(\gamma^n)$ that converges to a geodesic $\gamma$ through $\one$. 
One would guess that $\gamma$ is not Morse, but this is not true in general; explicit counterexamples can be constructed.
The problem is that the convergence to $\gamma$ can be much slower than the growth of $|\beta^n|$ so that the subsegment of $\gamma^n$ that agrees with $\gamma$ can be a vanishingly small fraction of $\beta^n$. In this case the segments $\delta^n$ may not have endpoints close to $\gamma$, so no conclusion can be drawn.

It seems difficult to fix this argument. 
Instead, our strategy, roughly, will be to  construct for each $i$ a translate $g_i\delta^i$ of $\delta^i$ and for each $n$ a geodesic ray that passes suitably close to both endpoints of $g_i\delta^i$ for all $i\leq n$. 
We argue that a subsequence of these rays converges to a non-Morse geodesic ray.
\end{remark}
\begin{proof}
Assume \eqref{item:raysuniformlycontracting}.
Since $G$ is infinite and finitely generated, there exists a geodesic ray $\alpha$ based at $\one$.
Recall that for $n\in\mathbb{N}$ the point $\alpha_n$ is a vertex of the Cayley graph, so it corresponds to a unique element of the group $G$. 
Thus,  $\alpha_n^{-1}\alpha$ is simply the translate of $\alpha$ by the isometry of $G$ defined by left multiplication by the element $\alpha_n^{-1}$.
Since $\one=\alpha_n^{-1}\alpha_n\in\alpha_n^{-1}\alpha$, the sequence $(\alpha_n^{-1}\alpha)_{n\in\mathbb{N}}$ has a subsequence that converges to a bi-infinite geodesic $\beta$ containing $\one$. 
By construction, subsegments of $\beta$ are close to subsegments of translates of $\alpha$, so by \fullref{lem:subsegment} and \fullref{lem:combination}, $\beta$ is contracting, with contraction function determined by the uniform bound for rays. Let $\beta^+$ and $\beta^-$ denote the two rays based at $\one$ such that $\beta=\beta^+\cup\beta^-$.

Let $g$ be an arbitrary non-trivial element of $G$.
Consider the ideal geodesic triangle with one side $g\beta$ and whose other two sides are geodesic rays based at $\one$ with endpoints $g\beta^+_\infty$ and $g\beta^-_\infty$.
The sides of this triangle are uniformly contracting, so it is uniformly thin, so there is some constant $C$ such that every point on $g\beta$ is $C$--close to one of the other two sides. 
In particular, $g$ is $C$--close to one of the other sides. 
Since the constant $C$ is independent of $g$, we have that that for every $g\in G$ there exists a geodesic ray $\gamma^g$ based at $\one$ and passing within distance $C$ of $g$.

Let $\delta$ be a geodesic segment with endpoints $h$ and $hg$ for some $g,\,h\in G$.
The contraction function of the geodesic $h^{-1}\delta$ from $\one$ to $g$ can be bounded in terms of $C$ and the contraction function of $\gamma^g$, but since rays have uniform contraction this gives us a bound for the contraction function for $h^{-1}\delta$, hence for $\delta$.
Since every geodesic segment is at Hausdorff distance at most 1/2 from a geodesic segment with endpoints at vertices, \fullref{Hausdorff} tells us that geodesic segments are uniformly contracting.
Thus, \eqref{item:raysuniformlycontracting} implies \eqref{item:uniformlycontracting}.

\medskip

If geodesic segments in $G$ are uniformly contracting then geodesic bigons are uniformly thin, so $G$ is hyperbolic by a theorem of Papasoglu \cite{MR1346209}. Thus, \eqref{item:uniformlycontracting} implies \eqref{item:hyperbolic}.

\medskip

By \cite[Theorem~3.10]{Cor15}, if $G$ is hyperbolic then $\cbdry^\dl G$ agrees with the Gromov boundary, which is compact, so \eqref{item:hyperbolic} implies \eqref{item:dlcompact}.

\medskip

$\dl$ is a refinement of $\fq$ by \fullref{dlrefinesfq}, so \eqref{item:dlcompact} implies \eqref{item:fqcompact}.

\medskip

If $G$ is virtually cyclic then \eqref{item:allcontracting} is true. 
If $G$ is not virtually cyclic then, by \fullref{finiteorbit}, if $\cbdry^\fq G$ is non-empty then it is infinite.
In particular, there are distinct points in $\cbdry^\fq G$. Choose two of them and connect them by a geodesic $\beta$, which is necessarily contracting. 
By translating $\beta$ we may assume that $\beta_0=\one$.
Suppose that $\alpha$ is an arbitrary geodesic ray based at $\one$.
As in the proof of \fullref{prop:countablebasis}, after possibly exchanging $\beta$ with $\bar{\beta}$ there is increasing $\sigma'\from\mathbb{N}\to\mathbb{N}$ such that for all $n\in\mathbb{N}$ we have that $\alpha_{\sigma'(n)}\beta_{[0,\infty)}$ does not backtrack far along $\alpha_{[0,\sigma'(n)]}$.
This means there are $\cL$ and $\cA$, independent of $n$, such that the concatenation of  $\alpha_{[0,\sigma'(n)]}$ and $\alpha_{\sigma'(n)}\beta_{[0,\infty)}$ is a continuous $(\cL,\cA)$--quasi-geodesic.

If $\cbdry^\fq G$ is compact then the sequence $(\alpha_{\sigma'(n)}\beta_\infty)_{n\in\mathbb{N}}$ has a convergent subsequence, so there is an increasing $\sigma''\from \mathbb{N}\to\mathbb{N}$ such that for $\sigma:=\sigma'\circ\sigma''$ we have  $(\alpha_{\sigma(n)}\beta_\infty)_{n\in\mathbb{N}}$ converges to a point $\zeta\in\cbdry^\fq G$. 
Then for every $r>1$ there exists an $N$ such that for all $n\geq N$ we have $\alpha_{\sigma(n)}\beta_\infty\in U(\zeta,r)$.
For $r>3\cL^2,\,3\cA$ we then have that the continuous $(\cL,\cA)$--quasi-geodesic $\alpha_{[0,\sigma(n)]}+\alpha_{\sigma(n)}\beta_{[0,\infty)}\in \alpha_{\sigma(n)}\beta_\infty$ comes $\kappa(\rho_\zeta,\cL,\cA)$--close to $\alpha^\zeta$ outside the ball of radius $r$ about $\one$, for all sufficiently large $n$, and therefore has initial segment of length at least $r$ contained in the $\kappa'(\rho_\zeta,\cL,\cA)$--neighborhood of $\alpha^\zeta$. 
Since this is true for all sufficiently large $n$ and since longer and longer initial segments of the $\alpha_{[0,\sigma(n)]}+\alpha_{\sigma(n)}\beta_{[0,\infty)}$ are initial segments of $\alpha$, we conclude that $\alpha$ is asymptotic to $\alpha^\zeta$, which implies that $\alpha$ is contracting.
Thus, \eqref{item:fqcompact} implies \eqref{item:allcontracting}.

\medskip

Finally, we prove \eqref{item:allcontracting} implies \eqref{item:raysuniformlycontracting}. 
We do so by assuming \eqref{item:allcontracting} is true and \eqref{item:raysuniformlycontracting} is false, and deriving a contradiction.
The strategy is as follows.
The fact that $\cbdry^\fq G$ is not uniformly contracting implies that no non-empty open subset of $\cbdry^\fq G$ is uniformly contracting. 
We construct a nested decreasing sequence of neighborhoods focused on points with successively worse contraction behavior.
We use properness of $G$ to conclude that a subsequence of representative geodesics of these focal points converges to a geodesic ray. 
The assumption  \eqref{item:allcontracting} implies the limiting ray is contracting, so it represents a point in $\cbdry^\fq G$, and we claim that this point is in the intersection of the nested sequence of neighborhoods. 
Furthermore, the details of the construction ensure that the limiting ray actually experiences the successively worse contraction behavior of the construction, with the conclusion that it is not a contracting ray, contradicting \eqref{item:allcontracting}.

\smallskip

If rays in $G$ fail to be uniformly contracting then so do bi-infinite geodesics in $G$.
To see this, fix a ray $\alpha$.
Any other ray has a translate $\beta$ with the same basepoint as $\alpha$. 
Since rays are contracting there is a contracting bi-infinite geodesic $\gamma$ with endpoints $\alpha_\infty$ and $\beta_\infty$. 
Now $\alpha$, $\beta$, and $\gamma$ make a geodesic triangle, so the contraction function of $\beta$ can be bounded in terms of those of $\alpha$ and $\gamma$.
If bi-infinite geodesics are all $\rho$--contracting then this would mean the contraction function for $\beta$ can be bounded in terms of only $\rho$ and $\rho_\alpha$, so rays would be uniformly contracting.
Combining this with \fullref{morseequivalent}, we have that $\neg$\eqref{item:raysuniformlycontracting} implies bi-infinite geodesics in $G$ are not uniformly Morse.
They are all Morse, as each bi-infinite geodesic can be written as a union of two rays, which are contracting, by \eqref{item:allcontracting}. 
For each $\cL\geq 1$,\,$\cA\geq 0$, and bi-infinite geodesic $\gamma$ in $G$, define $D(\gamma,\cL,\cA)$ to be the supremum of the set:
\[\{d(z,\gamma)\mid z \text{ is a point on a continuous $(\cL,\cA)$--quasi-geodesic with endpoints on $\gamma$}\}\]
Since $\gamma$ is Morse, $D(\gamma,\cL,\cA)$ exists for each $\cL$ and $\cA$.
If $\sup_\gamma D(\gamma,\cL,\cA)$ exists for every $\cL$ and $\cA$ then we can define $\mu(\cL,\cA):=\sup_\gamma D(\gamma,\cL,\cA)$ and we have that all bi-infinite geodesics are $\mu$--Morse, contrary to hypothesis, so there exist some $\cL\geq 1 $ and $\cA\geq 0$ such that for all $n\in\mathbb{N}$ there exists a bi-infinite geodesic $\gamma^n$ and a continuous $(\cL,\cA)$--quasi-geodesic $\delta^n$ with endpoints on $\gamma^n$ such that $\delta^n$ is not contained in $N_n\gamma^n$. 
By translating and shifting the parameterization of $\gamma^n$ we may assume that $\gamma^n_0=\one$ and that the distances from $\one$ to the two endpoints of $\delta^n$ differ by at most 1.

Now we make a claim and use it to finish the proof:
\begin{equation}\label{eq:12}\text{\parbox{.9\textwidth}{
Let $\gamma$ be a bi-infinite $\rho_\gamma$--contracting geodesic. Given $\zeta\in\cbdry^\fq G$, $R>1$, $r\geq 0$ there exists $\eta\in U(\zeta,R)$ and $g\in G$ such that $\alpha^\eta$ passes within distance $\kappa(\rho_\gamma,1,0)$ of both endpoints of a segment of $g\gamma$ containing $g\gamma_{[-r,r]}$.}
}\end{equation}

See \fullref{easycase} and \fullref{hardcase} for illustrations of \eqref{eq:12}.
Assuming \eqref{eq:12}, we construct a decreasing nested sequence of neighborhoods in $\cbdry^\fq G$ focusing on points with successively worse contraction behavior.
The key trick is to build extra padding into our constants to give the contraction function of the eventual limiting geodesic time to dominate.
Let $\theconstantformerlyknownasLambda$ and $\lambda$ be as in \fullref{keylemma}; recall that $\lambda(\phi,1,0)=8\kappa(\phi,1,0)$.
Let $\psi$ be as in \fullref{cor:annulus}.
Let $\zeta^{0}\in\cbdry^\fq G$ and $R_{0}>1$ be arbitrary. 
Now, supposing $\zeta^{i}$ and $R_{i}$ have been defined, 
consider $\gamma^{i+1}$. 
Let $r_{i+1}$ be the least integer greater than half the distance between endpoints of $\delta^i$ plus the quantity $(\theconstantformerlyknownasLambda+1)\kappa(\rho_{\gamma^i},1,0)+(6\theconstantformerlyknownasLambda+15)(i+1)$.
Apply \eqref{eq:12} to $\gamma^{i+1},\,\zeta^i,\,\psi(\zeta^i,R_i),\,r_{i+1}$ and get output $\zeta^{i+1}\in U(\zeta^i,\psi(\zeta^i,R_i))\subset U(\zeta^i,R_i)$ and $g_{i+1}\in G$.
Let $R''_{i+1}$ be $\kappa(\rho_{\gamma^i},1,0)$ plus the larger of the distances to $\one$ of the endpoints of the subsegment of $g_{i+1}\gamma^{i+1}$ given by \eqref{eq:12}. 
Define $R'_{i+1}:=R''_{i+1}+(\theconstantformerlyknownasLambda+1)\kappa(\rho_{\zeta^{i+1}},1,0)+9(i+1)$.
By \fullref{cor:annulus}, there is an open set $U_i$ such that $U(\zeta^i,\psi(\zeta^i,R_i))\subset U_i\subset U(\zeta^i,R_i)$, so since $\zeta^{i+1}\in U(\zeta^i,\psi(\zeta^i,R_i))$ we can choose $R_{i+1}\geq R'_{i+1}$ large enough to guarantee $U(\zeta^{i+1},R_{i+1})\subset U\subset U(\zeta^i,R_i)$.

Consider the sequence of geodesic rays $(\alpha^{\zeta^n})_{n\in\mathbb{N}}$.
Some subsequence converges to a geodesic ray $\alpha$. 
By hypothesis, all rays are contracting, so there exists some $\rho_\alpha$ such that $\alpha$ is $\rho_\alpha$--contracting. 

Pick any $i\geq \kappa'(\rho_\alpha,1,0)>\kappa(\rho_\alpha,1,0)$.
There is some $n\gg i$ such that $\alpha$ agrees with $\alpha^{\zeta^n}$ for distance $R_i+\kappa(\rho_{\zeta^i},1,0)$.
Since the neighborhoods are nested, by construction, $\zeta^n\in U(\zeta^i,R_i)$, which implies that $\alpha$ comes $\kappa(\rho_{\zeta^i},1,0)$--close to $\alpha^{\zeta^i}$ outside the ball of radius $R_i\geq R'_i$ about $\one$. 
The definition of $R'_i$ and the fact that $i>\kappa(\rho_\alpha,1,0)$ gives us, by \fullref{keylemma}, that $\alpha^{\zeta^i}$ comes $\kappa(\rho_\alpha,1,0)$--close to $\alpha$ outside the ball of radius $R''_i$ about $\one$. 
In particular, by \fullref{lem:hausdorff}, $\alpha$ passes $(\kappa'(\rho_\alpha,1,0)+\kappa(\rho_{\zeta^i},1,0))$-close to both endpoints of a subsegment of $g_i\gamma^i$ containing $g_i\gamma^i_{[-r_i,r_i]}$.
The definition of $r_i$ and the fact that $i\geq \kappa(\rho_\alpha,1,0)$ give us, by a second application of \fullref{keylemma} and \fullref{lem:hausdorff}, that $\alpha$ comes within distance $\kappa'(\rho_\alpha,1,0)$ of both endpoints of $g_i\delta^i$.
Connect the endpoints of $g_i\delta^i$ to $\alpha$ by shortest geodesic segments.
For $\cA':=\cA+2 \kappa'(\rho_\alpha,1,0)$, the resulting path $\delta_i'$ is a continuous $(\cL,\cA')$--quasi-geodesic that is contained in $\bar{N}_{\kappa'(\rho_\alpha,\cL,\cA')}(\alpha)$ but leaves the $(i-\kappa'(\rho_\alpha,1,0))$--neighborhood of the subsegment of $\alpha$ between its endpoints.
For sufficiently large $i$ this contradicts the fact that $\delta_i'$ is $(\cL,\cA')$--quasi-geodesic.

\medskip

We now prove \eqref{eq:12}.
The idea is to take an element $g$ that pushes $\gamma$ far out along $\alpha^\zeta$ and take $\eta$ to be one of the endpoints of $g\gamma$. 
Then $\alpha^\eta$ forms a geodesic triangle with a subsegment of $\alpha^\zeta$ and a subray of $g\gamma$.
Additionally, we arrange for $g\gamma_{[-r,r]}$ to be suitably far from the quasi-center of this triangle so that it is in one of the thin legs of the triangle, parallel to a segment of $\alpha^\eta$.

Let $\alpha:=\alpha^\zeta$.
Let $r'>r$ represent a number to be determined, and choose any $s>r'+2\kappa'(\rho_\gamma,1,0)$.
First, suppose that for arbitrarily large $t$ there exists $g\in N_{\kappa'(\rho_{\gamma},1,0)}\alpha_t$ such that $d(g\gamma_{s}^{-1}\gamma_{r'},\alpha_{[0,t]})\geq \kappa'(\rho_\gamma,1,0)$.
We claim that for any sufficiently large $t$ we can take such a $g$ and $\eta:=g\gamma_{s}^{-1}\gamma_{-\infty}$ as the output of \eqref{eq:12}. 
To see this, define a continuous quasi-geodesic by following $\alpha$ until we reach the first of either $\alpha_t$ or a point of $\pi_\alpha(g\gamma_{s}^{-1}\gamma_{r'})$, then follow a geodesic to $g\gamma_{s}^{-1}\gamma_{r'}$, then follow $g\gamma_{s}^{-1}\bar{\gamma}$ towards $\eta$. 
By an argument similar to \fullref{lem:nonasymptoticmakesquasigeodesic}, $\beta$ is an $(\cL,\cA)$--quasi-geodesic, for some $\cL$ and $\cA$  not depending on $g$, $r$, or $s$.
Now, $\alpha_{[0,t]}$, $g\gamma_{s}^{-1}\gamma_{(-\infty,s]}$, and $\alpha^{\eta}$ form a $\kappa'(\rho_\gamma,1,0)$--almost geodesic triangle, so the contraction function of $\alpha^\eta$ is bounded in terms of the contraction functions of $\alpha$ and $\gamma$. 
Thus, there is some $E$ such that $\beta$ and $\alpha^\eta$ are bounded Hausdorff distance $E$ from one another, independent of our choices. 
By the hypothesis on $g$ and \fullref{cor:asymp} the two sides $\alpha_{[0,t]}$ and $g\gamma_{s}^{-1}\gamma_{(-\infty,s]}$ of the almost geodesic triangle are diverging at a linear rate, and so $g\gamma_{s}^{-1}\gamma_{r'}$ is $H$--close to some point of $\alpha^\eta$ for some $H$, again independent of our choices.
Assume that we chose $r'\geq r+(\theconstantformerlyknownasLambda+1)H+9\kappa(\rho_\gamma,1,0)$.
Then, by \fullref{keylemma}, we have that $\alpha^\eta$ passes within distance $\kappa(\rho_\gamma,1,0)$ of some point of $g\gamma_{s}^{-1}\gamma_{[r,r']}$, and also of some point in $g\gamma_{s}^{-1}\gamma_{[-r',-r]}$.

We also need to show $\eta\in U(\zeta,R)$.
Any continuous $(\cL',\cA')$--quasi-geodesic in $\eta$ stays bounded Hausdorff distance $H'$ from $\beta$, with bound depending on $\cL'$ and $\cA'$, but not $g$, $s$, or $t$. 
We only need to consider $\cL'<\sqrt{R/3}$ and $\cA'<R/3$, so we can bound $H'$ in terms of $R$ (and $\rho_\alpha$ and $\rho_\gamma$).
We therefore have that such a quasi-geodesic passes $(H'+H)$--close to  $g\gamma_{s}^{-1}\gamma_{r'}$, which is $(s-r'+\kappa'(\rho_\gamma,1,0))$--close to $\alpha_t$.
Applying \fullref{keylemma} we see that such a geodesic passes $\kappa(\rho_\alpha,\cL',\cA')$--close to $\alpha$ outside the ball of radius $R$ provided that $t$ is chosen sufficiently large with respect to $R$, $s$, and the contraction functions for $\alpha$ and $\gamma$. 
By hypothesis, we can choose $t$ as large as we like, so  in this case we are done.

\begin{figure}[h]
  \centering
\labellist
\tiny
\pinlabel $\gamma_s$ [r] at 2 69
\pinlabel $\gamma_{r'}$ [r] at 2 24
\pinlabel $\gamma_{-r'}$ [r] at 2 14
\pinlabel $\alpha_t$ [t] at 188 19
\pinlabel $g=g\gamma_s^{-1}\gamma_s$ [l] at 241 27
\pinlabel $g\gamma_s^{-1}\gamma_{r'}$ [l] at 194 73
\pinlabel $g\gamma_s^{-1}\gamma$ [br] at 229 48
\pinlabel $\alpha^{\eta}$ [br] at 147 38
\pinlabel $\beta$ [r] at 157 31
\pinlabel $\eta=g\gamma_s^{-1}\gamma_{-\infty}$ [b] at 156 100
\endlabellist
  \includegraphics{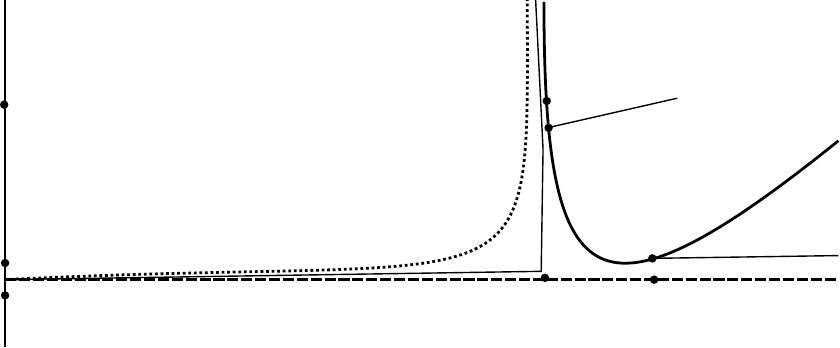}
  \caption{First case for \eqref{eq:12}.}
  \label{easycase}
\end{figure}

The other case is that there exists $T$ such that for every $g\in N_{\kappa'(\rho_{\gamma},1,0)}\alpha_{[T,\infty]}$ we have $d(g\gamma_{s}^{-1}\gamma_{r'},\alpha_{t'})<\kappa'(\rho_\gamma,1,0)$ for some $t'$ with $t>T$ and $t-t'=s-r'\pm 2\kappa'(\rho_\gamma,1,0)>0$.
Let $w$ be the word in the generators for $G$ read along the path $\gamma_{[r',s]}$.
Let $t_0>T$ be arbitrary, and let $g_0:=\alpha_{t_0}$.
Let $g_1:=g_0\gamma_{s}^{-1}\gamma_{r'}$.
By hypothesis there is a $t_1<t_0$ such that $d(g_1,\alpha_{t_1})\leq\kappa'(\rho_\gamma,1,0)$.
The segment $g_0\gamma_{s}^{-1}\gamma_{[r',s]}$ has edge label $w$ and, by \fullref{GIT}, is contained in $\bar{N}_{K}(\alpha)$ for some $K$ depending only on $\kappa'(\rho_\gamma,1,0)$ and $\rho_\gamma$.
If $t_1>T$ we can repeat, setting $g_2:=g_1\gamma_{s}^{-1}\gamma_{r'}$, so that the initial vertex $g_1$ of $g_1\gamma_{s}^{-1}\gamma_{[r',s]}$ agrees with the terminal vertex of $g_0\gamma_{s}^{-1}\gamma_{[r',s]}$.
Repeating this construction until $t_i\leq T$, we construct a path from $\alpha_t$ to $\bar{N}_{\kappa'(\rho_\gamma,1,0)}\alpha_{[0,T]}$ that is contained in the $K$--neighborhood of $\alpha$ and  whose edge label is a power of $w^{-1}$.
Since this is true for arbitrarily large $t$, we conclude that $w$ is a contracting element in $G$ and $\zeta=hw^{\infty}$ for some $h\in G$ that is $(s-r'+\kappa'(\rho_\gamma,1,0))$--close to $\alpha_T$.
Furthermore, we can also take $s'>s$ arbitrarily large and run the same argument to conclude that either we find the $g$ and $\eta$ we are looking for from the first case, or else arbitrarily long segments $\gamma_{[r',s']}$ can be sent into $\bar{N}_{K}\alpha_{[T,\infty)}$. 
We already know this tube contains an infinite path labelled by powers of $w$.
Therefore, there is $f$ which is $(s-r'+2\kappa'(\rho_\gamma,1,0))$--close to $\gamma_{s}$ such that $\gamma_{\infty}=fw^{\infty}$.

If $fw^{-\infty}=\gamma_{-\infty}$ then we can take $\eta:=\zeta$ and $g:=hw^af^{-1}$ for $a$ sufficiently large.
Otherwise, for any sufficiently large $t$ and $a$ we can take $g:=\alpha_t\gamma_s^{-1}fw^{-a}$ and $\eta:=g\gamma_{-\infty}$, see \fullref{hardcase}.
\begin{figure}[h]
  \centering
\labellist
\tiny
\pinlabel $\gamma_s$ [l] at 40 51
\pinlabel $\gamma_{r'}$ [r] at 40 24
\pinlabel $\gamma_{-r'}$ [r] at 40 14
\pinlabel $f$ [br] at 30 51
\pinlabel $fw^\infty=\gamma_\infty$ [b] at 39 99
\pinlabel $\alpha_t=\alpha_t\gamma_s^{-1}\gamma_s$ [t] at 216 18
\pinlabel $\alpha_t\gamma_s^{-1}\gamma$ [r] at 168 85
\pinlabel $\alpha_t\gamma_s^{-1}f$ [l] at 253 42
\pinlabel $\alpha_t\gamma_s^{-1}fw^{-a}$ [l] at 261 69
\pinlabel $\alpha_t\gamma_s^{-1}fw^{-a}\gamma_s$ [r] at 154 48
\pinlabel $\alpha_t\gamma_s^{-1}fw^{-a}\gamma_{r'}$ [r] at 154 68
\pinlabel $\alpha$ [t] at 97 19
\pinlabel $\alpha_t\gamma_s^{-1}fw^\infty$ [b] at 280 100
\pinlabel $\eta$ [b] at 215 100
\pinlabel $\alpha^{\eta}$ [br] at 212 48
\endlabellist
  \includegraphics{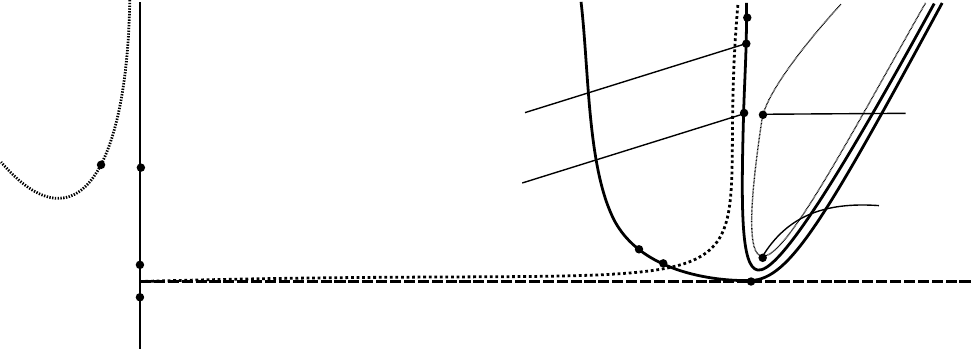}
  \caption{Second case for \eqref{eq:12}.}
  \label{hardcase}
\end{figure}
\end{proof}

\begin{proposition}\label{prop:compactclosure}
  Let $G$ be a group acting properly discontinuously on a proper geodesic metric space $X$.
Suppose that the orbit map $\phi\from g\mapsto g\bp$ takes contracting quasi-geodesics to contracting quasi-geodesics and has quasi-convex image. 
If $G\bp$ has compact closure in $\hat{X}$ then $G$ is hyperbolic.
If $G$ is infinite and hyperbolic and $\phi$ is Morse-controlled then $\limit(G)$ is non-empty and compact and $G\bp$ has compact closure in $\hat{X}$.
\end{proposition}
\begin{proof}
  If $G$ is finite then $G\bp$ is compact and $G$ is hyperbolic, so assume $G$ is infinite.
Since it has a quasi-convex orbit, $G$ is finitely generated and $\phi$ is a quasi-isometric embedding.
Since $G$ is infinite, $G\bp$ is unbounded, so if it has compact closure then $\limit(G)$ is non-empty and compact. 
By \fullref{prop:limset}, $\cbdry^\fq G$ is compact.
By \fullref{thm:compacthyperbolic}, $G$ is hyperbolic. 

Conversely, if $G$ is hyperbolic then, with respect to any finite generating set, geodesics in $G$ are uniformly Morse.
Therefore, there is a $\mu$ such that for any two points $g,\,h\in G$ there is a $\mu$--Morse geodesic $\gamma^{g,h}$ in $G$ from $g$ to $h$.
If $\phi$ is Morse-controlled, there exists a $\mu'$ depending only on $\mu$ such that $\phi(\gamma^{g,h})$ is $\mu'$--Morse. 
Any $(\cL,\cA)$--quasi-geodesic with endpoints $g\bp$ and $h\bp$ therefore stays $\mu'(\cL,\cA)$--close to $\phi(\gamma^{g,h})$, but $\phi(\gamma^{g,h})$ is a quasi-geodesic with integral points on $G\bp$, so it remains close to $G\bp$. 
Therefore, $G\bp$ is a Morse subset of $X$.
By  \fullref{prop:limset}, $\cbdry\phi\from \cbdry^\fq G\to \limit(G)$ is a homeomorphism.
Since $G$ is an infinite hyperbolic group, $\cbdry^\fq G$ is non-empty and compact, so $\limit(G)$ is as well.

It remains to show $\closure{G\bp}$ is compact.
Suppose $\mathcal{U}$ is an open cover of $\closure{G\bp}$.
Only finitely many  elements of $\mathcal{U}$ are required to cover $\limit(G)$. 
We claim that the part of $G\bp$ not covered by these finitely many sets is bounded, hence, finite, so only finitely many more elements of $\mathcal{U}$ are required to cover all of $\closure{G\bp}$. 
To see this, suppose $(g_n\bp)_{n\in\mathbb{N}}$ is an unbounded sequence in $G\bp$ that does not enter the chosen finite cover of $\limit(G)$. 
By passing to a subsequence, we may assume $d(\bp,g_n\bp)\geq n$, in which case $(g_n\bp)_{n\in\mathbb{N}}$ is a sequence with no convergent subsequence in $\hat{X}$.
For each $n$ pick a geodesic $\gamma^n$ from $\bp$ to $g_n\bp$. 
A subsequence $(\gamma^{\sigma(n)})_{n\in\mathbb{N}}$ converges to a geodesic ray $\gamma$ in $X$ based at $\bp$, but since the geodesics $\gamma^n$ were uniformly contracting, $\gamma$ is contracting.
Moreover, by uniform contraction the endpoints $g_{\sigma(n)}\bp$ converge to $\gamma_\infty$ in $\hat{X}$, which is a contradiction.
\end{proof}


\bibliographystyle{hypershort}
\bibliography{CCB} 

\end{document}